\documentclass[11pt]{article}

\usepackage{amsmath}
\usepackage{amssymb}
\usepackage{amsfonts}
\usepackage{amsthm}
\usepackage{array}
\usepackage{bm}
\usepackage[shortlabels]{enumitem}

\usepackage{amsxtra}
\usepackage{array}
\usepackage{mathrsfs}
\usepackage{slashed}
\usepackage{xcolor}
\usepackage{tikz-cd}
\usepackage{tkz-euclide} % no need to load TikZ
\usepackage{pgfplots}
\usepackage[toc,page]{appendix}
            
\newcolumntype{C}[1]{>{\centering\hspace{0pt}}p{#1}}
\usepackage[margin=1in]{geometry}

\newcommand{\SO}{\mathrm{SO}}
\newcommand{\Spin}{\mathrm{Spin}}
\newcommand{\U}{\mathrm{U}}
\newcommand{\SU}{\mathrm{SU}}
\newcommand{\Sp}{\mathrm{Sp}}
\newcommand{\G}{\mathrm{G}}

\newcommand{\Id}{\mathrm{Id}}
\newcommand{\Ker}{\mathrm{Ker}}

\newcommand{\End}{\mathrm{End}}

\newcommand{\Sym}{\mathrm{Sym}}

\newcommand{\Cone}{\mathrm{C}}

\newcommand{\R}{\mathbb{R}}
\newcommand{\C}{\mathbb{C}}
\newcommand{\HH}{\mathbb{H}}

\newcommand{\Sph}{\mathbb{S}}

\newcommand{\CP}{\mathbb{CP}}
\newcommand{\HP}{\mathbb{HP}}

\newcommand{\Gr}{\mathrm{Gr}}
\newcommand{\vol}{\mathrm{vol}}

\newcommand{\LB}{[\![}
\newcommand{\RB}{]\!]}

\newcommand{\del}{\partial}
\newcommand{\dil}{R}
\newcommand{\ddr}{\frac{\del}{\del r}}

\newtheorem{thm}{Theorem}[section]
\newtheorem{prop}[thm]{Proposition}
\newtheorem{lem}[thm]{Lemma}
\newtheorem{cor}[thm]{Corollary}

\theoremstyle{definition}
\newtheorem{defn}[thm]{Definition}
\newtheorem{example}[thm]{Example}
\newtheorem{rmk}[thm]{Remark}

\usepackage[nottoc]{tocbibind}
\numberwithin{equation}{section}

\usepackage{hyperref}
\title{Calibrated Geometry in \\ Hyperk\"{a}hler Cones, $3$-Sasakian Manifolds, \\
and Twistor Spaces}
\author{Benjamin Aslan, Spiro Karigiannis, Jesse Madnick}
\date{March 2024}

\newcommand{\Addresses}
{{  \bigskip
		\textsc{University College London} \par\nopagebreak
		\textsc{London, United Kingdom} \par\nopagebreak
		\texttt{ucahbas@ucl.ac.uk}
		
		\bigskip
		\textsc{University of Waterloo} \par\nopagebreak
		\textsc{Waterloo, ON, Canada}\par\nopagebreak
		\texttt{karigiannis@uwaterloo.ca}

		\bigskip
		\textsc{University of Oregon} \par\nopagebreak
		\textsc{Eugene, OR, United States}\par\nopagebreak
		\texttt{jmadnick@uoregon.edu}
}}

\begin{document}

\maketitle

\begin{abstract}
We systematically study calibrated geometry in hyperk\"ahler cones $C^{4n+4}$, their 3-Sasakian links $M^{4n+3}$, and the corresponding twistor spaces $Z^{4n+2}$, emphasizing the relationships between submanifold geometries in various spaces. Our analysis highlights the role played by a canonical $\Sp(n)\U(1)$-structure $\gamma$ on the twistor space $Z$. We observe that $\mathrm{Re}(e^{- i \theta} \gamma)$ is an $S^1$-family of semi-calibrations, and make a detailed study of their associated calibrated geometries. \\
\indent As an application, we obtain new characterizations of complex Lagrangian and complex isotropic cones in hyperk\"{a}hler cones, generalizing a result of Ejiri--Tsukada.  We also generalize a theorem of Storm on submanifolds of twistor spaces that are Lagrangian with respect to both the K\"{a}hler-Einstein and nearly-K\"{a}hler structures.
\end{abstract}

\tableofcontents

\section{Introduction}

\indent \indent Hyperk\"{a}hler manifolds $C$, equipped with a Riemannian metric $g_C$, complex structures $(I_1, I_2, I_3)$, and K\"ahler forms $(\omega_1, \omega_2, \omega_3)$, are a rich source of calibrated geometries.  They feature not only familiar geometries arising from the Calabi-Yau structure -- such as complex submanifolds and special Lagrangians -- but also less-familiar ones specific to the hyperk\"{a}hler setting.  For example, a submanifold $N^{2k+2} \subset C^{4n+4}$ is \emph{complex isotropic with respect to $I_1$} if it is simultaneously
$$I_1\text{-complex, \ \ }\omega_2\text{-isotropic, \ \ and \ \ } \omega_3\text{-isotropic.}$$

\emph{Complex Lagrangians} $N^{2n+2} \subset C^{4n+4}$, those complex isotropic submanifolds of top dimension $2n+2$, are particularly remarkable, as they are at once complex submanifolds with respect to $I_1$ and special Lagrangian with respect to $I_2$ and $I_3$. \\
\indent This paper seeks to systematically study the various calibrated cones of hyperk\"{a}hler manifolds $C$, with a particular focus on complex isotropic cones.  For this, it is of course necessary to assume that $(C^{4n+4}, g_C) = (\R^+ \times M^{4n+3}, dr^2 + r^2 g_M)$ is itself a Riemannian cone. \\

\indent Hyperk\"{a}hler cones $C^{4n+4}$ are themselves highly special objects: each induces three associated Einstein spaces, called $M$, $Z$, and $Q$, as we briefly recall.  The first of these, $M^{4n+3}$, is just the link of $C$, which inherits a $3$-Sasakian structure.  In view of the simple relationship between $C$ and $M$, $3$-Sasakian manifolds exhibit a wide array of semi-calibrated geometries.  Indeed, each of the calibrated cones in $C$ that we study has a semi-calibrated counterpart in $M$.
$$\begin{tabular}{| c | l || l | c |} \hline
$\dim(\mathrm{C}(L))$ & Calibrated cone $\mathrm{C}(L) \subset C$ & Semi-calibrated link $L \subset M$ & $\dim(L)$ \\ \hline \hline
$2k+2$ & Complex & CR & $2k+1$ \\ \hline
$2n+2$ & Special Lagrangian & Special Legendrian & $2n+1$ \\ \hline
$2k+2$ & Special isotropic & Special isotropic & $2k+1$ \\ \hline
$2n+2$ & Complex Lagrangian & CR Legendrian & $2n+1$ \\ \hline
$2k+2$ & Complex isotropic & CR isotropic & $2k+1$ \\ \hline
$4$ & Cayley & Associative & $3$ \\ \hline
\end{tabular}$$
The entries of this table will be explained in $\S$\ref{sec:CalGeoHK} and $\S$\ref{sec:CalGeo3Sas}. \\
 \indent Now, since $M$ is $3$-Sasakian, it admits three linearly independent Reeb vector fields $A_1, A_2, A_3$.  In fact, for each $v = (v_1, v_2, v_3) \in S^2$, the Reeb field $A_v = \sum v_i A_i$ yields a $1$-dimensional foliation $\mathcal{F}_v$ on $M$, the projection $p_v \colon M \to M/\mathcal{F}_v$ is a principal $S^1$-orbibundle, and the quotient $Z = M/\mathcal{F}_v$ is a $(4n+2)$-orbifold.  It is well-known that $Z$ naturally admits both a K\"{a}hler-Einstein structure $(g_{\mathrm{KE}}, J_{\mathrm{KE}}, \omega_{\mathrm{KE}})$ and a nearly-K\"{a}hler structure $(g_{\mathrm{NK}}, J_{\mathrm{NK}}, \omega_{\mathrm{NK}})$.  Indeed, $Z$ is the twistor space of a quaternionic-K\"{a}hler $4n$-orbifold $Q$ of positive scalar curvature. \\
 \indent The four Einstein spaces $C,M,Z,Q$ may be summarized in the following ``diamond diagram" in which $\tau \colon Z \to Q$ denotes the twistor $S^2$-bundle.
 \begin{equation} \label{diamond-diag}
 \begin{tikzcd}
   & M^{4n+3} \arrow[r, hook] \arrow[rd, "{p_v}"'] \arrow[dd, "h"'] & C^{4n+4} \arrow[d] &  \\
 &      &     Z^{4n+2} \arrow[ld, "\tau"]         &        \\
                              & Q^{4n}        &          & 
\end{tikzcd}
\end{equation}
The \emph{flat model} is $(C,M,Z,Q) = (\HH^{n+1}, \Sph^{4n+3}, \CP^{2n+1}, \HP^n)$, in which each $p_v \colon \Sph^{4n+3} \to \CP^{2n+1}$ is a complex Hopf fibration, and $h \colon \Sph^{4n+3} \to \HP^n$ is a quaternionic Hopf fibration. \\

\indent In addition to all of the structure already discussed, we recover an observation of Alexandrov~\cite{Alexandrov} that twistor spaces $Z$ admit a distinguished complex $3$-form $\gamma$ corresponding to an $\Sp(n)\U(1)$-structure. In fact we give two different proofs of this result, one in $\S$\ref{sub:GeomTwistor} via the $3$-Sasakian geometry of $M$, and the other in $\S$\ref{sec:QKManifolds} via the quaternionic-K\"{a}hler geometry of $Q$. Furthermore, we establish the new result that $\mathrm{Re}(\gamma)$ is a semi-calibration and we classify those $\mathrm{Re}(\gamma)$-calibrated submanifolds that are $\omega_{\mathrm{KE}}$-isotropic. More precisely:

\begin{thm} Let $Z$ be the $(4n+2)$-dimensional twistor space of a positive quaternionic-K\"{a}hler $4n$-orbifold.  Then $Z$ admits an $\Sp(n)\U(1)$-structure $\gamma \in \Omega^3(Z; \C)$ compatible with the K\"{a}hler-Einstein and nearly-K\"{a}hler structures.  Moreover:
\begin{itemize}
\item The $3$-form $\mathrm{Re}(\gamma)$ is a semi-calibration (i.e., has comass one).
\item If $\Sigma^3$ is compact, $\mathrm{Re}(\gamma)$-calibrated, and $\omega_{\mathrm{KE}}$-isotropic, then with respect to the K\"ahler-Einstein metric, $\Sigma$ is a geodesic circle bundle over a totally-complex surface in $Q$. (See Definition~\ref{defn:totally-complex}.) Conversely, any such circle bundle is $\mathrm{Re}(\gamma)$-calibrated and $\omega_{\mathrm{KE}}$-isotropic. (See Theorem~\ref{thm:ReGammaCircleBundle}.)
\end{itemize}
\end{thm}

We remark that there is a difference between the cases $n=1$ and $n \geq 2$, so our proof handles them separately. In $\S$\ref{sec:Re(gamma)3folds} we undertake a detailed study of $\mathrm{Re}(\gamma)$-calibrated $3$-folds in $Z^{4n+2}$. In a certain precise sense, these are generalizations of special Lagrangian $3$-folds in nearly-K\"{a}hler $6$-manifolds.  \\
\indent Geometric structures in place, we establish a series of relationships between the various classes of submanifolds in  $M$, $Z$, and $Q$; see diagram~\eqref{diamond-diag}. That is, given a submanifold $\Sigma \subset Z$, we ask how various first-order conditions on $\Sigma$ (e.g., complex, Lagrangian, etc.) influence the geometry of a local $p_{(1,0,0)}$-horizontal lift $\widehat{\Sigma} \subset M$ (provided one exists) and its $p_{(1,0,0)}$-circle bundle $p_{(1,0,0)}^{-1}(\Sigma) \subset M$, and vice versa.  Similarly, starting with a totally-complex $U \subset Q^{4n}$, we study its \emph{$\tau$-horizontal lift} $\widetilde{U} \subset Z$ and its \emph{geodesic circle bundle lift} $\mathcal{L}(U) \subset Z$:
\begin{align*}
\widetilde{U}|_x & = \left\{ j \in Z_x \colon j(T_xU) = T_xU\right\}\!, & \mathcal{L}(U)|_x & = \left\{ j \in Z_x \colon j(T_xU) \subset (T_xU)^\perp \right\}\!.
\end{align*}
See $\S$\ref{subsec:TotallyComplex} for a detailed discussion. \\
\indent Altogether, the litany of propositions and theorems --- proven in $\S$\ref{sec:SubmanifoldsMZ}, $\S$\ref{subsec:TotallyComplex}, and $\S$\ref{sec:CharCplxLagCones} --- comprise a sort of ``dictionary" of submanifold geometries.  As an example, in $\S$\ref{subsec:TotallyComplex}, we obtain the following characterization of the compact submanifolds of $Z$ that are Lagrangian with respect to both $\omega_{\mathrm{KE}}$ and $\omega_{\mathrm{NK}}$, generalizing a result of Storm \cite{Storm} to higher dimensions:

\begin{thm} \label{thm:Main2} Recall diagram~\eqref{diamond-diag}.
\begin{enumerate}[(a)]
\item If $\Sigma^{2n+1} \subset Z^{4n+2}$ is a compact $(2n+1)$-dimensional submanifold that is both $\omega_{\mathrm{KE}}$-Lagrangian and $\omega_{\mathrm{NK}}$-Lagrangian, then $\Sigma = \mathcal{L}(U)$ for some totally-complex $2n$-fold $U^{2n} \subset Q^{4n}$ (respectively, superminimal surface if $n = 1$).
\item Conversely, if $U^{2n} \subset Q^{4n}$ is totally-complex and $n \geq 2$, or if $U$ is a superminimal surface and $n = 1$, then $\mathcal{L}(U) \subset Z$ is $\omega_{\mathrm{KE}}$-Lagrangian and $\omega_{\mathrm{NK}}$-Lagrangian.
\end{enumerate}
\end{thm}

\indent As another example, in $\S$\ref{sec:CharCplxLagCones}, we provide several characterizations of complex isotropic cones in hyperk\"{a}hler cones $C^{4n+4}$ in terms of submanifold geometries in $M$, $Z$, and $Q$.  In particular, we prove the following theorem, generalizing a result of Ejiri--Tsukada \cite{ET-2012} on complex isotropic cones of top dimension $2n+2$ in $C = \HH^{n+1}$.

\begin{thm} \label{thm:Main3} Recall diagram~\eqref{diamond-diag}. Let $L^{2k+1} \subset M^{4n+3}$ be a compact submanifold, where $3 \leq 2k+1 \leq 2n+1$.  The following conditions are equivalent:
\begin{enumerate}[(1)]
\item The cone $\mathrm{C}(L)$ is complex isotropic with respect to $\cos(\theta) I_2 + \sin(\theta) I_3$ for some $e^{i \theta} \in S^1$.
\item The link $L$ is locally of the form $p_{(0, \cos(\theta), \sin(\theta))}^{-1}(\widetilde{U})$ for some totally-complex submanifold $U^{2k} \subset Q$ (resp., superminimal surface if $n = 1$) and some $e^{i\theta} \in S^1$.
\item The link $L$ is locally a $p_{(1,0,0)}$-horizontal lift of $\mathcal{L}(U) \subset Z$ for some totally-complex submanifold $U^{2k} \subset Q^{4n}$ (resp., superminimal surface $U^2 \subset Q^4$ if $n = 1$).
\end{enumerate}
\end{thm}
A more detailed statement appears as Theorem \ref{thm:MainArbDim}.  Moreover, additional characterizations are available for complex isotropic cones $\mathrm{C}(L) \subset C$ of top dimension $2n+2$ and lowest dimension $4$: see Theorems \ref{thm:TopDimension} and \ref{thm:LowDimension}, respectively. \\
\indent Intuitively, Theorem \ref{thm:Main3} states that the link $L^{2k+1} \subset M$ of a complex isotropic cone in $C^{4n+4}$ can be manufactured from a totally-complex submanifold $U^{2k} \subset Q$ in two ways.  By (2), one can first consider its $\tau$-horizontal lift $\widetilde{U} \subset Z$, and then take the resulting $p_{(0, \cos(\theta),\sin(\theta))}$-circle bundle.  On the other hand, by (3), one could instead begin with the geodesic circle bundle lift $\mathcal{L}(U) \subset Z$, and then take a $p_{(1,0,0)}$-horizontal lift to $M$.  Thus, in a sense, the operations of ``circle bundle lift" and ``horizontal lift" commute with one another.

\indent Broadly speaking, Theorems~\ref{thm:Main2} and~\ref{thm:Main3} illustrate that a great variety of distinct classes of semi-calibrated submanifolds of a hyperk\"ahler cone, $3$-Sasakian manifold, or twistor space can only arise as particular constructions built from totally-complex submanifolds, which is not at all evident from their definitions. Consequently, such submanifolds are essentially as plentiful as totally-complex submanifolds. See Example~\ref{ex:TC} for some explicit totally-complex submanifolds.

\subsection{Organization and Conventions}

\indent \indent In $\S$\ref{sec:CalGeoHK}, we discuss several calibrated geometries in hyperk\"{a}hler manifolds $C^{4n+4}$, including the complex, special Lagrangian, complex isotropic, special isotropic, and Cayley submanifolds.  Then, starting in $\S$\ref{sec:CalGeo3Sas}, we assume that $C = \Cone(M)$ is a hyperk\"{a}hler cone over a $3$-Sasakian manifold $M^{4n+3}$.  We spend $\S$\ref{sec:3SasLink} reviewing $3$-Sasakian geometry, turning to the submanifold theory of $M$ in $\S$\ref{sec:SubmanifoldsSE}-\ref{sec:Submanifolds3Sas}.  In $\S$\ref{sec:3SasCircleBundle}, we introduce a complex $3$-form $\Gamma_1 \in \Omega^3(M;\C)$ and prove that it descends via $p_{(1,0,0)} \colon M \to Z$ to a $3$-form $\gamma \in \Omega^3(Z;\C)$ on the twistor space. \\
\indent Section 4 concerns submanifold theory in twistor spaces.  After discussing $\Sp(n)\U(1)$-structures on arbitrary $(4n+2)$-manifolds in $\S$\ref{sub:Sp(n)U(1)}, we show in $\S$\ref{sub:GeomTwistor} that the $3$-form $\gamma \in \Omega^3(Z;\C)$ defines such a structure on the twistor space.  Then, in $\S$\ref{sec:Re(gamma)3folds}-\ref{sec:SubmanifoldsMZ}, we study various classes of submanifolds of $Z$, establishing a series of relationships between those in $Z$ and those in $M$. \\
\indent In $\S$\ref{subsec:TotallyComplex}, we consider totally-complex submanifolds of quaternionic-K\"{a}hler manifolds $Q$, and relate them to submanifold geometries in $M$ and $Z$.  Finally, in $\S$\ref{sec:CharCplxLagCones}, we provide several characterizations of complex isotropic cones in $C$.  This paper also includes two appendices: Appendix~\ref{appendix:calib} collects some results on the linear algebra of calibrations that we use, and Appendix~\ref{appendix:cones} gives a brief introduction to metric cones and their associated conical differential forms.\\

\noindent \textbf{Notation and conventions.}
\begin{itemize} \setlength\itemsep{-1mm}
\item We often use $c_{\theta}, s_{\theta}$ to denote $\cos \theta, \sin \theta$, respectively, for brevity.
\item Repeated indices are summed over all of their allowed values unless explicitly stated otherwise. The symbol $\epsilon_{pqr}$ is the permutation symbol on three letters, so it vanishes if any two indices are equal, and it equals $\mathrm{sgn}(\sigma)$ if $p, q, r = \sigma(1), \sigma(2), \sigma(3)$.
\item A superscript on a manifold always denotes its \emph{real} dimension.
\item For a manifold $M$, we use $\mathrm{C}(M) = \R^+ \times M$ with metric $dr^2 + r^2 g_M$ to denote the metric cone over $M$, as discussed in Appendix~\ref{appendix:cones}.
\item If $L$ is a submanifold of $M$, then $NL$ denotes its normal bundle.  Submanifolds are assumed to be embedded. (Much of what we discuss works for immersed submanifolds, but not everything. See also Remark~\ref{rmk:embedded}.) Unless stated otherwise, all submanifolds are assumed to be connected and orientable, and thus have exactly two orientations.
\item We use interchangeably the terms \emph{semi-calibration} and \emph{comass one}. That is, a differential form $\alpha$ is a \emph{calibration} if it is a semi-calibration that satisfies $d\alpha = 0$.
\item The twistor space $Z^{4n+2}$ and the quaternionic-K\"ahler $Q^{4n}$ are in general \emph{orbifolds}. However, we avoid technical complications and work only over the \emph{smooth} parts of $Z$ and $Q$. That is, all submanifolds are assumed to not pass through any orbifold points of $Z$ or $Q$.
\end{itemize}

\noindent \textbf{Acknowledgements:} The research of the second author is supported by an NSERC Discovery Grant. The second author thanks Luc\'ia Mart\'in-Merch\'an for useful discussions. The third author thanks Laura Fredrickson, McKenzie Wang, and Micah Warren for conversations. All three authors are grateful to the referee for helpful suggestions that improved an earlier draft of this article.

\section{Calibrated Geometry in Hyperk\"{a}hler Manifolds} \label{sec:CalGeoHK}

\indent \indent Let $C^{4n+4}$ be a hyperk\"{a}hler manifold with $n \geq 1$.  The hyperk\"{a}hler structure on $C$ consists of the following data:
\begin{itemize}
\item A Riemannian metric $g_C$.
\item A triple of integrable almost-complex structures $(I_1, I_2, I_3) = (I,J,K)$ satisfying the quaternionic relations $I_1 I_2 = I_3$, etc, each of which is orthogonal with respect to $g_C$.
\item A triple of closed $2$-forms $(\omega_1, \omega_2, \omega_3)$ given by $\omega_p(X,Y) = g_C(I_p X, Y)$.
\end{itemize}
Note that $\omega_p$ is a K\"ahler form with respect to $I_p$, so in particular it is of type $(1,1)$ with respect to $I_p$. This means that $\omega_p (I_p X, I_p Y) = \omega(X, Y)$ and thus $g_C (X, Y) = \omega_p (X, I_p Y)$. We also have
\begin{equation} \label{eq:HK-relations}
\omega_p (I_q X, Y) = g_C (I_p I_q X, Y) = \epsilon_{pqr} g_C (I_r X, Y) = \epsilon_{pqr} \omega_r (X, Y).
\end{equation}
In fact, we have an $S^2$-family of K\"ahler structures: for any $v = (v_1, v_2, v_3) \in S^2$, we can take $I_v = \sum_{p=1}^3 v_p I_p$ and $\omega_v (X, Y) = g_C(I_v X, Y)$.

One can show that $C$ inherits a triple of complex-symplectic forms $\sigma_1, \sigma_2, \sigma_3 \in \Omega^2(C; \C)$ via
\begin{align*}
\sigma_1 & := \omega_2 + i\omega_3, & \sigma_2 & := \omega_3 + i\omega_1, & \sigma_3 & := \omega_1 + i\omega_2.
\end{align*}
A calculation shows that $\sigma_1$ is of $I_1$-type $(2,0)$, and analogously for $\sigma_2, \sigma_3$. It follows that each $\sigma_p$ is a \emph{holomorphic} symplectic form with respect to $I_p$.

Further, $C$ inherits the following triple of $(2n+2)$-forms $\Upsilon_1, \Upsilon_2, \Upsilon_3$:
\begin{align*}
\Upsilon_1 & = \frac{1}{(n+1)!}\sigma_1^{n+1},  & \Upsilon_2 & = \frac{1}{(n+1)!}\sigma_2^{n+1},  & \Upsilon_3 & = \frac{1}{(n+1)!}\sigma_3^{n+1}.
\end{align*}
Each $\Upsilon_p$ is a \emph{holomorphic volume form} with respect to $I_p$, so that $(g_C, I_p, \omega_p, \Upsilon_p)$ is a Calabi-Yau structure on $C$.  More generally, fixing $I_1$ as a reference, by considering the holomorphic volume form $e^{i (n+1)\theta} \Upsilon_1 = \frac{1}{(n+1)!} (e^{i \theta} \sigma_1)^{n+1}$, we obtain an $S^1$-family of Calabi-Yau structures with respect to $I_1$. Since $e^{i \theta} \sigma_1 = (c_\theta \omega_2 - s_\theta \omega_3) + i (s_\theta \omega_2 + c_\theta \omega_3)$, this $S^1$-family corresponds to rotating the orthogonal pair $I_2, I_3$ by $\theta$ in the equator of $S^2$ determined by the poles $\pm I_1$.

Finally, $C$ also admits a quaternionic-K\"{a}hler structure via the real $4$-form
\begin{align*}
\Lambda = \frac{1}{6}\omega_1^2 + \frac{1}{6}\omega_2^2  + \frac{1}{6}\omega_3^2.
\end{align*}
(See Definition~\ref{defn:QK} for our definition of quaternionic-K\"ahler.)

In this section, we recall various classes of distinguished submanifolds of $C$.  Some of these classes --- e.g., the complex, Lagrangian, special Lagrangian, and quaternionic --- arise from a Calabi-Yau or quaternionic-K\"{a}hler structure.  Others arise from a complex-symplectic structure, or are otherwise special to the hyperk\"{a}hler setting.

\subsection{Submanifolds via the Calabi-Yau and QK Structures} \label{subsec:SubViaCY}

\indent \indent Recall that every hyperk\"{a}hler manifold is a K\"ahler manifold in an $S^2$-family of ways, and given such a choice, it is a Calabi-Yau manifold in an $S^1$-family of ways. Due to these structures, we may consider the following classes of submanifolds: 

\begin{defn}
A submanifold $N^{2k} \subset C^{4n+4}$ is \emph{$I_1$-complex} if
$$\left. \frac{1}{k!}\omega_1^k\right|_N = \vol_N.$$
That is, if it is calibrated with respect to $\frac{1}{k!}\omega_1^k$.

It is \emph{$I_1$-anti-complex}, or $-I_1$-complex, if it is calibrated with respect to $-\frac{1}{k!}\omega_1^k$. Equivalently, if it is $I_1$-complex when equipped with the opposite orientation.

A submanifold is $\pm I_1$-complex if and only if its tangent spaces are $I_1$-invariant:
$$I_1(T_xN) = T_xN, \ \ \forall x \in N.$$
The definitions of $I_2$-complex and $I_3$-complex are analogous.
\end{defn}

\begin{defn}
A submanifold $N \subset C^{4n+4}$ is \emph{$\omega_1$-isotropic} if
$$\left.\omega_1\right|_N = 0.$$
An $\omega_1$-isotropic submanifold satisfies $\dim(N) \leq 2n+2$.  An \emph{$\omega_1$-Lagrangian} submanifold is an $\omega_1$-isotropic submanifold of maximal dimension $2n+2$.
\end{defn}

Let $X, Y \in TL$. Since $\omega_1(X, Y) = g(I_1 X, Y)$, we see that $L$ is $\omega_1$-isotropic if and only if $I_1 (TL) \subseteq NL$. If $N$ has dimension $2n+2$, then $I_1 (TL) = NL$ if and only if $L$ is $\omega_1$-Lagrangian. We use these facts repeatedly.

\begin{defn} Fix $\theta \in [0,2\pi)$.  A $(2n+2)$-dimensional submanifold $N^{2n+2} \subset C^{4n+4}$ is called \emph{$\Upsilon_1$-special Lagrangian of phase $e^{i \theta}$} if
$$\left.\text{Re}(e^{-i\theta}\Upsilon_1)\right|_N = \vol_N.$$
Equivalently~\cite[Corollary 1.11]{Harvey-Lawson}, there exists an orientation on $N^{2n+2}$ making it $\Upsilon_1$-special Lagrangian of phase $e^{i \theta}$ if and only if
\begin{align*}
\left.\text{Im}(e^{-i \theta}\Upsilon_1)\right|_N & = 0, & \left.\omega_1\right|_N & = 0.
\end{align*}
When the phase is left unspecified, we assume it to be $e^{i \theta} = 1$.
\end{defn}

\begin{rmk}
Every hyperk\"{a}hler manifold is also quaternionic-K\"{a}hler, and such manifolds admit a distinguished class of \emph{quaternionic submanifolds}. However, Gray~\cite{Gray} proved that such submanifolds are always totally geodesic. We will not consider quaternionic submanifolds in this paper.
\end{rmk}

\subsection{Submanifolds via the Hyperk\"{a}hler Structure} \label{subsec:SubViaHK}

\indent \indent In addition to the submanifolds discussed above, hyperk\"{a}hler manifolds also admit three more notable classes of submanifolds: the complex isotropic, special isotropic, and generalized Cayley submanifolds.  We discuss each of these in turn.

\subsubsection{Complex Isotropic Submanifolds}

\begin{defn}
A $2k$-dimensional submanifold $L^{2k} \subset C^{4n+4}$ is called \emph{$I_1$-complex isotropic} if it is both $I_1$-complex and $\sigma_1$-isotropic.  That is, if:
\begin{align*}
\left.\frac{1}{k!}\omega_1^k\right|_L & = \vol_L, & \left.\sigma_1\right|_L & = 0.
\end{align*}
Said another way, $L$ is $I_1$-complex, $\omega_2$-isotropic, and $\omega_3$-isotropic:
\begin{align*}
\left.\frac{1}{k!}\omega_1^k\right|_L & = \vol_L, & \left.\omega_2\right|_L & = 0, & \left.\omega_3\right|_L & = 0.
\end{align*}
\indent An \emph{$I_1$-complex Lagrangian} submanifold $L^{2n+2} \subset C^{4n+4}$ is an $I_1$-complex isotropic submanifold of maximal dimension $2n+2$.  That is, an $I_1$-complex Lagrangian submanifold is simultaneously $I_1$-complex, $\omega_2$-Lagrangian, and $\omega_3$-Lagrangian.  The definitions of $I_2$- and $I_3$-complex isotropic (resp., complex Lagrangian) are analogous.
\end{defn}

\indent Complex isotropic submanifolds are interesting from several points of view.  For example, in algebraic geometry, one often considers holomorphic symplectic manifolds that are fibered by complex Lagrangians, as in Sawon~\cite{Sawon}.  As another example, Doan--Rezchikov~\cite{Doan-Rezchikov} use complex Lagrangians as part of a hyperk\"{a}hler Floer theory.  In the differential geometry literature, complex isotropic submanifolds have been studied by, for example, Bryant--Harvey~\cite{Bryant-Harvey}, Hitchin~\cite{Hitchin-CL}, and Grantcharov--Verbitsky~\cite{Grantcharov-Verbitsky}.

\begin{prop} \label{prop:CplxIsoEquiv} Let $L^{2k} \subset C^{4n+4}$ be a $2k$-dimensional submanifold.  The following are equivalent:
\begin{enumerate}[(i)]
\item $L$ is $I_1$-complex, $\omega_2$-isotropic, and $\omega_3$-isotropic.
\item $L$ is $I_1$-complex and $\omega_2$-isotropic.
\end{enumerate}
\end{prop}
\begin{proof}
One direction is immediate.  For the converse, suppose $L$ is $I_1$-complex and $\omega_2$-isotropic. Let $X \in TL$, so that $I_1X \in TL$, and thus $- I_3 X = I_2(I_1X) \in NL$. Hence $I_3X \in NL$. This shows that $L$ is $\omega_3$-isotropic.
\end{proof}

\indent In the complex Lagrangian case, we can say more:

\begin{prop} \label{prop:CplxLagEquiv} Let $L^{2n+2} \subset C^{4n+4}$ be a $(2n+2)$-dimensional submanifold.  The following are equivalent:
\begin{enumerate}[(i)]
\item $L$ is $I_1$-complex, $\omega_2$-Lagrangian, and $\omega_3$-Lagrangian.
\item $L$ is $I_1$-complex and $\omega_2$-Lagrangian.
\item $L$ is $\omega_2$-Lagrangian and $\omega_3$-Lagrangian.
\item $L$ is $I_1$-complex, $\Upsilon_2$-special Lagrangian of phase $i^{n+1}$, and $\Upsilon_3$-special Lagrangian of phase $1$.
\end{enumerate}
\end{prop}

\begin{proof}
The equivalence (i)$\iff$(ii) was observed above.  It is clear that (i)$\implies$(iii).  For (iii)$\implies$(i), suppose that $L$ is $\omega_2$- and $\omega_3$-Lagrangian.  Let $X \in TL$, so that $I_3X \in NL$, and thus $I_1X = I_2(I_3X) \in TL$. Hence $L$ is $I_1$-complex. \\
\indent It is clear that (iv)$\implies$(i).  For (i)$\implies$(iv), suppose that $L$ is $I_1$-complex, $\omega_2$-Lagrangian, and $\omega_3$-Lagrangian.  Then $L$ satisfies $\left.\frac{1}{(n+1)!}\omega_1^{n+1}\right|_L = \vol_L$ and $\left.\omega_2\right|_L = 0$ and $\left.\omega_3\right|_L = 0$.  Recalling that
\begin{align*}
(-i)^{n+1}\Upsilon_2 & =  \frac{1}{(n+1)!}(\omega_1 - i\omega_3)^{n+1}, & \Upsilon_3 & =  \frac{1}{(n+1)!}(\omega_1 + i\omega_2)^{n+1},
\end{align*}
we have 
\begin{align*}
\left.\text{Re}((-i)^{n+1} \Upsilon_2)\right|_L & = \left.\frac{1}{(n+1)!}\omega_1^{n+1}\right|_L = \vol_L, & \left.\text{Re}( \Upsilon_3)\right|_L = \left.\frac{1}{(n+1)!}\omega_1^{n+1}\right|_L = \vol_L.
\end{align*}
\end{proof}
\subsubsection{Special Isotropic Submanifolds}

\indent \indent The following definition is due to Bryant--Harvey~\cite{Bryant-Harvey}. We prove that these forms are calibrations in Theorem~\ref{thm:special-isotropic-comass-one} in the appendix.
\begin{defn}
The \emph{special isotropic forms} are the $2k$-forms $\Theta_{I,2k}, \Theta_{J,2k}, \Theta_{K,2k} \in \Omega^{2k}(C)$ defined by
\begin{align*}
\Theta_{I,2k} & = \frac{1}{k!}\text{Re}(\sigma_1^k), & \Theta_{J,2k} & = \frac{1}{k!}\text{Re}(\sigma_2^k), & \Theta_{K,2k} & = \frac{1}{k!}\text{Re}(\sigma_3^k).
\end{align*}
A $2k$-dimensional submanifold $N^{2k} \subset C^{4n+4}$ is \emph{$\Theta_{I,2k}$-special isotropic} if it is calibrated by $\Theta_{I,k}$:
$$\left.\Theta_{I,2k}\right|_N = \vol_N.$$
The definitions of $\Theta_{J,2k}$- and $\Theta_{K,2k}$-special isotropic $2k$-manifold are analogous.
\end{defn}
\noindent Let us highlight the cases $2k = 2, 4, 2n+2$.

\begin{example} ${}$
\begin{enumerate}[(a)]
\item For $2k = 2$, the special isotropic $2$-forms are
\begin{align*}
\Theta_{I,2} & = \omega_2, & \Theta_{J,2} & = \omega_3, & \Theta_{K,2} & = \omega_1.
\end{align*}
In particular, a $\Theta_{I,2}$-special isotropic $2$-fold is the same as an $I_2$-complex $2$-fold.
\item For $2k = 4$, the special isotropic $4$-forms are
\begin{align*}
\Theta_{I,4} & = \frac{1}{2}(\omega_2^2 - \omega_3^2), & \Theta_{J,4} & = \frac{1}{2}(\omega_3^2 - \omega_1^2), & \Theta_{K,4} & = \frac{1}{2}(\omega_1^2 - \omega_2^2).
\end{align*}
In particular, if $L$ is an $I_1$-complex isotropic $4$-fold, then $L$ is both $-\Theta_{J,4}$-special isotropic and $\Theta_{K,4}$-special isotropic.
\item For $2k = 2n+2$, the special isotropic $(2n+2)$-forms are
\begin{align*}
\Theta_{I,2n+2} & = \text{Re}(\Upsilon_1), & \Theta_{J, 2n+2} & = \text{Re}(\Upsilon_2), & \Theta_{K, 2n+2} & = \text{Re}(\Upsilon_3).
\end{align*}
In particular, a $\Theta_{I, 2n+2}$-special isotropic $(2n+2)$-fold is the same as an $\Upsilon_1$-special Lagrangian, which explains the name ``special isotropic".
\end{enumerate}
\end{example}

At present, it appears that little is known about special isotropic $2k$-folds in hyperk\"{a}hler $(4n+4)$-manifolds when $2 < 2k < 2n+2$.

\subsubsection{Cayley $4$-folds}

\indent \indent The following definition is due to Bryant-Harvey~\cite{Bryant-Harvey}, though our sign conventions are opposite to theirs.

\begin{defn}
The \emph{generalized Cayley $4$-forms} are the $4$-forms $\Phi_1, \Phi_2, \Phi_3 \in \Omega^4(C)$ defined by
\begin{align*}
\Phi_{1} & = -\frac{1}{2}\omega_1^2 + \frac{1}{2}\omega_2^2 + \frac{1}{2}\omega_3^2, & \Phi_2 & = \frac{1}{2}\omega_1^2 - \frac{1}{2}\omega_2^2 + \frac{1}{2}\omega_3^2, & \Phi_3 & = \frac{1}{2}\omega_1^2 + \frac{1}{2}\omega_2^2 - \frac{1}{2}\omega_3^2.
\end{align*}
Note that
\begin{align} \label{Phi2-expressions}
\Phi_2 & = \ \frac{1}{2}\omega_1^2 - \Theta_{I,4} = \frac{1}{2}\omega_3^2 + \Theta_{K,4} = - \frac{1}{2} \omega_2^2 + \frac{1}{2} (\omega_1^2 + \omega_3^2).
\end{align}
and similarly for cyclic permutations.  A $4$-dimensional submanifold $N^4 \subset C^{4n+4}$ is \emph{$\Phi_2$-Cayley} if it is calibrated by $\Phi_2$:
$$\left.\Phi_2\right|_N = \vol_N.$$
The definitions of $\Phi_1$-Cayley and $\Phi_3$-Cayley are analogous.  
\end{defn}

\begin{rmk} Bryant and Harvey \cite[Lemma 2.14]{Bryant-Harvey} computed that the $\SO(4n+4)$-stabilizer of the generalized Cayley $4$-forms in $\R^{4n+4}$ are
$$\mathrm{Stab}(\Phi_1) = \begin{cases}
\Spin(7) & \mbox{if } n = 1 \\
\Sp(n+1)\mathrm{O}(2) & \mbox{if } n \geq 2.
\end{cases}$$
\end{rmk}

\indent This above definition was inspired by $\Spin(7)$-geometry, as we now recall.  If $(X^8, (g, \omega, I, \Upsilon))$ is a Calabi-Yau $8$-manifold, where $\omega \in \Omega^2(X)$ is the K\"{a}hler form and $\Upsilon \in \Omega^4(X;\C)$ is the holomorphic volume form, then $X$ inherits a torsion-free $\Spin(7)$-structure via the following formula:
\begin{align} \label{Phi-form}
\Phi & = \frac{1}{2}\omega^2 - \text{Re}(\Upsilon).
\end{align}
The real $4$-form $\Phi \in \Omega^4(X)$ is called the \emph{Cayley $4$-form}, and a $4$-dimensional submanifold $N \subset X$ satisfying $\Phi|_N = \vol_N$ is called \emph{Cayley}.  The following fact is well-known, but we include a proof for completeness.

\begin{prop}
Let $(X^8, (g,\omega, I, \Upsilon))$ be a Calabi-Yau $8$-manifold, and equip $X$ with its induced $\Spin(7)$-structure.  Let $N^4 \subset X$ be a $4$-dimensional submanifold.
\begin{enumerate}[(a)]
\item If $N$ is complex, then $N$ is Cayley.
\item If $N$ is special Lagrangian of phase $e^{i\pi} = -1$, then $N$ is Cayley.
\end{enumerate}
\end{prop}
\begin{proof}
If $N$ is complex, each tangent space $T_x N$ admits a basis of the form $\{ e_1, I e_1, e_2, I e_2 \}$. Then $v_k = e_k - i I e_k$ is of type $(1,0)$ for $k = 1, 2$, and $T_x N = e_1 \wedge I e_1 \wedge e_2 \wedge I e_2$ is a multiple of $v_1 \wedge \overline{v_1} \wedge v_2 \wedge \overline{v_2}$ and thus of type $(2,2)$. Since $\text{Re}(\Upsilon)$ is type $(4,0) + (0,4)$, it vanishes on $T_x N$. But $\frac{1}{2} \omega^2$ restricts to the volume form on $T_x N$, so by~\eqref{Phi-form}, $N$ is calibrated by $\Phi$.

If $N$ is special Lagrangian with phase $-1$, it is calibrated by $- \text{Re}(\Upsilon)$. Since it is also Lagrangian, $\frac{1}{2} \omega^2$ vanishes on $N$, and thus, again by~\eqref{Phi-form}, $N$ is calibrated by $\Phi$.
\end{proof}

When the ambient space is hyperk\"{a}hler, Bryant and Harvey showed that the above fact can be generalized to higher dimensions in the following sense.

\begin{prop}[{\cite[Theorem 8.20]{Bryant-Harvey}}] \label{prop:CayleyCases} Let $C^{4n+4}$ be a hyperk\"{a}hler $(4n+4)$-manifold.  Let $L^4 \subset C^{4n+4}$ be a $4$-dimensional submanifold.  Then:
\begin{enumerate}[(a)]
\item If $N$ is $I_1$-complex or $I_3$-complex, then $N$ is $\Phi_2$-Cayley.
\item If $N$ is $-\Theta_{I,4}$-special isotropic or $\Theta_{K,4}$-special isotropic, then $N$ is $\Phi_2$-Cayley.
\item If $N$ is $I_1$-complex isotropic, then $N$ is simultaneously $I_1$-complex, $-\Theta_{J,4}$-special isotropic, and $\Theta_{K,4}$-special isotropic, and hence is $\Phi_2$-Cayley.
\end{enumerate}
\end{prop}
\begin{proof} Parts (a) and (b) are contained in~\cite[Theorem 8.20]{Bryant-Harvey}. It is easy to see from~\eqref{Phi2-expressions} that (a) holds. For example, if $N$ is $I_1$-complex, then $\frac{1}{2} \omega_1^2$ restricts to the volume form, but $- \Theta_{I, 4} = - \mathrm{Re} (\frac{1}{2} \sigma_1^2)$ is of $I_1$-type $(4,0) + (0,4)$, and thus vanishes on $N$ since the tangent spaces of $N$ are of $I_1$-type $(2,2)$. Part (b) is less obvious, and uses a normal form for the tangent spaces of $N$. Details are given in~\cite[Sections 2 and 3]{Bryant-Harvey}. Part (c) is immediate from the first two.
\end{proof}

\begin{rmk} \label{rmk:SpInvariance} Note that every calibration $\phi \in \Omega^k(C)$ discussed in this section is stabilized by the Lie group $\Sp(n+1)$, which acts transitively on the unit sphere in $T_xC \simeq \R^{4n+4}$.  Consequently, at any point $x \in C$, every unit vector $v \in T_xC$ lies in some $\phi$-calibrated $k$-plane.
\end{rmk}

\subsection{Bookkeeping: Summary of Forms on $C$}

\indent \indent Starting in the next section, we will assume that the hyperk\"{a}hler manifold $C^{4n+4}$ is a metric cone, say $C = \mathrm{C}(M)$ for some Riemannian $(4n+3)$-manifold $M$.  Studying the geometry of $M$ and its relationship with $C$ will require the introduction of further tensors and differential forms.  So, before continuing, we briefly summarize the tensors and forms already defined on $C$:
\begin{align*}
& g_\Cone & \text{Riemannian metric} \\
& I_1, I_2, I_3 & \text{Complex structures} \\
& \omega_1, \omega_2, \omega_3 & \text{K\"{a}hler 2-forms} \\
& \Upsilon_1, \Upsilon_2, \Upsilon_3   & \text{Complex volume }(2n+2)\text{-forms} \\
& \sigma_1, \sigma_2, \sigma_3 & \text{Complex symplectic 2-forms} \\
& \Theta_{I,2k}, \Theta_{J,2k}, \Theta_{K,2k} & \text{Special isotropic } 2k\text{-forms} \\
& \Phi_1, \Phi_2, \Phi_3  & \text{Cayley 4-forms} \\
& \Lambda  & \text{Quaternionic 4-form}
\end{align*}

\section{Calibrated Geometry in $3$-Sasakian Manifolds} \label{sec:CalGeo3Sas}

\indent \indent If $(C^{4n+4}, g_\Cone) = (M \times \R^+, dr^2 + r^2g_M)$ is a hyperk\"{a}hler cone, then its link $M^{4n+3}$ inherits a \emph{$3$-Sasakian structure}, as we recall in $\S$\ref{sec:3SasLink}.  Then, in $\S$\ref{sec:SubmanifoldsSE} and $\S$\ref{sec:Submanifolds3Sas}, we explain how each of the calibrated geometries of $C$ discussed previously has a semi-calibrated counterpart in the $3$-Sasakian link $M$. \\
\indent In $\S$\ref{sec:3SasCircleBundle}, we recall that $M$ is the total space of a natural $S^1$-bundle $p_1 \colon M \to Z$.  The base space, $Z^{4n+2}$, called a \textit{twistor space}, admits both K\"{a}hler-Einstein and nearly-K\"{a}hler structures.  It is interesting to ask exactly how much geometric structure the map $p_1 \colon M \to Z$ preserves.  In this regard, we discover that every $3$-Sasakian manifold $M$ admits a natural $\C$-valued $3$-form $\Gamma_1 \in \Omega^3(M; \C)$ that descends to a $3$-form on $Z$ (Proposition \ref{prop:Descent}).  Later, in $\S$\ref{sub:GeomTwistor}, we will prove that the descended $3$-form endows $Z$ with a canonical $\Sp(n)\U(1)$-structure. \\
\indent Finally, in Theorem \ref{thm:ReGammaChar1}, we observe that $\mathrm{Re}(\Gamma_1) \in \Omega^3(M)$ is a semi-calibration, and classify the $\mathrm{Re}(\Gamma_1)$-calibrated $3$-folds in terms of more familiar geometries.

\subsection{$3$-Sasakian Manifolds as Links} \label{sec:3SasLink}

\begin{defn} Let $M$ be an odd-dimensional manifold.  An \emph{almost contact metric structure} on $M$ is a triple $(g_M, \alpha, \mathsf{J})$ consisting of a Riemannian metric $g_M$, a $1$-form $\alpha \in \Omega^1(M)$, and an endomorphism $\mathsf{J} \in \Gamma(\End(TM))$ satisfying
\begin{align*}
\alpha(\mathsf{J}X) & = 0, & \mathsf{J}(A) & = 0, & \mathsf{J}^2|_{\text{Ker}(\alpha)} & = -\Id, & g_M(\mathsf{J}X, \mathsf{J}Y) & = g_M(X,Y) - \alpha(X)\alpha(Y),
\end{align*}
where $A := \alpha^\sharp \in \Gamma(TM)$ is the \emph{Reeb vector field}. It follows that $\alpha(A) = 1$.
\end{defn}

\indent Thus, if $M$ is equipped with an almost contact metric structure, then each tangent space splits as
$$T_xM = \R A|_x \oplus \Ker(\alpha|_x).$$
Further, restricting to the hyperplane $\Ker(\alpha|_x) \subset T_xM$, the endomorphism $\mathsf{J} \colon \Ker(\alpha|_x) \to \Ker(\alpha|_x)$ is a $g_M$-orthogonal complex structure.  Thus, the hyperplane field $\Ker(\alpha) \subset TM$ is naturally endowed with the Hermitian structure $(g_M, \mathsf{J}, \Omega)$, where $\Omega := g_M(\mathsf{J} \cdot, \cdot)$ is the corresponding non-degenerate $2$-form.

\begin{defn} Let $M$ be a $(4n+3)$-manifold.  An \emph{$(\Sp(n) \times 3)$-structure} (or \emph{almost 3-contact metric structure}) on $M$ consists of data $(g_M, (\alpha_1, \alpha_2, \alpha_3), (\mathsf{J}_1, \mathsf{J}_2, \mathsf{J}_3))$ such that:
\begin{itemize}
\item Each triple $(g_M, \alpha_p, \mathsf{J}_p)$ is an almost contact metric structure ($p = 1, 2, 3$); and
\item Letting $A_p := \alpha_p^\sharp \in \Gamma(TM)$ denote the corresponding Reeb fields, we require
\begin{align*}
\mathsf{J}_p \circ \mathsf{J}_q - \alpha_p \otimes A_q & = \epsilon_{pqr}\mathsf{J}_r - \delta_{pq}\,\Id, \\
\mathsf{J}_p(A_q) & = \epsilon_{pqr}A_r.
\end{align*}
\end{itemize}
Note that there is no sum over $r$ in the above equations. For example, the above equations say $\mathsf{J}_1(A_1) = 0$, $\mathsf{J}_1(A_2) = A_3$, $\mathsf{J}_1(A_3) = - A_2$, that $\mathsf{J}_1^2 = - \Id$ on $\Ker(\alpha_1)$, and that $\mathsf{J}_1 \mathsf{J}_2 = \mathsf{J}_3$. Similarly for cyclic permutations of $1, 2, 3$.
\end{defn}

\indent Let $M^{4n+3}$ carry an $(\Sp(n) \times 3)$-structure.  We make three remarks.  First, for each $p = 1, 2, 3$, the tangent bundle splits as
\begin{equation} \label{eq:SplitOnce}
TM = \R A_p \oplus \Ker(\alpha_p),
\end{equation}
and the hyperplane field $\Ker(\alpha_p) \subset TM$ carries a Hermitian structure $(g_M, \mathsf{J}_p, \Omega_p)$, where $\Omega_p := g_M(\mathsf{J}_p \cdot, \cdot)$.  In fact, each $\Ker(\alpha_p)$ is also endowed with the complex volume form $\Psi_p \in \Lambda^{2n+1,0}(\Ker(\alpha_p))$ given by
\begin{equation} \label{eq:Psis}
\begin{aligned}
\Psi_1 & = (\alpha_2 + i\alpha_3) \wedge \frac{1}{n!}(\Omega_2 + i\Omega_3)^n, \\
\Psi_2 & = (\alpha_3 + i\alpha_1) \wedge \frac{1}{n!}(\Omega_3 + i\Omega_1)^n, \\
\Psi_3 & = (\alpha_1 + i\alpha_2) \wedge \frac{1}{n!}(\Omega_1 + i\Omega_2)^n.
\end{aligned}
\end{equation}
\indent Second, considering (\ref{eq:SplitOnce}) for $p = 1, 2, 3$ simultaneously, we see that the tangent bundle splits further as
\begin{equation} \label{eq:SplitThrice}
TM = \widetilde{\mathsf{H}} \oplus \widetilde{\mathsf{V}},
\end{equation}
where
\begin{align*}
\widetilde{\mathsf{H}} & = \Ker(\alpha_1, \alpha_2, \alpha_3), & \widetilde{\mathsf{V}} &= \R A_1 \oplus \R A_2 \oplus \R A_3.
\end{align*}
Note that the $4n$-plane field $\widetilde{\mathsf{H}} \subset TM$ is preserved by the three endomorphisms $\mathsf{J}_1, \mathsf{J}_2, \mathsf{J}_3$.  In fact, the restrictions of $\mathsf{J}_1, \mathsf{J}_2, \mathsf{J}_3$ to $\widetilde{\mathsf{H}}$ are $g_M$-orthogonal complex structures that satisfy the quaternionic relations $\mathsf{J}_1 \mathsf{J}_2 = \mathsf{J}_3$, etc. \\

\indent Third, we consider the relationship between the structure on a manifold $(M^{4n+3}, g_M)$ and that of its metric cone
$$ C^{4n+4} = \Cone(M) = (\R^+ \times M, g_\Cone = dr^2 + r^2 g_M). $$
In one direction, if $(M, g_M)$ is equipped with a compatible $(\Sp(n) \times 3)$-structure $(g_M, (\alpha_p), (\mathsf{J}_p))$, then the $(4n+4)$-manifold $C$ inherits a Riemannian metric $g_\Cone$, a triple of $g_\Cone$-orthogonal almost-complex structures $(I_1, I_2, I_3)$ satisfying $I_1I_2 = I_3$, etc., and a triple of non-degenerate $2$-forms $\omega_p$ defined by
\begin{align*}
g_\Cone & = dr^2 + r^2 g_M,  & \omega_p(X,Y) & = g_\Cone(I_pX,Y), & I_p(X) & = \begin{cases} \mathsf{J}_pX - \alpha_p(X)r\partial_r &\mbox{if } X \in TM, \\ A_p & \mbox{if } X = r\partial_r, \end{cases}
\end{align*}
where $X,Y \in TC$.  A computation shows that for each $p = 1,2,3$,
\begin{equation} \label{eq:omega-cone}
\omega_p = r\,dr \wedge \alpha_p + r^2\Omega_p.
\end{equation}
Altogether, the data $(g_\Cone, (\omega_1, \omega_2, \omega_3), (I_1, I_2, I_3))$ is an almost hyper-Hermitian structure on $C$. \\
\indent Conversely, if the metric cone $(C^{4n+4}, g_\Cone = dr^2 + r^2g_M)$ carries an almost hyper-Hermitian structure $(g_\Cone, (\omega_1, \omega_2, \omega_3), (I_1, I_2, I_3))$ that is conical in the sense of Definition~\ref{defn:homogeneous-form}, namely that
$$\mathscr{L}_{r\partial_r}(\omega_p) = 2\omega_p, \ \ \ p = 1,2,3,$$
then its link $(M, g_M)$ inherits a compatible $(\Sp(n) \times 3)$-structure $(g_M, (\alpha_p), (\mathsf{J}_p))$ via
\begin{align*}
\alpha_p & = (r\partial_r\,\lrcorner\,\omega_p)|_M, & \mathsf{J}_p & = \begin{cases} I_p & \mbox{on }\Ker(\alpha_p), \\ 0 & \mbox{on } \R A_p.\end{cases}
\end{align*}
This relationship leads to the following definition:

\begin{defn} Let $M$ be a $(4n+3)$-manifold.  A \emph{$3$-Sasakian structure} on $M$ is an $(\Sp(n) \times 3)$-structure $(g_M, (\alpha_p), (\mathsf{J}_p))$ for which the induced almost hyper-Hermitian structure $(g_\Cone, (\omega_p), (I_p))$ on its metric cone $\Cone(M) = \R^+ \times M$ hyperk\"{a}hler. \\
\indent Note that this is equivalent to requiring that the $2$-forms $\omega_1, \omega_2, \omega_3$ are all closed. (See, for example, \cite[Section 2]{Hitchin-SL}.)
\end{defn}

\subsubsection{Distinguished Forms on $3$-Sasakian Manifolds}

\indent \indent For the remainder of this work, $M^{4n+3}$ will denote a $3$-Sasakian $(4n+3)$-manifold with $3$-Sasakian structure $(g_M, (\alpha_1, \alpha_2, \alpha_3), (\mathsf{J}_1, \mathsf{J}_2, \mathsf{J}_3))$.  The induced conical hyperk\"{a}hler structure on $C^{4n+4} = \R^+ \times M$ will be denoted $(g_\Cone, (\omega_1, \omega_2, \omega_3), (I_1, I_2, I_3))$.  In this section, we record some of the distinguished differential forms on $M$ and compute their exterior derivatives. \\

\indent To begin, we consider the contact $1$-forms $\alpha_1, \alpha_2, \alpha_3 \in \Omega^1(M)$ and the transverse K\"{a}hler forms $\Omega_1, \Omega_2, \Omega_3 \in \Omega^2(M)$ defined by $\Omega_p(X,Y) = g_M(\mathsf{J}_pX, Y)$.  By (\ref{eq:omega-cone}), we may compute
\begin{align*}
0 = d\omega_p = d(r \,dr \wedge \alpha_p) + d(r^2 \Omega_p) = r\,dr \wedge (-d\alpha_p + 2 \Omega_p) + r^2 d\Omega_p
\end{align*}
which implies that
\begin{align}
d\alpha_p & = 2\Omega_p, & d\Omega_p & = 0. \label{eq:deriv-alpha}
\end{align}
(The first equation in~\eqref{eq:deriv-alpha} shows that each $\alpha_p$ is indeed a contact form. That is, that $\alpha_p \wedge (d\alpha_p)^{2n+1}$ is nowhere zero.)

\indent Next, we decompose the $2$-forms $\Omega_1, \Omega_2, \Omega_3$ according to the splitting
$$\Lambda^2(T^*M) = \Lambda^2(\widetilde{\mathsf{V}}^*) \oplus (\widetilde{\mathsf{V}} \otimes \widetilde{\mathsf{H}}) \oplus \Lambda^2(\widetilde{\mathsf{H}}^*).$$
One can show that each $\Omega_p$ has no component in $\widetilde{\mathsf{V}}^* \otimes \widetilde{\mathsf{H}}^*$, and that the $\Lambda^2(\widetilde{\mathsf{V}}^*)$-component of $\Omega_1$ is $\alpha_2 \wedge \alpha_3$.  Letting $\kappa_1, \kappa_2, \kappa_3$ denote the $\Lambda^2(\widetilde{\mathsf{H}}^*)$-component of $\Omega_p$, we arrive at the formulas
\begin{align} \label{eq:Omegas}
\Omega_1 & = \alpha_2 \wedge \alpha_3 + \kappa_1, & \Omega_2 & = \alpha_3 \wedge \alpha_1 + \kappa_2, & \Omega_3 & = \alpha_1 \wedge \alpha_2 + \kappa_3.
\end{align}
Taking $d$ of~\eqref{eq:Omegas} and using~\eqref{eq:deriv-alpha} shows that
\begin{align}
d\kappa_1 & = 2(\alpha_2 \wedge \Omega_3 - \alpha_3 \wedge \Omega_2) = 2(\alpha_2 \wedge \kappa_3 - \alpha_3 \wedge \kappa_2), \nonumber  \\
d\kappa_2 & = 2(\alpha_3 \wedge \Omega_1 - \alpha_1 \wedge \Omega_3) = 2(\alpha_3 \wedge \kappa_1 - \alpha_1 \wedge \kappa_3), \label{eq:deriv-kappa} \\
d\kappa_3 & = 2(\alpha_1 \wedge \Omega_2 - \alpha_2 \wedge \Omega_1) = 2(\alpha_1 \wedge \kappa_2 - \alpha_2 \wedge \kappa_1). \nonumber
\end{align}
Finally, recalling the transverse complex volume forms $\Psi_1, \Psi_2, \Psi_3 \in \Omega^{2n+1}(M; \C)$ of~\eqref{eq:Psis}, we compute
\begin{align} \label{eq:deriv-Psi}
d\Psi_1 & = \frac{2}{n!}(\Omega_2 + i\Omega_3)^{n+1}, & d\Psi_2 & = \frac{2}{n!}(\Omega_3 + i\Omega_1)^{n+1}, & d\Psi_3 & = \frac{2}{n!}(\Omega_1 + i\Omega_2)^{n+1}.
\end{align}

\indent To conclude this section, we summarize the relationships between various forms on the hyperk\"{a}hler cone $C^{4n+4}$ and those on its $3$-Sasakian link $M^{4n+3}$.  

\begin{prop} We have
\begin{align}
\omega_1 & = r\,dr \wedge \alpha_1 + r^2\,\Omega_1, \label{eq:OmegaCone1} \\
\frac{1}{2}\omega_1^2 & = r^3\,dr \wedge (\alpha_1 \wedge \Omega_1) + r^4\,\frac{1}{2}\Omega_1^2, \label{eq:OmegaCone2} \\
\frac{1}{k!}\omega_1^k & = r^{2k-1}\,dr \wedge \frac{1}{(k-1)!}(\alpha_1 \wedge \Omega_1^{k-1}) + r^{2k}\,\frac{1}{k!}\Omega_1^k. \label{eq:OmegaConek}
\end{align}
Consequently,
\begin{align}
\Upsilon_1 & = r^{2n+1}\,dr \wedge \Psi_1 + r^{2n+2}\,\frac{1}{(n+1)!} (\Omega_2 + i\Omega_3)^{n+1}, \label{eq:UpsilonCone} \\
\Theta_{I,4} & = r^3\,dr \wedge (\alpha_2 \wedge \Omega_2 - \alpha_3 \wedge \Omega_3) + r^4 \frac{1}{2}(\Omega_2^2 - \Omega_3^2),  \nonumber \\
\Phi_1 & = r^3\,dr \wedge (-\alpha_1 \wedge \Omega_1 + \alpha_2 \wedge \Omega_2 + \alpha_3 \wedge \Omega_3) + r^4\, \frac{1}{2}(-\Omega_1^2 + \Omega_2^2 + \Omega_3^2), \nonumber \\
\Lambda & = r^3\,dr \wedge \frac{1}{3}(\alpha_1 \wedge \Omega_1 + \alpha_2 \wedge \Omega_2 + \alpha_3 \wedge \Omega_3) + r^4\, \frac{1}{6}(\Omega_1^2 + \Omega_2^2 + \Omega_3^2). \nonumber
\end{align}
\end{prop}
\begin{proof} Each of these follow from a straightforward calculation.
\end{proof}

\subsection{Submanifolds via the Sasaki-Einstein Structure} \label{sec:SubmanifoldsSE}

\indent \indent By analogy with our discussion in $\S$\ref{subsec:SubViaCY} and $\S$\ref{subsec:SubViaHK}, we now consider the various classes of submanifolds of $M$.  We begin with those defined in terms of a Sasaki-Einstein structure. \\

\indent By Remark~\ref{rmk:SpInvariance}, we can apply Proposition~\ref{prop:rich-calibs} to (\ref{eq:OmegaConek}) with $k$ replaced by $k+1$. We deduce that for $p = 1,2,3$, the $(2k+1)$-forms
$$\frac{1}{k!}(\alpha_p \wedge \Omega_p^k) \in \Omega^{2k+1}(M)$$
are semi-calibrations.  Their calibrated submanifolds are called \emph{$I_p$-CR submanifolds}.  To be precise:

\begin{prop} \label{prop:CREquivalence} Let $L^{2k+1} \subset M^{4n+3}$ be a $(2k+1)$-dimensional submanifold.  We say $L$ is \emph{$I_1$-CR} if any of the following equivalent conditions holds:
\begin{enumerate}[(i)]
\item $\mathrm{C}(L) \subset C$ is $I_1$-complex.  That is, each tangent space of $\mathrm{C}(L)$ is $I_1$-invariant.
\item $\mathrm{C}(L)$ is (up to a change of orientation) $\frac{1}{(k+1)!}\omega_1^{k+1}$-calibrated:
$$\left.\frac{1}{(k+1)!}\,\omega_1^{k+1}\right|_{\mathrm{C}(L)} = \vol_{\mathrm{C}(L)}.$$
\item Each tangent space $T_xL$ is $\mathsf{J}_1$-invariant and contains the Reeb vector $A_1$.
\item $L$ satisfies (up to a change of orientation) that
$$\left.\frac{1}{k!}(\alpha_1 \wedge \Omega_1^k)\right|_L = \vol_L.$$
\end{enumerate}
\end{prop}
\begin{proof} The equivalences (i)$\iff$(ii)$\iff$(iii) are well-known.  The equivalence (ii)$\iff$(iv) follows from Proposition~\ref{prop:rich-calibs}.
\end{proof}

\begin{prop} \label{prop:LegendrianEquivalence} Let $L^k \subset M^{4n+3}$ be a submanifold.  We say $L$ is \emph{$\alpha_1$-isotropic} (resp., \emph{$\alpha_1$-Legendrian} if $k = 2n+1$) if any of the following equivalent conditions holds:
\begin{enumerate}[(i)]
\item $\mathrm{C}(L)$ is $\omega_1$-isotropic: $\left.\omega_1\right|_{\mathrm{C}(L)} = 0.$
\item $\left.\alpha_1\right|_L = 0.$
\item  $\left.\alpha_1\right|_L = 0$ and $\left.\Omega_1\right|_L = 0.$
\end{enumerate}
In particular, an $\alpha_1$-isotropic submanifold $L \subset M$ satisfies $\dim(L) \leq 2n+1$.
\end{prop}
\begin{proof} The first equation in~\eqref{eq:deriv-alpha} shows the equivalence (ii)$\iff$(iii). The equivalence (i)$\iff$(iii) follows from~\eqref{eq:OmegaCone1}.
\end{proof}

\indent Next, from formula (\ref{eq:UpsilonCone}) together with Proposition~\ref{prop:rich-calibs} and Remark \ref{rmk:SpInvariance}, we observe that for $p = 1,2,3$ and a constant $e^{i \theta} \in S^1$, the $(2n+1)$-forms
$$\mathrm{Re}(e^{-i\theta}\Psi_p) \in \Omega^{2n+1}(M)$$
are semi-calibrations. Their calibrated submanifolds are called \emph{$\Psi_p$-special Legendrian submanifolds of phase $e^{i \theta}$}.  We observe:

\begin{prop} \label{prop:ConeSpecialLag} Let $L^{2n+1} \subset M^{4n+3}$ be a $(2n+1)$-dimensional submanifold.  We say $L$ is \emph{$\Psi_1$-special Legendrian} if any of the following equivalent conditions holds:
\begin{enumerate}[(i)]
\item $\mathrm{C}(L)$ is (up to a change of orientation) $\Upsilon_1$-special Lagrangian: $\left.\mathrm{Re}(\Upsilon_1)\right|_{\mathrm{C}(L)} = \vol_{\mathrm{C}(L)}$.
\item $\mathrm{C}(L)$ satisfies $\left.\omega_1\right|_{\Cone(L)} = 0$ and $\left.\mathrm{Im}(\Upsilon_1)\right|_{\Cone(L)} = 0$.
\item $L$ satisfies (up to a change of orientation) that $\left.\mathrm{Re}(\Psi_1)\right|_{L} = \vol_{L}$.
\item $L$ satisfies $\left.\alpha_1\right|_{L} = 0$ and $\left.\mathrm{Im}(\Psi_1)\right|_{L} = 0$.
\end{enumerate}
\end{prop}
\begin{proof} The equivalence (i)$\iff$(ii) is well-known. The equivalence (ii)$\iff$(iv) follows from equation~\eqref{eq:UpsilonCone} and Proposition~\ref{prop:LegendrianEquivalence}. The equivalence (i)$\iff$(iii) follows from~\eqref{eq:UpsilonCone}, Remark~\ref{rmk:SpInvariance}, and Proposition~\ref{prop:rich-calibs}.
\end{proof}

\subsection{Submanifolds via the $3$-Sasakian Structure} \label{sec:Submanifolds3Sas}

\indent \indent We now turn to those submanifolds of $M$ whose definition requires more than the Sasaki-Einstein structure.  Here, we will discuss the \emph{CR isotropic}, \emph{special isotropic}, and \emph{associative} submanifolds.

\subsubsection{CR Isotropic Submanifolds} \label{subsub:CRIso}

\begin{prop} \label{prop:ConeComplexLag} Let $L^{2k+1} \subset M^{4n+3}$ be a $(2k+1)$-dimensional submanifold, $1 \leq k \leq n$.  We say $L$ is \emph{$I_1$-CR isotropic} (respectively, \emph{$I_1$-CR Legendrian} if $k = n$) if any of the following equivalent conditions holds:
\begin{enumerate}[(i)]
\item $\Cone(L) \subset C$ is $I_1$-complex, $\omega_2$-isotropic, and $\omega_3$-isotropic.
\item $\Cone(L) \subset C$ is $I_1$-complex and $\omega_2$-isotropic.
\item $L$ is $I_1$-CR, $\alpha_2$-isotropic, and $\alpha_3$-isotropic.
\item $L$ is $I_1$-CR and $\alpha_2$-isotropic.
\end{enumerate}
\end{prop}

\begin{proof} The equivalence (i)$\iff$(ii) was shown in Proposition \ref{prop:CplxIsoEquiv}. The equivalences (i)$\iff$(iii) and (ii)$\iff$(iv) both follow directly from Proposition \ref{prop:CREquivalence} and Proposition \ref{prop:LegendrianEquivalence}.
\end{proof}

In the CR Legendrian case, we can say more:

\begin{cor} \label{cor:CRLegEquiv} Let $L^{2n+1} \subset M^{4n+3}$ be a $(2n+1)$-dimensional submanifold.  The following are equivalent:
\begin{enumerate}[(i)]
\item $\mathrm{C}(L)$ is $I_1$-complex, $\omega_2$-Lagrangian, and $\omega_3$-Lagrangian (i.e., $\mathrm{C}(L)$ is $I_1$-complex Lagrangian).
\item $\mathrm{C}(L)$ is $\omega_2$-Lagrangian and $\omega_3$-Lagrangian.
\item $\mathrm{C}(L)$ is $I_1$-complex, $\Upsilon_2$-special Lagrangian of phase $i^{n+1}$, and $\Upsilon_3$-special Lagrangian of phase $1$.
\item $L$ is $I_1$-CR, $\alpha_2$-Legendrian, and $\alpha_3$-Legendrian (i.e., $L$ is $I_1$-CR Legendrian).
\item $L$ is $\alpha_2$-Legendrian and $\alpha_3$-Legendrian.
\item $L$ is $I_1$-CR, $\Psi_2$-special Legendrian of phase $i^{n+1}$, and $\Psi_3$-special Lagrangian of phase $1$.
\end{enumerate}
\end{cor}

\begin{proof} The equivalence (i)$\iff$(ii)$\iff$(iii) was shown in Proposition \ref{prop:CplxLagEquiv}.  The equivalence (i)$\iff$(iv) was shown in Proposition \ref{prop:ConeComplexLag}.  Finally, (ii)$\iff$(v) follows from Proposition \ref{prop:LegendrianEquivalence}, and (iii)$\iff$(vi) follows from Proposition \ref{prop:ConeSpecialLag}.
\end{proof}

\indent Examples of CR isotropic submanifolds can be constructed via Example~\ref{ex:TC} together with Corollary~\ref{cor:LiftOfCircleBundle}.

\subsubsection{Special Isotropic Submanifolds}

\begin{defn} \label{defn:special-isotropic} The \emph{special isotropic forms} on $M$ are the real $(2k-1)$-forms $\theta_{I,2k-1}, \theta_{J,2k-1}, \theta_{K,2k-1} \in \Omega^{2k-1}(M)$ defined by
\begin{align*}
\theta_{I, 2k-1} & := (r \partial_r\,\lrcorner\,\Theta_{I,2k})|_M, & \theta_{J,2p-1} & := (r \partial_r\,\lrcorner\,\Theta_{J, 2k})|_M, & \theta_{K,2k-1} & := (r \partial_r\,\lrcorner\,\Theta_{K,2k})|_M.
\end{align*}
In particular, for $2k-1 = 1, 3, 2n+1$, these are:
\begin{align*}
\theta_{I,1} & = \alpha_2, \\
\theta_{I,3} & = \alpha_2 \wedge \Omega_2 - \alpha_3 \wedge \Omega_3 = \alpha_2 \wedge \kappa_2 - \alpha_3 \wedge \kappa_3, \\
\theta_{I,2n+1} & = \text{Re}(\Psi_1).
\end{align*}
By Remark~\ref{rmk:SpInvariance}, Proposition~\ref{prop:rich-calibs}, and Theorem~\ref{thm:special-isotropic-comass-one}, the special isotropic forms $\theta_{I, 2k-1}, \theta_{J, 2k-1}, \theta_{K,2k-1}$ are semi-calibrations.
\end{defn}

\begin{prop} \label{prop:SpecIsoEquivLink} Let $L^{2k-1} \subset M^{4n+3}$ be a $(2k-1)$-dimensional submanifold, $1 \leq k \leq n+1$.  We say $L$ is \emph{$\theta_{I,2k-1}$-special isotropic} if either of the following equivalent conditions holds:
\begin{enumerate}[(i)]
\item $\mathrm{C}(L) \subset C$ is $\Theta_{I,2k}$-special isotropic.
\item $L$ is $\theta_{I,2k-1}$-special isotropic.
\end{enumerate}
\end{prop}
\begin{proof} This follows from Remark~\ref{rmk:SpInvariance} and Proposition~\ref{prop:rich-calibs}.
\end{proof}

\subsubsection{Associative $3$-folds}

\indent The following definition is due to Bryant-Harvey~\cite{Bryant-Harvey}.

\begin{defn} The \emph{generalized associative $3$-forms} are the real $3$-forms $\phi_1, \phi_2, \phi_3 \in \Omega^3(M)$ defined by
\begin{align*}
\phi_1 & = -\alpha_1 \wedge \Omega_1 + \alpha_2 \wedge \Omega_2 + \alpha_3 \wedge \Omega_3, \\
\phi_2 & = \alpha_1 \wedge \Omega_1 - \alpha_2 \wedge \Omega_2 + \alpha_3 \wedge \Omega_3, \\
\phi_3 & = \alpha_1 \wedge \Omega_1 + \alpha_2 \wedge \Omega_2 - \alpha_3 \wedge \Omega_3.
\end{align*}
Equivalently,
\begin{align*}
\phi_1 & = \alpha_1 \wedge \alpha_2 \wedge \alpha_3 - \alpha_1 \wedge \kappa_1 + \alpha_2 \wedge \kappa_2 + \alpha_3 \wedge \kappa_3, \\
\phi_2 & = \alpha_1 \wedge \alpha_2 \wedge \alpha_3 + \alpha_1 \wedge \kappa_1 - \alpha_2 \wedge \kappa_2 + \alpha_3 \wedge \kappa_3, \\
\phi_3 & = \alpha_1 \wedge \alpha_2 \wedge \alpha_3 + \alpha_1 \wedge \kappa_1 + \alpha_2 \wedge \kappa_2 - \alpha_3 \wedge \kappa_3,
\end{align*}
where the $\kappa_j$ were defined in~\eqref{eq:Omegas}. A $3$-dimensional submanifold $L^3 \subset M^{4n+3}$ is \emph{$\phi_1$-associative} if it is calibrated by $\phi_1$:
$$\left.\phi_1\right|_L = \vol_L.$$
The definitions of $\phi_2$-associative and $\phi_3$-associative are analogous.
\end{defn}

\indent Observing that
$$\Phi_1 = r^3\,dr \wedge \phi_1 + r^4\, \frac{1}{2}(-\Omega_1^2 + \Omega_2^2 + \Omega_3^2),$$
we obtain:

\begin{prop} \label{prop:AssocCayley} Let $L^3 \subset M^{4n+3}$ be a $3$-dimensional submanifold.  The following are equivalent:
\begin{enumerate}[(i)]
\item $\Cone(L) \subset C$ is $\Phi_1$-Cayley.
\item $L \subset M$ is $\phi_1$-associative.
\end{enumerate}
\end{prop}
\begin{proof} This follows from Remark~\ref{rmk:SpInvariance} and Proposition~\ref{prop:rich-calibs}.
\end{proof}

\indent Finally, we remark on the relationships between the above submanifolds. Let us recall that a manifold is called Sasaki-Einstein if its cone is Calabi-Yau and that a $7$-manifold is called nearly-parallel $\G_2$ if its cone is a $\mathrm{Spin}(7)$-manifold. Suppose now that $(Y^7, (g, \alpha, \mathsf{J}, \Psi))$ is a Sasaki-Einstein $7$-manifold.  It is well-known that $Y$ inherits a nearly-parallel $\G_2$-structure by the following formula:
\begin{equation*}
\phi = \alpha \wedge \Omega - \text{Re}(\Psi).
\end{equation*}
The real $3$-form $\phi \in \Omega^3(M)$ is called the \emph{associative $3$-form}, and a $3$-dimensional submanifold $\Sigma^3 \subset M$ satisfying $\phi|_\Sigma = \vol_\Sigma$ is called \emph{associative}.  The following fact is well-known, although we prove a more general result in Proposition~\ref{prop:AssocCases}.

\begin{prop} Let $(Y^7, (g, \alpha, \mathsf{J}, \Psi))$ be a Sasaki-Einstein $7$-manifold, and equip $Y$ with its induced nearly-parallel $\G_2$-structure $\phi$.  Let $L^3 \subset Y$ be a $3$-dimensional submanifold.  Then:
\begin{enumerate}[(a)]
\item If $L$ is CR, then $L$ is associative.
\item If $L$ is special Legendrian of phase $e^{i\pi} = -1$, then $L$ is associative.
\end{enumerate}
\end{prop}

\indent When the ambient space is $3$-Sasakian, the above fact generalizes to higher dimensions in the following way:

\begin{prop} \label{prop:AssocCases} Let $M^{4n+3}$ be a $3$-Sasakian $(4n+3)$-manifold.  Let $L^3 \subset M$ be a $3$-dimensional submanifold.  Then:
\begin{enumerate}[(a)]
\item If $L$ is $I_1$-CR or $I_3$-CR, then $L$ is $\phi_2$-associative.
\item If $L$ is $-\theta_{I,3}$-special isotropic or $\theta_{K,3}$-special isotropic, then $L$ is $\phi_2$-associative.
\item If $L$ is $I_1$-CR isotropic, then $L$ is simultaneously $I_1$-CR, $-\theta_{J,3}$-special isotropic, and $\theta_{K,3}$-special isotropic, and hence is $\phi_2$-associative.
\end{enumerate}
\end{prop}

\begin{proof} (a) If $L \subset M$ is $I_1$-CR (resp., $I_3$-CR), then Proposition \ref{prop:CREquivalence} implies that its cone $\mathrm{C}(L) \subset C$ is $I_1$-complex (resp., $I_3$-complex).  By Proposition \ref{prop:CayleyCases}(a), $\mathrm{C}(L)$ is $\Phi_2$-Cayley, so by Proposition \ref{prop:AssocCayley}, $L$ is $\phi_2$-associative. \\
\indent (b) If $L \subset M$ is $-\theta_{I,3}$-special isotropic (resp., $\theta_{K,3}$-special isotropic), then Proposition \ref{prop:SpecIsoEquivLink} implies that its cone $\mathrm{C}(L) \subset C$ is $-\Theta_{I,4}$-special isotropic (resp., $\Theta_{K,4}$-special isotropic). By Proposition \ref{prop:CayleyCases}(b), $\mathrm{C}(L)$ is $\Phi_2$-Cayley, so by Proposition \ref{prop:AssocCayley}, $L$ is $\phi_2$-associative.  \\
\indent (c) If $L$ is $I_1$-CR isotropic, then Proposition \ref{prop:ConeComplexLag} implies that $\mathrm{C}(L) \subset C$ is $I_1$-complex isotropic, and the result follows from an argument analogous to those used in parts (a) and (b).  Alternatively, if $L$ is $I_1$-CR isotropic, then by definition, $L$ is $I_1$-CR, $\alpha_2$-isotropic, and $\alpha_3$-isotropic.  Recalling that
\begin{align*}
\theta_{J,3} & = \alpha_3 \wedge \Omega_3 - \alpha_1 \wedge \Omega_1 & \theta_{K,3} & = \alpha_1 \wedge \Omega_1 - \alpha_2 \wedge \Omega_2,
\end{align*}
we observe that $L$ is $-\theta_{J,3}$- and $\theta_{K,3}$-special isotropic.
\end{proof}

\begin{rmk} Where associative $3$-folds in $3$-Sasakian manifolds $M^{4n+3}$ are concerned, the case $n = 1$ has received the most attention in light of the connection to $\mathrm{G}_2$-geometry.  Recently, several studies have considered the two $1$-parameter families of \emph{squashed} associative $3$-forms on $M^7$ given by
\begin{align*}
-\phi_{1,t}^- & = \alpha_1 \wedge \alpha_2 \wedge \alpha_3 + t^2(- \alpha_1 \wedge \kappa_1 + \alpha_2 \wedge \kappa_2 + \alpha_3 \wedge \kappa_3), \\
\phi_{1,t}^+ & = \alpha_1 \wedge \alpha_2 \wedge \alpha_3 - t^2(\alpha_1 \wedge \kappa_1 + \alpha_2 \wedge \kappa_2 + \alpha_3 \wedge \kappa_3).
\end{align*}
See, for example,~\cite{BM},~\cite{KL}, or~\cite{LO}.
\end{rmk}

\subsubsection{Summary}

The following table summarizes the relationships discussed above.
$$\begin{tabular}{| c | c || c | c |} \hline
$\dim(\mathrm{C}(L))$ & Cone $\mathrm{C}(L) \subset C$ & Link $L \subset M$ & $\dim(L)$ \\ \hline \hline
$2k$ & $I_1$-complex & $I_1$-CR & $2k-1$ \\ \hline
$2n+2$ & $\omega_1$-Lagrangian & $\alpha_1$-Legendrian & $2n+1$ \\ \hline
$\leq 2n+2$ & $\omega_1$-isotropic & $\alpha_1$-isotropic & $\leq 2n+1$ \\ \hline
$2n+2$ & $\Upsilon_1$-special Lagrangian & $\Psi_1$-special Legendrian & $2n+1$ \\ \hline
$2n+2$ & $I_1$-complex Lagrangian & $I_1$-CR Legendrian & $2n+1$ \\ \hline
$2k$ & $I_1$-complex isotropic & $I_1$-CR isotropic & $2k-1$ \\ \hline
$2k$ & $\Theta_{I,2k}$-special isotropic & $\theta_{I,2k-1}$-special isotropic & $2k-1$ \\ \hline
$4$ & $\Phi_1$-Cayley & $\phi_1$-associative & $3$ \\ \hline
\end{tabular}$$

\noindent With the exception of $\alpha_1$-Legendrian and $\alpha_1$-isotropic submanifolds, all of the ``link" submanifolds $L \subset M^{4n+3}$ that appear in the table are minimal (i.e., have zero mean curvature), because a calibrated cone is minimal, and the link of a minimal cone is minimal.

\subsection{$3$-Sasakian Manifolds as Circle Bundles} \label{sec:3SasCircleBundle}

\indent \indent From now on, $3$-Sasakian $(4n+3)$-manifolds $M$ are assumed to be compact.  Above, we viewed $M$ as the link of a hyperk\"{a}hler cone $C$.  In this section, we adopt a different perspective, viewing $M$ as the total space of a circle bundle. The starting point is the following result:

\begin{thm}[Boyer-Galicki~\cite{Boyer-Galicki}, Theorems 7.5.1, 13.2.5, 13.3.1] Let $M$ be a compact $3$-Sasakian $(4n+3)$-manifold.  For $v = (v_1, v_2, v_3) \in S^2$, let $A_v = v_1 A_1 + v_2 A_2 + v_3 A_3$ denote the corresponding Reeb field.  Then:
\begin{enumerate}[(a)]
\item Each $A_v$ defines a locally-free $S^1$-action on $M$ and quasi-regular foliation $\mathcal{F}_v \subset M$.  Let $Z_v := M/\mathcal{F}_v$ denote the corresponding leaf space, and let $p_v \colon M \to Z_v$ denote the projection.
\item The projection $p_v \colon M \to Z_v$ is a principal $S^1$-orbibundle with connection $1$-form $\alpha_v = \sum v^i \alpha_i$, and it is an orbifold Riemannian submersion.
\item For $v, v' \in S^2$, there is a diffeomorphism $Z_{v} \approx Z_{v'}$.  In fact, each $Z_v$ may be identified with the (orbifold) twistor space $Z$ of the quaternionic-K\"{a}hler $4n$-orbifold $Q = M/\mathcal{F}_A$, where $\mathcal{F}_A$ is the  $3$-dimensional foliation determined by the vector fields $A_1, A_2, A_3$.
\end{enumerate}
\end{thm}

\indent Thus, every compact $3$-Sasakian $(4n+3)$-manifold $M$ has a natural $S^2$-family of projections $p_v \colon M^{4n+3} \to Z^{4n+2}$.  For definiteness, we choose to work with $p_1 := p_{(1,0,0)} \colon M \to Z$, with respect to which $\alpha_1 \in \Omega^1(M)$ is a connection $1$-form.  On $M$, the choice of $p_1$ preferences the splitting $TM = \R A_1 \oplus \Ker(\alpha_1)$.  On the hyperk\"{a}hler cone $C^{4n+4} = \Cone(M)$, our choice distinguishes the K\"{a}hler structure $(g_{\Cone}, I_1, \omega_1)$.

\subsubsection{The $3$-forms $\Gamma_1, \Gamma_2, \Gamma_3$ and $4$-forms $\Xi_1, \Xi_2, \Xi_3$}

\indent We now introduce $\C$-valued $3$-forms $\Gamma_1, \Gamma_2, \Gamma_3 \in \Omega^3(M; \C)$ and $\R$-valued $4$-forms $\Xi_1, \Xi_2, \Xi_3 \in \Omega^4(M)$ that will play a key role in understanding the structure on the twistor space $Z$. These forms do not appear to have been studied before. Recalling the $2$-forms $\kappa_j$ defined in~\eqref{eq:Omegas}, we define
\begin{align}
\Gamma_1 & = (\alpha_2 - i\alpha_3) \wedge (\kappa_2 + i\kappa_3), \nonumber \\
\Gamma_2 & = (\alpha_3 - i\alpha_1) \wedge (\kappa_3 + i\kappa_1), \label{eq:Gammas} \\
\Gamma_3 & = (\alpha_1 - i\alpha_2) \wedge (\kappa_1 + i\kappa_2),\nonumber
\end{align}
and
\begin{align} \label{eq:Xis}
\Xi_1 & = \kappa_2^2 + \kappa_3^2, & \Xi_2 & = \kappa_3^2 + \kappa_1^2, & \Xi_3 & = \kappa_1^2 + \kappa_2^2.
\end{align}
Note that the real and imaginary parts of $\Gamma_1$ are given by
\begin{equation} \label{eq:ReImGammas}
\begin{aligned}
\text{Re}(\Gamma_1) & = \alpha_2 \wedge \kappa_2 + \alpha_3 \wedge \kappa_3, \\
\text{Im}(\Gamma_1) & = \alpha_2 \wedge \kappa_3 - \alpha_3 \wedge \kappa_2.
\end{aligned}
\end{equation}
Their exterior derivatives are given by:

\begin{prop} \label{prop:ExtDerivGammaXi} We have
\begin{align*}
d\,\mathrm{Re}(\Gamma_1) & = 2\Xi_1 - 4 \alpha_2 \wedge \alpha_3 \wedge \kappa_1, \\
d\,\mathrm{Im}(\Gamma_1) & = 0, \\
d\Xi_1 & = -4\kappa_1 \wedge \mathrm{Im}(\Gamma_1).
\end{align*}
\end{prop}

\begin{proof} This is a straightforward computation using the definitions ~\eqref{eq:ReImGammas} and~\eqref{eq:Xis} and the exterior derivative formulas ~\eqref{eq:deriv-alpha} and~\eqref{eq:deriv-kappa}.
\end{proof}

\begin{rmk} We remark in passing that one can compute
$$\kappa_1^2 + \kappa_2^2 + \kappa_3^2 = \textstyle \frac{1}{2}\,d(\alpha_{123} + \alpha_1 \wedge \kappa_1 + \alpha_2 \wedge \kappa_2 + \alpha_3 \wedge \kappa_3),$$
showing that the natural $4$-form $\kappa_1^2 + \Xi_1 = \kappa_1^2 + \kappa_2^2 + \kappa_3^2$ is exact.
\end{rmk}

\indent To clarify the geometric meaning of $\Gamma_1 \in \Omega^3(M; \C)$, we consider the $2$-form
$$\widetilde{\Omega}_1 := 2\kappa_1 - \alpha_2 \wedge \alpha_3.$$
Using equations~\eqref{eq:deriv-alpha},~\eqref{eq:Omegas},~\eqref{eq:deriv-kappa},~\eqref{eq:ReImGammas}, and Proposition \ref{prop:ExtDerivGammaXi}, we derive the identities
\begin{align*}
d \widetilde{\Omega}_1 & = 3\,\mathrm{Im}(2\Gamma_1), \\
d\,\mathrm{Re}(2\Gamma_1) & = 2(2\Xi_1 - 4\alpha_2 \wedge \alpha_3 \wedge \kappa_1).
\end{align*}
When $n = 1$, in which case $\dim(Z) = 6$ and $\dim(M) = 7$, there is a coincidence $\kappa_1^2 = \kappa_2^2 = \kappa_3^2$, which implies $\Xi_1 = 2\kappa_1^2$, and therefore
\begin{align*}
d \widetilde{\Omega}_1 & = 3\,\mathrm{Im}(2\Gamma_1), \\
d \,\mathrm{Re}(2\Gamma_1) & = 2\widetilde{\Omega}_1^2,
\end{align*}
which is familiar from the geometry of nearly-K\"{a}hler $6$-manifolds~\cite{ReyesCarrion-Salamon}.  So, when $n = 1$, the forms $\widetilde{\Omega}_1 \in \Omega^2(M)$ and $2\Gamma_1 \in \Omega^3(M;\C)$ are the pullbacks via $p_1 \colon M^7 \to Z^6$ of the nearly-K\"{a}hler $2$-form and complex volume form on $Z$, respectively.

\indent In $\S$\ref{sub:Sp(n)U(1)}-$\S$\ref{sub:GeomTwistor}, we will see that aspects of this picture persist in higher dimensions.  That is, for any $n \geq 1$, the $2$-form $\widetilde{\Omega}_1$ is the pullback of the nearly-K\"{a}hler $2$-form, while $2\Gamma_1$ is the pullback of a natural $3$-form that (together with other geometric data) defines an $\Sp(n)\U(1)$-structure on $Z$.  When $n = 1$, the $\Sp(1)\U(1) \cong \U(2)$-structure on $Z$ induces the familiar $\SU(3)$-structure, but when $n > 1$ the group $\Sp(n)\U(1)$ is not contained in $\SU(2n+1)$.

\subsubsection{$\mathrm{Re}(\Gamma_1)$-Calibrated $3$-folds}

\indent \indent The real parts of the $3$-forms $\Gamma_1, \Gamma_2, \Gamma_3 \in \Omega^3(M;\C)$ turn out to be semi-calibrations (Corollary \ref{cor:ReGammaComass}), and thus give rise to a distinguished class of $3$-folds of $M$.  The following theorem characterizes these submanifolds; we defer the proof to $\S$\ref{subsub:ReGammaClassify}, where the result is restated as Theorem \ref{thm:FirstCharDim3}. 

\begin{thm} \label{thm:ReGammaChar1}  ${}$
Let $L^3 \subset M^{4n+3}$ be a $3$-dimensional submanifold.  The following are equivalent:
\begin{enumerate}[(i)]
\item $\mathrm{C}(L)$ is a $(c_\theta I_2 + s_\theta I_3)$-complex isotropic $4$-fold for some constant $e^{i \theta} \in S^1$.
\item $L$ is a $(c_\theta I_2 + s_\theta I_3)$-CR isotropic $3$-fold for some constant $e^{i \theta} \in S^1$.
\item $L$ is $\mathrm{Re}(\Gamma_1)$-calibrated.
\end{enumerate}
\end{thm}

\indent Examples of $\mathrm{Re}(\Gamma_1)$-calibrated submanifolds can be constructed via Example~\ref{ex:TC} together with Theorem~\ref{thm:LowDimension}.

\subsubsection{Descent to $Z$}

\indent \indent  To conclude this section, we observe that certain differential forms defined on $M$ descend to the twistor space $Z$ via the map $p_1 \colon M \to Z$.  For this, we recall that a $k$-form $\phi \in \Omega^k(M)$ is called \emph{$p_1$-semibasic} if $\iota_X\phi = 0$ for all $X \in \mathrm{Ker}( (p_1)_*)$.  Since the fibers of $p_1 \colon M \to Z$ are connected, it is a standard fact that a $k$-form $\phi \in \Omega^k(M)$ descends to $Z$ if and only if both $\phi$ and $d\phi$ are $p_1$-semibasic.

\begin{prop} \label{prop:Descent} Consider the projection $p := p_1 \colon M \to Z$. 
\begin{enumerate}[(a)]
\item There exist $\R$-valued differential $2$-forms $\omega_{\mathsf{V}}$, $\omega_{\mathsf{H}}$, $\omega_{\mathrm{KE}}, \omega_{\mathrm{NK}} \in \Omega^2(Z)$ satisfying
\begin{align*}
\alpha_2 \wedge \alpha_3 & = p^*(\omega_{\mathsf{V}}), & \kappa_1 + \alpha_2 \wedge \alpha_3 = \Omega_1 & = p^*(\omega_\mathrm{KE}), \\
\kappa_1 & = p^*(\omega_{\mathsf{H}}), & 2\kappa_1 - \alpha_2 \wedge \alpha_3  = \widetilde{\Omega}_1 & = p^*(\omega_\mathrm{NK}).
\end{align*}
\item There exists a $\C$-valued differential $3$-form $\gamma \in \Omega^3(Z; \C)$ and an $\R$-valued differential $4$-form $\xi \in \Omega^4(Z)$ satisfying
\begin{align*}
\Gamma_1 & = p^*(\gamma) \\
\Xi_1 & = p^*(\xi).
\end{align*}
\end{enumerate}
\end{prop}

\begin{proof} (a) By equations (\ref{eq:deriv-alpha}) and (\ref{eq:deriv-kappa}), we have
\begin{align*}
d(\alpha_2 \wedge \alpha_3) & = -2(\alpha_2 \wedge \kappa_3 - \alpha_3 \wedge \kappa_2), & d\kappa_1 & = 2(\alpha_2 \wedge \kappa_3 - \alpha_3 \wedge \kappa_2).
\end{align*}
Therefore, both $\alpha_2 \wedge \alpha_3$ and $d(\alpha_2 \wedge \alpha_3)$ are $p_1$-semibasic, and similarly for $\kappa_1$ and $d\kappa_1$. \\
\indent (b) By Proposition \ref{prop:ExtDerivGammaXi}, we have
\begin{align*}
d\Gamma_1 & = 2(\kappa_2^2 + \kappa_3^2) - 4\alpha_2 \wedge \alpha_3 \wedge \kappa_1, & d\Xi_1 & = -4\kappa_1 \wedge (\alpha_2 \wedge \kappa_3 - \alpha_3 \wedge \kappa_2).
\end{align*}
Therefore, both $\Gamma_1$ and $d\Gamma_1$ are $p_1$-semibasic, and similarly for $\Xi_1$ and $d\Xi_1$.
\end{proof}

\begin{rmk} By contrast, one can check that the following forms on $M$ do \emph{not} descend via $p_1 \colon M \to Z$ to forms on $Z$:
\begin{align*}
& \kappa_2, \kappa_3, & & \Gamma_2, \Gamma_3, & & \alpha_1, \alpha_2, \alpha_3, & & \phi_1, \phi_2, \phi_3, \\
& \Omega_2, \Omega_3, & & \Xi_2, \Xi_3, & & \Psi_1, \Psi_2, \Psi_3. & &
\end{align*}
\end{rmk}

\begin{rmk} One must be careful to distinguish the $3$-form $\Gamma_1 = (\alpha_2 - i\alpha_3) \wedge (\kappa_2 + i\kappa_3)$ from the special isotropic $3$-form
\begin{align*}
\textstyle (r\partial_r\,\lrcorner\,\frac{1}{2}\sigma_1^2)|_M &  = (\alpha_2 + i\alpha_3) \wedge (\kappa_2 + i\kappa_3).
\end{align*}
While $\Gamma_1$ descends to $Z$, the special isotropic $3$-form $(r\partial_r\,\lrcorner\,\frac{1}{2}\sigma_1^2)|_M$ does not, because its exterior derivative has $\alpha_1$ terms.  Note that for $n = 1$, the object $(r\partial_r\,\lrcorner\,\frac{1}{2}\sigma_1^2)|_M = \Psi_1$ is a $3$-form on $M^7$ whose real part calibrates special Legendrian $3$-folds.
\end{rmk}

\section{Calibrated Geometry in Twistor Spaces} \label{sec:GalGeoTwistor}

\indent \indent We now turn to the submanifold theory of twistor spaces $Z$, organizing our discussion as follows.  In $\S$\ref{sub:Sp(n)U(1)}, we briefly discuss $\Sp(n)\U(1)$-geometry on arbitrary $(4n+2)$-manifolds $Y^{4n+2}$.  Then, in $\S$\ref{sub:GeomTwistor} (Theorem \ref{thm:ExtraStructureTwistor}), we prove that every twistor space $Z^{4n+2}$ admits a canonical $\Sp(n)\U(1)$-structure, which (among other data) entails a distinguished $3$-form $\gamma \in \Omega^3(Z; \C)$.  In Proposition \ref{prop:GammaSemiCal}, we prove that $\mathrm{Re}(\gamma) \in \Omega^3(Z)$ is a semi-calibration, and devote $\S$\ref{sec:Re(gamma)3folds} to the study of $\mathrm{Re}(\gamma)$-calibrated $3$-folds.  In a certain sense (Proposition \ref{prop:GammaSemiCal}(b)), these are higher-codimension generalizations of special Lagrangian $3$-folds in $6$-dimensional nearly-K\"{a}hler twistor spaces. \\
\indent Finally in $\S$\ref{sec:SubmanifoldsMZ}, we study the relationships between submanifolds of $M^{4n+3}$ and those in $Z^{4n+2}$.  More specifically, distinguishing the map $p_1 \colon M \to Z$, we consider how various submanifolds $\Sigma^k \subset Z$ behave under the operations of $p_1$-circle bundle lift $p_1^{-1}(\Sigma)^{k+1} \subset M$ and $p_1$-horizontal lift $\widehat{\Sigma}^k \subset M$.
\indent We remind the reader that as mentioned in the introduction, we only consider submanifolds of $Z$ that do not meet any orbifold points.

\subsection{$\Sp(n)\U(1)$-Structures} \label{sub:Sp(n)U(1)}

\indent \indent Let $Y^{4n+2}$ be a smooth $(4n+2)$-manifold with $n \geq 1$.

\begin{defn} A \emph{$(\U(2n) \times \U(1))$-structure} on $Y^{4n+2}$ is an almost-Hermitian structure $(g, J_+, \omega_+)$ together with a distribution of $J_+$-invariant $4n$-planes $\mathsf{H} \subset TY$.  Equivalently, it is an almost-Hermitian structure $(g, J_+, \omega_+)$ together with an orthogonal splitting
$$TY = \mathsf{H} \oplus \mathsf{V}$$
where $\mathsf{H} \subset TY$  and $\mathsf{V} \subset TY$  are $J_+$-invariant subbundles with $\mathrm{rank}(\mathsf{H}) = 4n$ and  $\mathrm{rank}(\mathsf{V}) = 2$.
\end{defn}

\indent Given a $(\U(2n) \times \U(1))$-structure $(g, J_+, \omega_+, \mathsf{H})$, we split $(g, J_+, \omega_+)$  into horizontal and vertical parts as follows:
\begin{align*}
g & = g_{\mathsf{H}} + g_{\mathsf{V}}, & J_+ & = \left.J_+\right|_{\mathsf{H}} + \left.J_+\right|_{\mathsf{V}}, & \omega_+ & = \omega_{\mathsf{H}} + \omega_{\mathsf{V}}.
\end{align*}
Further, we can extend it to a one-parameter family $(g(t), J_+, \omega_+(t), \mathsf{H})$ by defining
\begin{align*}
g(t) & = t^2g_{\mathsf{H}} + g_{\mathsf{V}}, & \omega_+(t) & = t^2\omega_{\mathsf{H}} + \omega_{\mathsf{V}}.
\end{align*}
Moreover, by reversing the orientation of the vertical subbundle $\mathsf{V} \subset TY$, we obtain a \emph{second} one-parameter family $(g(t)$, $J_-$, $\omega_-(t)$, $\mathsf{H})$ by defining
\begin{align*}
J_- & = \left.J_+\right|_{\mathsf{H}} - \left.J_+\right|_{\mathsf{V}}, & \omega_-(t) & =  t^2\omega_{\mathsf{H}} - \omega_{\mathsf{V}}.
\end{align*}
For calculations on $Y$, we will need local frames adapted to the geometry of the $(\U(2n) \times \U(1))$-structure.  To be precise:

\begin{defn}  A \emph{$(\U(2n) \times \U(1))$-coframe at $y \in Y$} is a $g$-orthonormal coframe
$$(\rho, \mu) := (\rho_{10}, \rho_{11}, \rho_{12}, \rho_{13}, \ldots, \rho_{n0}, \rho_{n1}, \rho_{n2}, \rho_{n3}, \mu_2, \mu_3) \colon T_yY \to \R^{4n} \times \R^{2}$$
for which
\begin{align*}
\left.\omega_{\mathsf{V}}\right|_y & = \mu_2 \wedge \mu_3, & \left.\omega_{\mathsf{H}}\right|_y & = \sum_{j=1}^n (\rho_{j0} \wedge \rho_{j1} + \rho_{j2} \wedge \rho_{j3}).
\end{align*}
\end{defn}

\indent For example, we will soon recall (Theorem \ref{thm:BasicStructureTwistor}) that every twistor space $Z^{4n+2}$ admits a natural $(\U(2n) \times \U(1))$-structure. In fact, we will show (Theorem \ref{thm:ExtraStructureTwistor}) that twistor spaces admit an additional piece of data:

\begin{defn} \label{def:SpnU1Str} Let $Y^{4n+2}$ be a $(4n+2)$-manifold with a $(\U(2n) \times \U(1))$-structure $(g, J_+, \omega_+, \mathsf{H})$.  A \emph{compatible $\Sp(n)\U(1)$-structure} is a complex $3$-form $\gamma \in \Omega^3(Y;\C)$ with the following property: At each $y \in Y$, there exists a $(\U(2n) \times \U(1))$-coframe $(\rho, \mu)$ such that
$$\left.\gamma\right|_y = (\mu_2 - i \mu_3) \wedge  \sum_{j=1}^n (\rho_{j0} + i \rho_{j1}) \wedge (\rho_{j2} + i\rho_{j3}).$$
Note that if $\gamma$ is a compatible $\Sp(n)\U(1)$-structure, then $\gamma$ has $J_+$-type $(2,1)$ and $J_-$-type $(3,0)$.
\end{defn}

\indent To justify this terminology, we make a digression into linear algebra. Consider the following $\Sp(n)\U(1)$-representation on $\R^{4n+2}$. For $(A,\lambda) \in \Sp(n) \times \U(1)$ and $(h,z) \in \HH^n \oplus \C$, define:
\begin{equation} \label{eq:SpnU1Rep}
(A,\lambda) \cdot (h,z) := (Ah \lambda^{-1}, \,\lambda^{-2}z).
\end{equation}
Identify $\HH^n \simeq \C^{2n}$ by writing $h = h_1 + jh_2$ with $h_1, h_2 \in \C^n$.  This identification endows $\HH^n$ with the complex structure given by right multiplication by $i$, which in turn yields an embedding $\iota \colon \Sp(n) \to \U(2n)$.  In this way, the representation (\ref{eq:SpnU1Rep}) induces an embedding
\begin{align*}
\Sp(n)\U(1) & \to \U(2n) \times \U(1) \\
(A,\lambda) & \mapsto (\iota(A)\lambda^{-1}, \, \lambda^{-2}).
\end{align*}
The image of this map is
\begin{equation} \label{eq:SpnU1Eqn}
\{ (B, \nu) \in \U(2n) \times \U(1) \colon \nu^{-1/2}B \in \Sp(n)\}.
\end{equation}
Since $\Sp(n)$ contains the element $-\text{Id}$, the condition $\nu^{-1/2} B \in \Sp(n)$ does not depend on the choice of square root. \\
\indent Let $(e_{10}, e_{11}, e_{12}, e_{13}, \ldots, e_{n0}, e_{n1}, e_{n2}, e_{n3}, f_2, f_3)$ denote the standard basis of $\mathbb{R}^{4n+2}$, and let $(e^{10}, e^{11}, e^{12}, e^{13}, \ldots, e^{n0}, e^{n1}, e^{n2}, e^{n3}, f^2, f^3)$ denote its dual basis.  We identify $\R^{4n+2} \simeq \C^{2n} \oplus \C$ via the complex structure $J_0$ whose K\"{a}hler form is
$$\omega_0 = f^2 \wedge f^3 + \sum (e^{j0} \wedge e^{j1} + e^{j2} \wedge e^{j3}).$$
Identifying $\C^{2n} \simeq \HH^n$, the standard hyperk\"{a}hler triple on $\HH^n$ is:
\begin{align} \label{eq:StdHKTriple}
\beta_1 & = \sum (e^{j0} \wedge e^{j1} + e^{j2} \wedge e^{j3}), & \beta_2 & = \sum (e^{j0} \wedge e^{j2} - e^{j1} \wedge e^{j3}), & \beta_3 & = \sum (e^{j0} \wedge e^{j3} + e^{j1} \wedge e^{j2}).
\end{align}
We consider the $3$-form $\gamma_0 \in \Lambda^3( (\R^{4n+2})^*)$ given by
$$\gamma_0 = (f^2 - if^3) \wedge (\beta_2 + i\beta_3).$$
Then:

\begin{prop} With respect to the standard $(\U(2n) \times \U(1))$-action on $\R^{4n+2}$, the stabilizer of $\gamma_0 \in \Lambda^3( (\R^{4n+2})^*)$ is the subgroup $\Sp(n)\U(1) \leq \U(2n) \times \U(1)$ given by (\ref{eq:SpnU1Eqn}).
\end{prop}

\begin{proof} Let $(B, \nu) \in \U(2n) \times \U(1)$, and set $\tau = f^2 - if^3$ and $\beta = \beta_2 + i\beta_3$.  Since $\tau$ has $J_0$-type $(0,1)$ and $\beta$ has $J_0$-type $(2,0)$, we have
\begin{align*}
\nu^*\tau & = \nu^{-1}\tau, & \nu^*\beta & = \nu^2\beta,
\end{align*}
and hence
\begin{align*}
(B,\nu)^*\gamma_0 = (B, \nu)^*(\tau \wedge \beta) & = \nu^*\tau \wedge B^*\beta = \tau \wedge (\nu^{-1/2}B)^*\beta.
\end{align*}
If $(B,\nu) \in \Sp(n)\U(1)$, then $\nu^{-1/2}B \in \Sp(n)$ by~\eqref{eq:SpnU1Eqn}. Thus, since $\Sp(n)$ stabilizes $\beta_1, \beta_2, \beta_3$, we get
$$ (B,\nu)^*\gamma_0 = \tau \wedge \beta = \gamma_0. $$
Conversely, if $(B,\nu) \in \U(2n) \times \U(1)$ stabilizes $\gamma_0$, then
$$\tau \wedge \beta = \tau \wedge (\nu^{-1/2}B)^*\beta.$$
Contracting both sides with the vector $f_2 + if_3$ implies that $\beta = (\nu^{-1/2}B)^*\beta$, so that $\nu^{-1/2}B \in \U(2n)$ stabilizes $\beta$.  Since the $\U(2n)$-stabilizer of $\beta$ is $\Sp(n)$, we deduce that $\nu^{-1/2}B \in \Sp(n)$, and hence $(B,\nu) \in \Sp(n)\U(1)$.
\end{proof}

\begin{example} \label{ex:InducedSU3} The case $n=1$ is particularly special.  Let $Y^6$ be a $6$-manifold with a $(\U(2) \times \U(1))$-structure $(g, J_+, \omega_+, \mathsf{H})$.  By definition, a compatible $\Sp(1)\U(1)$-structure is a complex $3$-form $\gamma \in \Omega^3(Y;\C)$ such that at each $y \in Y$, there exists a $(\U(2) \times \U(1))$-coframe $(\rho_0, \rho_1, \rho_2, \rho_3, \mu_2, \mu_3) \colon T_yY \to \R^4 \times \R^2$ for which
$$\gamma|_y = (\mu_2 - i\mu_3) \wedge (\rho_0 + i\rho_1) \wedge (\rho_2 + i\rho_3).$$
So, $\gamma$ is a non-vanishing $3$-form of $J_-$-type $(3,0)$ satisfying
$$-\frac{i}{8}\gamma \wedge \overline{\gamma} = \mu_2 \wedge \mu_3 \wedge \rho_0 \wedge \rho_1 \wedge \rho_2 \wedge \rho_3.$$
As such, $\gamma \in \Omega^3(Y;\C)$ defines an $\SU(3)$-structure on $Y$. \\
\indent Alternatively, the presence of a compatible $\SU(3)$-structure on $Y^6$ follows abstractly from the following group isomorphism of $\Sp(1)\U(1) \cong \U(2)$ onto a subgroup of $\SU(3)$. Using $\Sp(1) \cong \SU(2)$, we have:
\begin{align*}
\Sp(1)\U(1) = \left\{ \begin{pmatrix} B & 0 \\ 0 & \nu \end{pmatrix} \colon \nu^{-1/2}B \in \SU(2) \right\} & \cong \left\{ \begin{pmatrix} T & 0 \\ 0 & (\det T)^{-1} \end{pmatrix} \colon T \in \U(2) \right\} \leq \SU(3),  \\
\begin{pmatrix} B & 0 \\ 0 & \det B \end{pmatrix} & \mapsto \begin{pmatrix} B & 0 \\ 0 & (\det B)^{-1} \end{pmatrix}.
\end{align*}
(This is a group homomorphism because $\U(1)$ is abelian.) The situation is described by the following diagram:

\small $$\begin{tikzcd}
                                                       & \mathrm{U}(3) \arrow[ld, no head] \arrow[rd, no head] &                                    \\
\mathrm{U}(2) \times \mathrm{U}(1) \arrow[rd, no head] &                                                       & \mathrm{SU}(3) \arrow[ld, no head] \\
                                                       & \mathrm{Sp}(1)\mathrm{U}(1) \cong \mathrm{U}(2)       &                                   
\end{tikzcd}$$
\normalsize
\end{example}

\begin{rmk} The notation in this remark is that made standard in the monograph of Salamon~\cite{Salamon-book}. Let $T = \mathsf{H} \oplus \mathsf{V} \simeq \R^{4n+2}$ denote the (real) $\Sp(n)\U(1)$-representation of (\ref{eq:SpnU1Rep}).  Let $E \simeq \C^{2n}$ denote the standard complex $\Sp(n)$-representation, and let $L \simeq \C$ denote the standard complex $\U(1)$-representation.  Then, by refining the splitting $\Lambda^2(T^*) = \Lambda^2(\mathsf{H}^*) \otimes (\mathsf{H}^* \otimes \mathsf{V}^*) \otimes \Lambda^2(\mathsf{V}^*)$, one can decompose the space of real $2$-forms into $\Sp(n)\U(1)$-irreducible representations as follows:
\begin{align*}
\Lambda^2(\mathsf{H}^*) & \cong \R \omega_{\mathsf{H}} \oplus [\Sym^2(E)] \oplus [\Lambda^2_0(E)] \oplus \LB L^2 \RB \oplus \LB \Lambda^2_0(E) \otimes L^2 \RB \\
\mathsf{H}^* \otimes \mathsf{V}^* & \cong \LB E \otimes L^3 \RB \oplus \LB E \otimes L \RB \\
\Lambda^2(\mathsf{V}^*) & \cong \R \omega_{\mathsf{V}}.
\end{align*}
Alternatively, by refining the $J_+$-type splitting $\Lambda^2(T^*) = \LB\Lambda^{2,0}\RB \oplus [\Lambda^{1,1}]$, one obtains:
\begin{align*}
\LB \Lambda^{2,0} \RB & \cong \LB \Lambda^2_0(E) \otimes L^2 \RB \oplus \LB L^2 \RB \oplus \LB E \otimes L^3 \RB  \\
[\Lambda^{1,1}] & \cong \R\omega_{\mathsf{V}} \oplus \R \omega_{\mathsf{H}} \oplus [\Sym^2_0(E)] \oplus [\Lambda^2_0(E)] \oplus \LB E \otimes L \RB.
\end{align*}
\end{rmk}

\subsection{The Geometry of Twistor Spaces} \label{sub:GeomTwistor}

\indent \indent We now return to the study of twistor spaces $Z$.  The following fact is well-known:

\begin{thm} \label{thm:BasicStructureTwistor} Let $M^{4n+3}$ be a $3$-Sasakian manifold, and fix a projection $p = p_1 \colon M \to Z$.  The quotient $Z$ admits a $(\U(2n) \times \U(1))$-structure $(g, J_+, \omega_+, \mathsf{H})$ for which:
\begin{itemize}
\item $\left(g(1),\, J_+,\, \omega_+(1)\right)$ is K\"{a}hler-Einstein with positive scalar curvature.
\item $\left(g(\sqrt{2}),\, J_-,\, \omega_-(\sqrt{2})\right)$ is nearly-K\"{a}hler.
\item $p_*(\widetilde{\mathsf{H}}) = \mathsf{H}$ and $p_*(\mathrm{span}(A_2, A_3)) = \mathsf{V}$.
\end{itemize}
\end{thm}

\begin{proof} The K\"ahler-Einstein structure is very well-known and has been extensively studied. The statement about the nearly-K\"ahler structure is~\cite[Theorem 14.3.9]{Boyer-Galicki}. Details can be found in Alexandrov--Grantcharov--Ivanov~\cite{AGI} or Nagy~\cite{Nagy}.
\end{proof}

\indent From now on, the twistor space $Z$ will carry the $(\U(2n) \times \U(1))$-structure  $(g, J_+, \omega_+, \mathsf{H})$ described in the previous proposition.  We will write
\begin{align*}
(g_{\mathrm{KE}}, J_{\mathrm{KE}}, \omega_{\mathrm{KE}}) & := \left(g(1),\, J_+,\, \omega_+(1)\right), \\
(g_{\mathrm{NK}}, J_{\mathrm{NK}}, \omega_{\mathrm{NK}}) & := (g(\sqrt{2}),\, J_-,\, \omega_-(\sqrt{2})).
\end{align*}
In particular,
\begin{align} \label{eq:omegaKE-omegaNK}
\omega_{\mathrm{KE}} & = \omega_{\mathsf{H}} + \omega_{\mathsf{V}}, & \omega_{\mathrm{NK}} & = 2\omega_{\mathsf{H}} - \omega_{\mathsf{V}}.
\end{align}

\noindent We now recover the important observation of Alexandrov~\cite{Alexandrov} that $Z$ naturally admits \emph{even more} structure:

\begin{thm} \label{thm:ExtraStructureTwistor} Let $Z$ be a twistor space with its $(\U(2n) \times \U(1))$-structure $(g_{\mathrm{KE}},$ $J_{\mathrm{KE}}$, $\omega_{\mathrm{KE}}$, $\mathsf{H})$.  Then $Z$ naturally admits a compatible $\Sp(n)\U(1)$-structure $\gamma \in \Omega^3(Z; \C)$.
\end{thm}

\begin{proof} By Proposition \ref{prop:Descent}(b), there exists a unique $3$-form $\gamma \in \Omega^3(Z;\C)$ satisfying
$$p^*(\gamma) = \Gamma_1 = (\alpha_2 - i\alpha_3) \wedge (\kappa_2 + i\kappa_3).$$
This $3$-form is an $\Sp(n)\U(1)$-structure.
\end{proof}

\indent In $\S$\ref{subsec:TwistorSpace}, we will give a second proof of Theorem \ref{thm:ExtraStructureTwistor}  from the perspective of quaternionic-K\"{a}hler geometry.  For now, using Proposition \ref{prop:Descent}, we can compute the following exterior derivatives:
\begin{align*}
d\omega_{\mathsf{V}} & = -\mathrm{Im}(2\gamma), & d\omega_{\mathrm{KE}} & = 0, & d\,\mathrm{Re}(\gamma) & = 2\xi - 4\omega_{\mathsf{H}} \wedge \omega_{\mathsf{V}}, \\
d\omega_{\mathsf{H}} & = \mathrm{Im}(2\gamma), & d\omega_{\mathrm{NK}} & = 3\,\mathrm{Im}(2\gamma), & d\,\mathrm{Im}(\gamma) & = 0, \\
 & & & & d\xi & = -4\omega_{\mathsf{H}} \wedge \mathrm{Im}(\gamma).
\end{align*}

\begin{example} \label{ex:NKEqns} When $n = 1$, there is a coincidence $\xi = 2\omega_{\mathsf{H}}^2$.  Therefore, in this case, using that $\omega_{\mathrm{NK}}^2 = (2\omega_{\mathsf{H}} - \omega_{\mathsf{V}})^2 = 4\omega_{\mathsf{H}}^2 - 4 \omega_{\mathsf{H}} \wedge \omega_{\mathsf{V}} = 2\xi - 4\omega_{\mathsf{H}} \wedge \omega_{\mathsf{V}}$, we recover the equations
\begin{align*}
d\omega_{\mathrm{NK}} & = 3\,\mathrm{Im}(2\gamma), \\
d\,\mathrm{Re}(2\gamma) & = 2\, \omega_{\mathrm{NK}} \wedge \omega_{\mathrm{NK}},
\end{align*}
familiar from the theory of nearly-K\"{a}hler $6$-manifolds.
\end{example}

\subsection{$\mathrm{Re}(\gamma)$-Calibrated $3$-folds} \label{sec:Re(gamma)3folds}

\indent \indent Let $Z^{4n+2}$ be a twistor space equipped with the $(\U(2n) \times \U(1))$-structure $(g_{\mathrm{KE}}, J_{\mathrm{KE}}, \omega_{\mathrm{KE}}, \mathsf{H})$.  With respect to this structure, one can consider several classes of submanifolds of $Z$, such as:
\begin{itemize}
\item $J_{\mathrm{KE}}$-complex (resp., $J_{\mathrm{NK}}$-complex) submanifolds.
\item Horizontal submanifolds (i.e., those tangent to $\mathsf{H}$).
\item $\omega_{\mathrm{KE}}$-isotropic (resp., $\omega_{\mathrm{NK}}$-isotropic) submanifolds.
\end{itemize}
These submanifolds have been the subject of numerous studies, particularly when $\dim(Z) = 6$.  However, since we have now shown that $Z$ admits a compatible $\Sp(n)\U(1)$-structure $\gamma \in \Omega^3(Z; \C)$, twistor spaces also admit a distinguished class of $3$-folds.  In this section we explore these. \\
\indent We begin by showing that $\mathrm{Re}(\gamma) \in \Omega^3(Z)$ is a semi-calibration, for which we need a preliminary lemma.

\begin{lem} \label{lem:v-hook-Regamma}
For any horizontal unit vector $v \in \mathsf{H}$, the $2$-form $\iota_v(\mathrm{Re}(\gamma)) \in \Omega^2(Z)$ is a semi-calibration.  Moreover, its calibrated $2$-planes lie in the $6$-plane $L \oplus \mathsf{V}$, where $L$ is the quaternionic line spanned by $v$.
\end{lem}

\begin{proof} It suffices to work at a fixed point $z \in Z$.  Let $(\rho,\mu)$ be a $\Sp(n)\U(1)$-coframe at $z$ as in Definition \ref{def:SpnU1Str}.  We may then write $\gamma|_z = \tau \wedge (\beta_2 + i\beta_3)$, where
\begin{align*}
\tau & = \mu_2 - i\mu_3, & \beta_2 & = \sum_{i=1}^n (\rho_{j0} \wedge \rho_{j2} - \rho_{j1} \wedge \rho_{j3}), & \beta_3 & = \sum_{i=1}^n (\rho_{j0} \wedge \rho_{j3} + \rho_{j1} \wedge \rho_{j2}).
\end{align*}
Define complex structures $J_2$ and $J_3$ on $\mathsf{H}|_z$ by declaring
\begin{align*}
J_2(\rho_{j0}) & = \rho_{j2}, & J_3(\rho_{j0}) & = \rho_{j3}, \\
J_2(\rho_{j1}) & = -\rho_{j3}, & J_3(\rho_{j1}) & = \rho_{j2},
\end{align*}
which implies $J_+J_2 = J_3$ and $g(J_2 \cdot, \cdot) = \beta_2$ and $g(J_3 \cdot, \cdot) = \beta_3$.  Note that $\tau$, $\beta_2, \beta_3$, as well as $J_2, J_3$, depend on the choice of $\Sp(n)\U(1)$-frame. \\
\indent Now, let $v \in \mathsf{H}$ be a horizontal unit vector.  Let $w = J_2v$, so that
\begin{align*}
\iota_v\gamma = \iota_v(\tau \wedge (\beta_2 + i\beta_3)) = \tau \wedge \iota_v(\beta_2 + i\beta_3) & = \tau \wedge (g(J_2v, \cdot) + ig(J_+J_2v, \cdot)) \\
& = \tau \wedge (w^\flat - iJ_+ w^\flat) \\
& = (\mu_2 - i \mu_3) \wedge (w^\flat - iJ_- w^\flat)
\end{align*}
since $J_+ = J_-$ on horizontal vectors. This $2$-form is decomposable and has $J_-$-type $(2,0)$. Moreover, $\{\mu_2, \mu_3, w^\flat, J_- w^\flat\}$ is an orthonormal set.  Thus, this $2$-form is a standard complex volume form, and hence its real part is a semi-calibration.
\end{proof}

\begin{rmk} \label{rmk:v-hook-Regamma}
The above proof shows slightly more, namely that the $\iota_v (\mathrm{Re}(\gamma))$-calibrated $2$-planes lie in the $4$-plane $\text{span}(w, J_+w) \oplus \mathsf{V} = \text{span}(J_2 v, J_3 v) \oplus \mathsf{V}$.
\end{rmk}

\begin{prop} \label{prop:GammaSemiCal} The $3$-form $\mathrm{Re}(\gamma) \in \Omega^3(Z)$ is a semi-calibration.  Moreover, let $E \in \Gr_3^+(TZ)$ be an oriented $3$-plane.
\begin{enumerate}[(a)]
\item $E$ is $\mathrm{Re}(\gamma)$-calibrated if and only if $E = \R v \oplus E'$ for some $v \in E \cap \mathsf{H}$ and some $2$-plane $E'$ that is $\iota_v(\mathrm{Re}(\gamma))$-calibrated.
\item  If $E$ is a $\mathrm{Re}(\gamma)$-calibrated $3$-plane, there is a quaternionic line $L \subset \mathsf{H}$ such that $E$ is contained in $L \oplus \mathsf{V}$. 
\item If $E$ is $\mathrm{Re}(\gamma)$-calibrated, then $E$ is $\omega_{\mathrm{NK}}$-isotropic.
\end{enumerate}
\end{prop}

\begin{proof} If $E \in \Gr_3^+(T_zZ)$ is an oriented $3$-plane at $z \in Z$, then $\dim(E \cap \mathsf{H}) \geq 1$, so there exists a unit vector $v \in E \cap \mathsf{H}$, and we may orthogonally split $E = \R v \oplus E'$.  Then
$$(\mathrm{Re} (\gamma)) (E) = ( \iota_v \mathrm{Re}(\gamma))(E') \leq 1$$
by Lemma~\ref{lem:v-hook-Regamma}, so the comass of $\mathrm{Re} (\gamma)$ is at most $1$.  Now, let $v$ be a horizontal unit vector and let $E'$ be an $\iota_v (\mathrm{Re} (\gamma))$-calibrated $2$-plane, which exists by Lemma~\ref{lem:v-hook-Regamma}.  Then $E = \R v \oplus E'$ is $\mathrm{Re}(\gamma)$-calibrated, which shows that $\mathrm{Re}(\gamma)$ has comass equal to one.  Further, we have seen that an oriented $3$-plane $E$ is $\mathrm{Re}(\gamma)$-calibrated if and only if $E'$ is $\iota_v(\mathrm{Re}(\gamma))$-calibrated, which proves (a). \\
\indent Part (b) follows from Remark~\ref{rmk:v-hook-Regamma}. Finally, since $\gamma$ is of $J_-$-type $(3,0)$, part (c) follows from Proposition \ref{prop:IsotropyLemma}.
\end{proof}

Returning to the $3$-Sasakian manifold $M^{4n+3}$, we can now establish the following:
\begin{cor} \label{cor:ReGammaComass} The $3$-form $\mathrm{Re}(\Gamma_1) \in \Omega^3(M)$ is a semi-calibration.
\end{cor}

\begin{proof} Recall that $p_1 \colon M \to Z$ is a Riemannian submersion, that $\mathrm{Re}(\Gamma_1) = p_1^*(\mathrm{Re}(\gamma))$, and that $\mathrm{Re}(\gamma) \in \Omega^3(Z)$ has comass one. The result now follows from Proposition~\ref{prop:submersion-calibration}.
\end{proof}

\begin{rmk}
We pause to make two remarks.  First, Proposition \ref{prop:GammaSemiCal} shows that $\mathrm{Re}(\gamma)$-calibrated $3$-folds $L^3 \subset Z^{4n+2}$ are $\omega_{\mathrm{NK}}$-isotropic.  However, we emphasize that such $3$-folds need not be $\omega_{\mathrm{KE}}$-isotropic in general.  Later (Theorem \ref{thm:ReGammaCircleBundle}), we will characterize the $\mathrm{Re}(\gamma)$-calibrated $3$-folds $L \subset Z$ satisfying $\left.\omega_{\mathrm{KE}}\right|_L = 0$. \\
\indent Second, we clarify that Proposition \ref{prop:GammaSemiCal} asserts $\mathrm{Re}(\gamma)$ is a semi-calibration with respect to the metric $g_{\mathrm{KE}}$.  Therefore, by Proposition~\ref{prop:scaling-calib}, the $3$-form $\mathrm{Re}(t^2\gamma)$ is a semi-calibration with respect to the metric $g(t) = t^2 g_{\mathsf{H}} + g_{\mathsf{V}}$.  In particular, $\mathrm{Re}(2\gamma)$ is a semi-calibration with respect to $g_{\mathrm{NK}} = 2 g_{\mathsf{H}} + g_{\mathsf{V}}$.
\end{rmk}

\subsubsection{A Normal Form for $\mathrm{Re}(\gamma)$-Calibrated $3$-planes}

\indent \indent We now aim to establish a normal form for $\mathrm{Re}(\gamma)$-calibrated $3$-planes in $Z$.  Since the subsequent discussion is a matter of linear algebra, we work in $\R^{4n+2} \simeq \HH^n \oplus \C$.  As we have done previously, we let
$$(e_{10}, e_{11}, e_{12}, e_{13}, \ldots, e_{n0}, e_{n1}, e_{n2}, e_{n3}, f_2, f_3)$$
denote the standard basis of $\R^{4n+2}$, let $(e^{10}, e^{11}, \ldots, f^2, f^3)$ denote its dual basis, let $\beta_1, \beta_2, \beta_3$ be the standard hyperk\"{a}hler triple on $\HH^n$ as in (\ref{eq:StdHKTriple}), and consider the $3$-form $\gamma_0 \in \Lambda^3((\R^{4n+2})^*)$ given by
$$\gamma_0 = (f^2 - if^3) \wedge (\beta_2 + i\beta_3).$$
Now, for $e^{i \theta} \in S^1$, define the $2$-plane
\begin{align} \label{eq:Vtheta}
V_\theta & = \mathrm{span}\left( c_\theta(-f_2 - e_{13}) + s_\theta(- f_3 - e_{12}), s_\theta(-f_2 + e_{13}) + c_\theta (- f_3 + e_{12}) \right)\!.
\end{align}
In particular, we highlight
\begin{align} \label{eq:Vpi4}
V_{\frac{\pi}{4}} & = \mathrm{span}\left( f_2 + f_3 + e_{12} + e_{13}, f_2 + f_3 - e_{12} - e_{13} \right).
\end{align}

\begin{prop} \label{prop:GammaNormalForm} Consider the $\Sp(n)\U(1)$-action on $\HH^n \oplus \C$ given in (\ref{eq:SpnU1Rep}). Let $E \subset \HH^n \oplus \C$ be a $\mathrm{Re}(\gamma_0)$-calibrated $3$-plane. Then there exists $(A, \lambda) \in \Sp(n)\U(1)$ and a unique $\theta \in [0, \frac{\pi}{4}]$ such that $(A, \lambda) \cdot E = \R e_{10} \oplus V_\theta$. Moreover, the following are equivalent:
\begin{enumerate}[(a)]
\item $\dim(E \cap \HH^n) = 2$.
\item $E = (E \cap \HH^n) \oplus (E \cap \C)$.
\item $E$ is $\omega_{\mathrm{KE}}$-isotropic.
\item $\theta = \frac{\pi}{4}$.
\end{enumerate}
\end{prop}

\begin{proof}  Let $E \subset \HH^n \oplus \C$ be a $\mathrm{Re}(\gamma_0)$-calibrated $3$-plane. By Proposition~\ref{prop:GammaSemiCal}, there exists a quaternionic line $L \subset \HH^n$ for which $E \subset L \oplus \C$.  Since the subgroup $\Sp(n) \leq \Sp(n) \U(1)$ acts transitively on the quaternionic lines of $\HH^n$, there exists $A_0 \in \Sp(n)$ such that $A_0 \cdot L = L_0$, where $L_0$ is the standard quaternionic line
$$L_0 = \mathrm{span}(e_{10}, e_{11}, e_{12}, e_{13}).$$
Thus $(A_0, 1) \cdot E \subset L_0 \oplus \C$, so we can without loss of generality suppose that $E \subset L_0 \oplus \C$. \\
\indent Now, $L_0 \oplus \C$ is a complex $3$-plane, and the restriction of $\gamma_0$ to $L_0 \oplus \C$ is a complex volume form.  Thus, the problem reduces to finding a normal form for special Lagrangian $3$-planes in a complex $3$-space with respect to the action of $\Sp(1)\U(1) \cong \U(2)$.  Such a normal form was established in~\cite[Proposition 3.2]{Aslan}. (Translating between notations, the $b_1, i b_1, b_2, i b_2, b_3, i b_3$ of~\cite{Aslan} corresponds to our $e_{10}, e_{11}, e_{12}, e_{13}, f_2, - f_3$.) \\
\indent For $\theta \in [0, \frac{\pi}{4}]$, write $W_\theta = \R e_{10} \oplus V_\theta$. We observe that the conditions (a), (b), (c) above are invariant under the action of $\Sp(n) \U(1)$, so it is enough to verify that for $W_{\theta}$ they are equivalent to $\theta = \frac{\pi}{4}$. If $\theta = \frac{\pi}{4}$, we have
$$ W_{\frac{\pi}{4}} = \mathrm{span}(e_{10},\, e_{12} + e_{13}, \, f_2 + f_3) = (W_{\frac{\pi}{4}} \cap \HH^n) \oplus (W_{\frac{\pi}{4}} \cap \C), $$
so both (a) and (b) hold. If $\theta \neq \frac{\pi}{4}$, then one can compute from~\eqref{eq:Vtheta} that $\dim(W_\theta \cap \HH^n) = 1$.  Since a $\mathrm{Re}(\gamma_0)$-calibrated $3$-plane cannot contain any complex lines, we have $\dim(W_\theta \cap \C) < 2$, and hence
$$ \dim( (W_\theta \cap \HH^n) \oplus (W_\theta \cap \C)) = \dim(W_\theta \cap \HH^n) + \dim(W_\theta \cap \C) < 3 = \dim(W_\theta), $$
so both (a) and (b) do not hold. \\
\indent With respect to the above basis, we have $\omega_{\mathrm{KE}} = \beta_1 + f^2 \wedge f^3$. Letting
$$ v_2 = c_\theta(-f_2 - e_{13}) + s_\theta(- f_3 - e_{12}), \quad v_3 = s_\theta(-f_2 + e_{13}) + c_\theta (- f_3 + e_{12}), $$
so $V_\theta = \mathrm{span}(v_2, v_3)$ and $W_{\theta} = \R e_{10} \oplus V_{\theta}$, a computation shows that
$$ \omega_{\mathrm{KE}}(e_{10}, v_2) = \omega_{\mathrm{KE}}(e_{10}, v_3) = 0, \qquad \omega_{\mathrm{KE}}(v_2, v_3) = 2 (c_\theta^2 - s_\theta^2), $$
so $\left.\omega_{\mathrm{KE}}\right|_{W_{\theta}} = 0$ if and only if $\theta = \frac{\pi}{4}$.
\end{proof}

\subsubsection{HV Compatibility}

\begin{defn} A submanifold $\Sigma^k \subset Z^{4n+2}$ is called \emph{HV-compatible} if at each $x \in \Sigma$, we have
$$T_x\Sigma = (T_x\Sigma \cap \mathsf{H}) \oplus (T_x\Sigma \cap \mathsf{V}).$$
\end{defn}

\indent HV compatibility is a rather stringent condition.  Nevertheless, we now observe that certain natural classes of submanifolds of $Z$ automatically satisfy it.

\begin{prop} \label{prop:HV-KE-Iso} Let $\Sigma^k \subset Z^{4n+2}$ be a submanifold, $1 \leq k \leq 2n+1$.
\begin{enumerate}[(a)]
\item If $\Sigma$ is HV-compatible, then $\Sigma$ is $\omega_{\mathrm{KE}}$-isotropic if and only if $\Sigma$ is $\omega_{\mathrm{NK}}$-isotropic.
\item Suppose $\dim(\Sigma) = 2n+1$.  If $\Sigma$ is $\omega_{\mathrm{KE}}$-Lagrangian and $\omega_{\mathrm{NK}}$-Lagrangian, then $\Sigma$ is HV-compatible.  Moreover, $\dim(T_z\Sigma \cap \mathsf{H}) = 2n$ and $\dim(T_z\Sigma \cap \mathsf{V}) = 1$ at each $z \in \Sigma$.
\item Suppose $\dim(\Sigma) = 3$.  If $\Sigma$ is $\mathrm{Re}(\gamma)$-calibrated, then $\Sigma$ is HV-compatible if and only if $\Sigma$ is $\omega_{\mathrm{KE}}$-isotropic.  In this case, $\dim(T_z\Sigma \cap \mathsf{H}) = 2$ and $\dim(T_z\Sigma \cap \mathsf{V}) = 1$ at each $z \in \Sigma$.
\end{enumerate}
\end{prop}

\begin{proof} (a) Suppose $\Sigma$ is HV-compatible. If $\Sigma$ is $\omega_{\mathrm{KE}}$-isotropic, then~\eqref{eq:omegaKE-omegaNK} says that
\begin{equation} \label{eq:HV-isotropic-temp}
\left.\omega_{\mathsf{V}}\right|_{\Sigma} = - \left.\omega_{\mathsf{H}}\right|_{\Sigma}.
\end{equation}
We claim that $\omega_{\mathsf{H}}|_{\Sigma} = \omega_{\mathsf{V}}|_{\Sigma} = 0$, which would imply again by~\eqref{eq:omegaKE-omegaNK} that $\Sigma$ is also $\omega_{\mathrm{NK}}$-isotropic. Let $u_1, u_2 \in T_x \Sigma$, and decompose them orthogonally as $u_j = u_j^{\mathsf{H}} + u_j^{\mathsf{V}}$, where $u_j^{\mathsf{H}} \in \mathsf{H}$ and $u_j^{\mathsf{V}} \in \mathsf{V}$. Since $\Sigma$ is HV-compatible, both $u_j^{\mathsf{H}}$ and $u_j^{\mathsf{V}}$ are in $T_x \Sigma$ for $j = 1, 2$. Using~\eqref{eq:HV-isotropic-temp} and the facts that $\omega_{\mathsf{H}} \in \Lambda^2 (\mathsf{H}^*)$ and $\omega_{\mathsf{V}} \in \Lambda^2 (\mathsf{V}^*)$, we have
\begin{align*}
\omega_{\mathsf{V}}(u_1, u_2) & = \omega_{\mathsf{V}}(u_1^{\mathsf{H}} + u_1^{\mathsf{V}}, u_2^{\mathsf{H}} + u_2^{\mathsf{V}}) = \omega_{\mathsf{V}}(u_1^{\mathsf{V}}, u_2^{\mathsf{V}}) \\
& = -\omega_{\mathsf{H}}(u_1^{\mathsf{V}}, u_2^{\mathsf{V}}) = 0.
\end{align*}
The argument in the other direction is essentially the same, with~\eqref{eq:HV-isotropic-temp} replaced by $\left.\omega_{\mathsf{V}}\right|_{\Sigma} = \left.2 \omega_{\mathsf{H}}\right|_{\Sigma}$.
\\

\indent (b) Let $\Sigma^{2n+1} \subset Z$ be $\omega_{\mathrm{KE}}$-Lagrangian and $\omega_{\mathrm{NK}}$-Lagrangian, so that $\left.\omega_{\mathsf{V}}\right|_\Sigma = 0$ and $\left.\omega_{\mathsf{H}}\right|_\Sigma = 0$.  Fix $z \in \Sigma$, let  $\pi_{\mathsf{H}} \colon T_zZ \to \mathsf{H}$ and $\pi_{\mathsf{V}} \colon T_zZ \to \mathsf{V}$ denote the projection maps, so that
$$T_z\Sigma \subset \pi_{\mathsf{H}}(T_z\Sigma) \oplus \pi_{\mathsf{V}}(T_z\Sigma).$$
Let $(\rho, \mu) \colon T_zZ \to \R^{4n + 2}$ be an $\Sp(n)\U(1)$-coframe at $z$.  Since $\left.\mu^2 \wedge \mu^3\right|_\Sigma = \left.\omega_{\mathsf{V}}\right|_\Sigma = 0$, we have $\left.\mu^2 \wedge \mu^3\right|_{\pi_{\mathsf{V}}(T_z\Sigma)} = 0$, so that $\dim( \pi_{\mathsf{V}}(T_z\Sigma)) \leq 1$.  Moreover, since $\omega_{\mathsf{H}}$ is a non-degenerate $2$-form on the $4n$-plane $\mathsf{H}$, the condition $\omega_{\mathsf{H}}|_{ \pi_{\mathsf{H}}(T_z\Sigma) } = 0$ implies that $\dim(\pi_{\mathsf{H}}(T_z\Sigma)) \leq 2n$.  Therefore, since
$$\dim(\pi_{\mathsf{H}}(T_z\Sigma) \oplus \pi_{\mathsf{V}}(T_z\Sigma)) = \dim(\pi_{\mathsf{H}}(T_z\Sigma)) + \dim(\pi_{\mathsf{V}}(T_z\Sigma)) \leq 2n+1 = \dim(T_z\Sigma),$$
we deduce that $T_z\Sigma = \pi_{\mathsf{H}}(T_z\Sigma) \oplus \pi_{\mathsf{V}}(T_z\Sigma)$, which implies the result. \\

\indent (c) This is immediate from Proposition~\ref{prop:GammaNormalForm}.
\end{proof}

\subsubsection{Other phases}

\indent \indent Thus far, we have studied the real $3$-form $\mathrm{Re}(\gamma) \in \Omega^3(Z)$.  More generally, one can consider the $S^1$-family of real $3$-forms $\mathrm{Re}(e^{-i\theta}\gamma)$ for constant $e^{i\theta} \in S^1$.  We now explore the corresponding submanifold theory, beginning with a familiar situation:

\begin{example} \label{ex:NK6} Suppose that $n = 1$, so that the twistor space $Z$ is $6$-dimensional, and $\gamma \in \Omega^3(Z;\C)$ is an $\Sp(1)\U(1) = \U(2)$-structure.  By the discussion in Examples \ref{ex:InducedSU3} and \ref{ex:NKEqns}, the $3$-form $\gamma$ induces an $\SU(3)$-structure on $Z^6$  and satisfies
\begin{align*}
d\omega_{\mathrm{NK}} & = 3\,\mathrm{Im}(2\gamma), \\
d\,\mathrm{Re}(2\gamma) & = 2\, \omega_{\mathrm{NK}} \wedge \omega_{\mathrm{NK}}.
\end{align*}
Now, let $L^3 \subset Z^6$ be an oriented $3$-dimensional submanifold.  It is well-known that $L$ is $\omega_{\mathrm{NK}}$-Lagrangian if and only if $L$ is $\gamma$-special Lagrangian of phase $1$.  That is:
\begin{align*}
\mathrm{Re}(2\gamma)|_L = \vol_L \ \ \iff \ \ \mathrm{Im}(2\gamma)|_L = 0 \ \text{ and } \ \left.\omega_{\mathrm{NK}}\right|_L = 0 \ \ \ \iff \ \ \ L \text{ is }\omega_{\mathrm{NK}}\text{-Lagrangian}.
\end{align*}
More generally, one might wish to consider $\gamma$-special Lagrangian $3$-folds of other phases $e^{i\theta} \in S^1$. However, it is well-known that if $L^3 \subset Z^6$ satisfies $\mathrm{Re}(e^{-i\theta}\gamma)|_L = \vol_L$, then $e^{-i\theta} = \pm 1$.
\end{example}

\indent Example~\ref{ex:NK6} is the special case $n=1$ of the following more general statement, which is new:

\begin{prop} \label{prop:Re-gamma-phases} Let $L^3 \subset Z^{4n+2}$ be a $3$-dimensional submanifold.
\begin{enumerate}[(a)]
\item If $L$ is $\mathrm{Re}(e^{-i\theta}\gamma)$-calibrated, then $e^{i\theta} = \pm 1$.
\item If $L$ is $\mathrm{Re}(\gamma)$-calibrated, then $\omega_{\mathrm{NK}}|_L = 0$ and $\mathrm{Im}(\gamma)|_L = 0$.  If $n = 1$, then the converse also holds.
\end{enumerate}
\end{prop}

\begin{proof} Suppose that $L \subset Z^{4n+2}$ is $\mathrm{Re}(e^{-i\theta}\gamma)$-calibrated.  By the same argument as in Proposition \ref{prop:GammaSemiCal}, we have $\left.\omega_{\mathrm{NK}}\right|_L = 0$.  Since $d\omega_{\mathrm{NK}} = 6\,\mathrm{Im}(\gamma)$, it follows that $\left.\mathrm{Im}(\gamma)\right|_L = 0$.  Therefore,
$$ \pm \vol_L = \left.\mathrm{Re}(e^{-i\theta}\gamma)\right|_L = \left.\cos(\theta)\,\mathrm{Re}(\gamma)\right|_L.$$
Since $\mathrm{Re}(\gamma)$ has comass one, it follows that $\cos(\theta) = \pm 1$. (The converse of (b) when $n=1$ is the well-known result discussed in Example~\ref{ex:NK6}.)
\end{proof}

\subsection{Relations between Submanifolds in $M$ and $Z$} \label{sec:SubmanifoldsMZ}

\indent \indent We now systematically discuss the relationships between the various classes of submanifolds in $Z^{4n+2}$ and those in $M^{4n+3}$.  Broadly speaking, given a submanifold $\Sigma \subset Z$, there are two natural ways to construct a corresponding submanifold of $M$.  The first is to consider the circle bundle $p_1^{-1}(\Sigma) \subset M$, and the second is to consider its $p_1$-horizontal lift $\widehat{\Sigma} \subset M$ (provided it exists).  We will examine both constructions.

\subsubsection{Circle Bundle Constructions}

\indent \indent We begin by considering submanifolds of the form $p_1^{-1}(\Sigma) \subset M$ for some submanifold $\Sigma \subset Z$.  First, we consider those that are $I_1$-CR.  In general, Proposition \ref{prop:AssocCases}(a) shows that every $I_1$-CR $3$-fold of $M$ is $\phi_2$-associative.  For circle bundles, the converse also holds:

\begin{prop} \label{prop:CRCircleBundle} Let $\Sigma^{2k} \subset Z^{4n+2}$ be a submanifold, $2 \leq 2k \leq 4n$. Then $\Sigma$ is $J_+$-complex if and only if $p_1^{-1}(\Sigma)$ is $I_1$-CR.  Moreover, in the case of $2k = 2$, these conditions are also equivalent to: $p_1^{-1}(\Sigma)$ is $\phi_2$-associative.
\end{prop}

\begin{proof} Let $\Sigma \subset Z$ be a submanifold, and set $L = p_1^{-1}(\Sigma) \subset M$.  Fix $x \in L$ and let $z = p_1(x) \in \Sigma$.  Note that:
\begin{align*}
\Sigma \text{ is }J_+\text{-complex} & \iff \left.(\omega_{\mathrm{KE}})^k\right|_\Sigma = k!\,\vol_\Sigma, \\
L \text{ is } I_1\text{-CR} & \iff \left.(\alpha_1 \wedge \Omega_1^k)\right|_L = k!\, \vol_L.
\end{align*}

\indent Since $A_1 \in T_xL$, we can write $T_xL = \R A_1 \oplus \widetilde{U}$ for some subspace $\widetilde{U} \subset \Ker(\alpha_1)$.  Let $\{\widetilde{u}_1, \ldots, \widetilde{u}_{2k-1}\}$ be an orthonormal basis of $\widetilde{U}$ such that $\{A_1, \widetilde{u}_1, \ldots, \widetilde{u}_{2k-1}\}$ is an oriented orthonormal basis of $T_xL$. Setting $u_j = (p_1)_*(\widetilde{u}_j)$, and noting that
$$p_1|_{\text{Ker}(\alpha_1)} \colon \Ker(\alpha_1)|_x \to T_zZ$$
is an isometry, we see that $\{u_1, \ldots, u_{2k-1}\}$ is an orthonormal basis of $T_z\Sigma$. Therefore, recalling that $\Omega_1 = p_1^*(\omega_{\mathrm{KE}})$, we have:
\begin{align*}
L \text{ is }I_1\text{-CR } \iff (\alpha_1 \wedge \Omega_1^k)(A_1, \widetilde{u}_1, \ldots, \widetilde{u}_{2k-1}) = k! & \iff \Omega_1^k(\widetilde{u}_1, \ldots, \widetilde{u}_{2k-1}) = k! \\
& \iff \omega_{\mathrm{KE}}^k(u_1, \ldots, u_{2k-1}) = k! \\ 
& \iff \Sigma \text{ is }J_+\text{-complex.}
\end{align*}

\indent Now suppose $k = 1$. Observe that
\begin{align*}
\phi_2 & = \alpha_1 \wedge \Omega_1 - \alpha_2 \wedge \Omega_2 + \alpha_3 \wedge \Omega_3 \\
& = \alpha_1 \wedge \Omega_1 - \alpha_2 \wedge \kappa_2 + \alpha_3 \wedge \kappa_3.
\end{align*}
Since $\iota_{A_1}(-\alpha_2 \wedge \kappa_2 + \alpha_3 \wedge \kappa_3) = 0$, we have $\left.(-\alpha_2 \wedge \kappa_2 + \alpha_3 \wedge \kappa_3)\right|_L = 0$.  Therefore, we see that $\left.\phi_2\right|_L = \left.(\alpha_1 \wedge \Omega_1)\right|_L$, which gives the result.
\end{proof}

\indent The previous proposition shows that a circle bundle $p_1^{-1}(\Sigma)$ is $I_1$-CR if and only if $\Sigma$ is $J_+$-complex.  In fact, any $I_1$-CR submanifold is \emph{locally} a circle bundle:

\begin{prop} \label{prop:CRImpliesBundle} Let $L^{2k+1} \subset M^{4n+3}$ be a submanifold, $2 \leq 2k \leq 4n$.  Then $L$ is $I_1$-CR if and only if $L$ is locally of the form $p_1^{-1}(\Sigma)$ for some $J_+$-complex submanifold $\Sigma^{2k} \subset Z^{4n+2}$.
\end{prop}

\begin{proof} ($\Longleftarrow$) This follows from Proposition \ref{prop:CRCircleBundle}. \\
\indent ($\Longrightarrow$) Let $L \subset M$ be $I_1$-CR, and abbreviate $p := p_1$.  At each $x \in L$, we have $\left.A_1\right|_x \in T_xL$, so (short-time) integral curves of $A_1$ lie in $L$.  That is, at each $x \in L$, there exists an open set $I_x \subset p^{-1}(p(x))$ such that $x \in I_x$ and $I_x \subset L$. \\
\indent We claim that $p(L) \subset Z$ is an embedded $2k$-dimensional submanifold of $Z$.  To see this, fix $z \in p(L)$, and let $x \in L$ have $p(x) = z$.  Letting $\ell$ satisfy $(2k+1) + \ell = 4n+3$, we choose a neighborhood $W \subset M$ of $x$ and a chart $\widetilde{\phi} \colon W \to \R^{4n+3} = \R^{2k} \times \R \times \R^{\ell}$ with coordinate functions denoted $\widetilde{\phi} = (t^1, \ldots, t^{2k}, u, v^1, \ldots, v^\ell)$ such that
\begin{align*}
\widetilde{\phi}(L \cap W) & \subset \R^{2k} \times \R \times 0, \\
\widetilde{\phi}_*(A_1) & = \frac{\partial}{\partial u} \ \ \ \ \text{on } L \cap W.
\end{align*}
\indent Since $p \colon M \to Z$ is a submersion, it is an open map, and therefore $p(W) \subset M$ is an open set.  Letting $\pi \colon \R^{2k} \times \R \times \R^{\ell} \to \R^{2k} \times \R^{\ell}$ denote the natural projection map, we observe that $\pi \circ \widetilde{\phi} \colon W \to \R^{4n+2} = \R^{2k} \times \R^{\ell}$ descends to a chart $\phi \colon p(W) \to \R^{4n+2} = \R^{2k} \times \R^\ell$ such that
$$\phi(p(L) \cap p(W)) \subset \R^{2k} \times 0.$$
This provides slice coordinates at $z \in p(L)$, showing that $p(L) \subset Z$ is an embedded $2k$-fold. \\
\indent It follows that $p^{-1}(p(L)) \subset M$ is an embedded $(2k+1)$-dimensional submanifold of $M$, so that $L \subset p^{-1}(p(L))$ is an open set for dimension reasons.  That $\Sigma := p(L)$ is $J_+$-complex follows from Proposition \ref{prop:CRCircleBundle}.
\end{proof}

Next, for any submanifold $\Sigma \subset Z$, we note that its circle bundle $p_1^{-1}(\Sigma) \subset M$ is never $\alpha_1$-isotropic.  However, in special situations, it can be $\alpha_2$-isotropic.  In this direction, we first observe:

\begin{lem} \label{lem:HorzIsotropic} Let $\Sigma^k \subset Z^{4n+2}$ be a submanifold with $1 \leq k \leq 2n$.  The following are equivalent:
\begin{enumerate}[(i)]
\item $p_1^{-1}(\Sigma)$ is $\alpha_2$-isotropic.
\item $p_1^{-1}(\Sigma)$ is $\alpha_3$-isotropic.
\item $\Sigma$ is horizontal.
\end{enumerate}
\end{lem}

\begin{proof} Let $\Sigma \subset Z^{4n+2}$ be a submanifold with $\dim(\Sigma) \leq 2n$, and set $L = p_1^{-1}(\Sigma) \subset M$.  Fix $x \in L$ and let $z = p_1(x) \in \Sigma$.   \\
\indent (i)$\iff$(ii). Suppose that $L$ is $\alpha_2$-isotropic at $x$.  By Proposition \ref{prop:LegendrianEquivalence}, we have both $T_xL \subset \Ker(\alpha_2)$ and $\Omega_2|_{T_xL} = 0$.  That is, the subspace $T_xL \subset \Ker(\alpha_2)$ is $\Omega_2$-isotropic.  Therefore, since $A_1 \in T_xL$, it follows that $A_3 = -\mathsf{J}_2(A_1)$ is orthogonal to $T_xL$, and hence $T_xL \subset \text{Ker}(\alpha_3)$, showing that $L$ is $\alpha_3$-isotropic at $x$. \\
\indent (ii)$\iff$(iii). Since $A_1 \in T_xL$, we can write $T_xL = \R A_1 \oplus \widetilde{U}$ for some subspace $\widetilde{U} \subset \Ker(\alpha_1)$.  Since $p_1|_{\Ker(\alpha_1)} \colon \Ker(\alpha_1)|_x \to T_zZ$ is an isometry, it follows that $(p_1)_*(\widetilde{U}) = T_z\Sigma$.  Now, observe that:
\begin{align*}
T_z\Sigma \subset \mathsf{H} & \iff (p_1)_*(\widetilde{U}) \subset (p_1)_*(\widetilde{\mathsf{H}})  \iff \widetilde{U} \subset \Ker(\alpha_1, \alpha_2, \alpha_3) \iff T_xL \subset \Ker(\alpha_2, \alpha_3).
\end{align*}
Thus, if $\Sigma$ is horizontal at $z$, then $T_z\Sigma \subset \mathsf{H}$, so that $T_xL \subset \mathrm{Ker}(\alpha_2, \alpha_3)$, and hence $L$ is both $\alpha_2$- and $\alpha_3$-isotropic at $x$.  Conversely, if $L$ is $\alpha_2$-isotropic at $x$, then by the previous paragraph, $L$ is also $\alpha_3$-isotropic at $x$, so $T_xL \subset \Ker(\alpha_2, \alpha_3)$, and hence $\Sigma$ is horizontal at $z$.
\end{proof}

\begin{cor} \label{cor:CRIsoBundle} Let $\Sigma^{2k} \subset Z^{4n+2}$ be a submanifold, $2 \leq 2k \leq 2n$.  Then $\Sigma$ is $J_+$-complex and horizontal if and only if $p_1^{-1}(\Sigma)$ is $I_1$-CR isotropic  (i.e.: $I_1$-CR, $\alpha_2$-isotropic, and $\alpha_3$-isotropic).
\end{cor}

\begin{proof} This follows immediately from Proposition \ref{prop:CRCircleBundle} and Lemma \ref{lem:HorzIsotropic}.
\end{proof}

\begin{cor} \label{cor:CRIsoChar1} Let $L^{2k+1} \subset M^{4n+3}$ be a submanifold, $3 \leq 2k+1 \leq 2n+1$.  Then $L$ is $I_1$-CR isotropic if and only if $L$ is locally of the form $p_1^{-1}(\Sigma)$ for some horizontal $J_+$-complex submanifold $\Sigma^{2k} \subset Z^{4n+2}$.
\end{cor}

\begin{proof} This follows from Proposition \ref{prop:CRImpliesBundle} and Corollary \ref{cor:CRIsoBundle}.
\end{proof}

\indent When $\Sigma$ is $2n$-dimensional, the situation is particularly special:

\begin{cor} \label{cor:2nEquivalence}
Let $\Sigma^{2n} \subset Z^{4n+2}$ be $2n$-dimensional.  The following are equivalent:
\begin{enumerate}[(i)]
\item $\Sigma$ is $J_+$-complex and horizontal.
\item $\Sigma$ is horizontal.
\item $p_1^{-1}(\Sigma)$ is $\alpha_2$-Legendrian.
\item $p_1^{-1}(\Sigma)$ is $\alpha_3$-Legendrian.
\item $p_1^{-1}(\Sigma)$ is $I_1$-CR Legendrian (i.e.: $I_1$-CR, $\alpha_2$-Legendrian, and $\alpha_3$-Legendrian).
\item $p_1^{-1}(\Sigma)$ is $\Psi_2$-special Legendrian of phase $i^{n+1}$ and $\Psi_3$-special Legendrian of phase $1$.
\end{enumerate}
\end{cor}

\begin{proof} The equivalence (ii)$\iff$(iii)$\iff$(iv) is Lemma \ref{lem:HorzIsotropic}.  The equivalence (i)$\iff$(v) is Corollary \ref{cor:CRIsoBundle}. \\
\indent It is clear that (v)$\implies$(iv).  Conversely, if (iv) holds, then $L := p_1^{-1}(\Sigma)$ is both $\alpha_3$-Legendrian and $\alpha_2$-Legendrian, so that $\mathrm{C}(L) \subset C$ is both $\omega_2$-Lagrangian and $\omega_3$-Lagrangian, and therefore $\mathrm{C}(L)$ is $I_1$-complex Lagrangian.  This proves (v). \\
\indent It remains only to involve condition (vi).  For this, note that (v)$\implies$(vi) follows from Corollary \ref{cor:CRLegEquiv}, and (vi)$\implies$(iii) follows from Proposition \ref{prop:ConeSpecialLag}.
\end{proof}

\noindent The results of this subsection can be summarized in the following table.
\small
$$\begin{tabular}{| c | l || l | c | c |} \hline
$\dim(p_1^{-1}(\Sigma))$ & Circle bundle $p_1^{-1}(\Sigma) \subset M$ & Base $\Sigma \subset Z$ & $\dim(\Sigma)$ & Ref. \\ \hline \hline
$2k+1$ & $I_1$-CR & $J_+$-complex & $2k$ & \ref{prop:CRImpliesBundle} \\ \hline
$3$ & $\phi_2$-associative & $J_+$-complex & $2$ & \ref{prop:CRCircleBundle} \\ \hline
$\leq 2n+1$ & $\alpha_2$-isotropic & Horizontal & $\leq 2n$ & \ref{lem:HorzIsotropic} \\ \hline
$2n+1$ & $\alpha_2$-Legendrian & ($J_+$-complex and) horizontal & $2n$ & \ref{cor:2nEquivalence} \\ \hline
$2n+1$ & $\Psi_2$-special Legendrian & ($J_+$-complex and) horizontal & $2n$ & \ref{cor:2nEquivalence} \\
& of phase $i^{n+1}$ & & & \\ \hline
$2n+1$ & $I_1$-CR Legendrian & ($J_+$-complex and) horizontal & $2n$ & \ref{cor:2nEquivalence} \\ \hline
$2k+1 \leq 2n+1$ & $I_1$-CR isotropic & $J_+$-complex and horizontal & $2k \leq 2n$ & \ref{cor:CRIsoChar1} \\ \hline
\end{tabular}$$
\normalsize

\subsubsection{$p_1$-Horizontal Lifts}

\indent \indent Let $L \subset M^{4n+3}$ be a submanifold, and recall that $L$ is $p_1$-horizontal if and only if it is $\alpha_1$-isotropic.  In this case, $\dim(L) \leq 2n+1$, and its projection $p_1(L) \subset Z$ is $\omega_{\mathrm{KE}}$-isotropic.  Conversely:

\begin{prop} \label{prop:GenHorzLift} Let $\Sigma \subset Z^{4n+2}$ be a submanifold.  Then $\Sigma$ locally lifts to a $p_1$-horizontal submanifold of $M$ if and only if $\Sigma$ is $\omega_{\mathrm{KE}}$-isotropic.  In this case, $\dim(\Sigma) \leq 2n+1$.
\end{prop}

\begin{proof} Suppose first that $\Sigma$ locally lifts to a $p_1$-horizontal submanifold $\widehat{\Sigma} \subset M$.  Since $\widehat{\Sigma}$ is $p_1$-horizontal, we have that $\left.\alpha_1\right|_{\widehat{\Sigma}} = 0$.  Therefore, Proposition \ref{prop:LegendrianEquivalence} implies that $\left.(p_1^*\omega_{\mathrm{KE}})\right|_{\widehat{\Sigma}} = \left.\Omega_1\right|_{\widehat{\Sigma}} = 0$, and hence $\left.\omega_{\mathrm{KE}}\right|_\Sigma = 0$. \\
\indent Conversely, suppose that $\Sigma$ is $\omega_{\mathrm{KE}}$-isotropic.  Since $p_1 \colon M \to Z$ is a Riemannian submersion, the restriction of the derivative $(p_1)_* \colon TM \to TZ$ to the $p_1$-horizontal subbundle $\Ker(\alpha_1) \subset TM$ is an isometric isomorphism.  Consider the distribution on $M$ defined by $D := (p_1)_*|_{\Ker(\alpha_1)}^{-1}(T\Sigma) \subset TM$.  Since $\Sigma$ is $\omega_{\mathrm{KE}}$-isotropic, we have $\omega_{\mathrm{KE}}|_{T\Sigma} = 0$, and therefore $2\Omega_1|_D = 2(p_1^*\omega_{\mathrm{KE}})|_{D} = 0$.  Since, by~\eqref{eq:deriv-alpha}, $2\Omega_1$ is the curvature $2$-form of the connection $\alpha_1$ on the bundle $p_1 \colon M \to Z$, an application of the Frobenius Theorem implies that $D$ is locally integrable.  By construction, the integral submanifolds of $D$ are (local) $p_1$-horizontal lifts of $\Sigma$.
\end{proof}

\subsubsection{$p_1$-Horizontality and CR Isotropic Submanifolds}

\indent \indent Note that if $L \subset M$ is $p_1$-horizontal, then $L$ cannot be $I_1$-CR.  Nevertheless, it is possible for $L$ to be $I_2$-CR or $I_3$-CR.  Moreover, it is also possible for $L$ to be both $p_1$- and $p_2$-horizontal simultaneously.  The following proposition elaborates on this.

\begin{prop} Let $L^{2k+1} \subset M$ be a $(2k+1)$-dimensional submanifold, $3 \leq 2k+1 \leq 2n+1$.  Then:
\begin{enumerate}[(a)]
\item $L$ is $I_2$-CR and $p_1$-horizontal if and only if $L$ is $I_2$-CR isotropic.
\item Suppose $\dim(L) = 2n+1$.  Then $L$ is $I_2$-CR and $p_1$-horizontal$\iff$$L$ is $I_2$-CR Legendrian$\iff$$L$ is $p_3$-horizontal and $p_1$-horizontal.
\end{enumerate}
\end{prop}

\begin{proof} (a) This follows from Proposition \ref{prop:ConeComplexLag} (iii)$\iff$(iv) with indices 1,2,3 replaced by 2,1,-3.

\indent (b) This follows from Corollary \ref{cor:CRLegEquiv} (iv)$\iff$(v), again with 1,2,3 replaced by 2,1,-3.
\end{proof}

\indent Now, given a CR isotropic submanifold $L \subset M$, we consider the geometric properties of its projection $p_1(L) \subset Z$.  To state the result, we introduce the following notation.  For a vertical unit vector $V \in \mathsf{V}_z \subset T_zZ$, we let $\beta_V := \iota_V(\mathrm{Re}\,\gamma)$ denote the induced non-degenerate $2$-form on $\mathsf{H}_z$, and let $J_V \in \End(\mathsf{H}_z)$ denote the corresponding complex structure on $\mathsf{H}_z$.

\begin{prop} \label{prop:CRIsoProjection} Let $L^{2k+1} \subset M$ be a $(2k+1)$-dimensional submanifold, $3 \leq 2k+1 \leq 2n+1$.
\begin{enumerate}[(a)]
\item If $L$ is $\alpha_1$-isotropic and $(-s_\theta \alpha_2 + c_\theta \alpha_3)$-isotropic for some $e^{i\theta} \in S^1$, then  $p_1(L) \subset Z$ is $\omega_{\mathrm{KE}}$-isotropic and $\omega_{\mathrm{NK}}$-isotropic.
\item If $L$ is $(c_\theta I_2 + s_\theta I_3)$-CR isotropic for some $e^{i\theta} \in S^1$, then $\Sigma := p_1(L) \subset Z$ is $\omega_{\mathrm{KE}}$-isotropic, $\omega_{\mathrm{NK}}$-isotropic, and HV-compatible.  Moreover, $\dim(T_z\Sigma \cap \mathsf{V}) = 1$ for all $z \in \Sigma$, and the $2k$-plane $T_z\Sigma \cap \mathsf{H}$ is $J_V$-invariant for any vertical unit vector $V \in T_z\Sigma \cap \mathsf{V}$.
\end{enumerate}
\end{prop}

\begin{proof} (a) Suppose $L \subset M$ is $\alpha_1$-isotropic and $(-s_\theta \alpha_2 + c_\theta \alpha_3)$-isotropic for some constant $e^{i\theta} \in S^1$.  On $L$, we have $\alpha_1 = 0$ and $-s_\theta \alpha_2 + c_\theta \alpha_3 = 0$.  This second equation implies
\begin{align*}
c_\theta \,\alpha_2 \wedge \alpha_3 & = 0 & s_\theta \,\alpha_2 \wedge \alpha_3 & = 0,
\end{align*}
and hence $\alpha_2 \wedge \alpha_3 = 0$ on $L$.  Therefore, $\alpha_1 = 0$ implies $0 = d\alpha_1 = 2 \Omega_1 = 2(\alpha_2 \wedge \alpha_3 + \kappa_1) = 2\kappa_1$, so that $\kappa_1 = 0$ on $L$.  We deduce that $\Omega_1|_L = 0$ and $\widetilde{\Omega}_1|_L = 0$.  Therefore, on the projection $p_1(L) \subset Z$, we have both $\left.\omega_{\mathrm{KE}}\right|_{p_1(L)} = 0$ and $\left.\omega_{\mathrm{NK}}\right|_{p_1(L)} = 0$. \\

\indent (b) Suppose $L \subset M$ is $(c_\theta I_2 + s_\theta I_3)$-CR isotropic, so that $L$ is $\alpha_1$-isotropic and $(-s_\theta \alpha_2 + c_\theta \alpha_3)$-isotropic, and $(c_\theta I_2 + s_\theta I_3)$-CR.  By part (a), the projection $\Sigma := p_1(L)$ is both $\omega_{\mathrm{KE}}$-isotropic and $\omega_{\mathrm{NK}}$-isotropic. \\
\indent  Fix $x \in L$, let $z = p(x) \in \Sigma$, set $\widetilde{V} = c_\theta A_2 + s_\theta A_3 \in T_xM$, and let $\mathsf{J}_V = c_\theta \mathsf{J}_2 + s_\theta \mathsf{J}_3$.  By assumption, we can write $T_xL = H_L \oplus \R \widetilde{V}$ for some $\mathsf{J}_V$-invariant $2k$-plane $H_L \subset \widetilde{\mathsf{H}}$.  It follows that
$T_z\Sigma = H_\Sigma \oplus \R V,$ where $H_\Sigma := p_*(H_L) \subset \mathsf{H}$ is a horizontal $2k$-plane, and $V = p_*(\widetilde{V}) \in \mathsf{V}$ is a vertical unit vector.  In particular, this shows that $\Sigma$ is HV compatible, and that $\dim(T_z\Sigma \cap \mathsf{V}) = 1$. \\
\indent Now, since $\mathrm{Re}(\Gamma_1) = p^*(\mathrm{Re}(\gamma))$, we have that $\iota_{\widetilde{V}}(\mathrm{Re}\,\Gamma_1) = p^*(\iota_{V}(\mathrm{Re}\,\gamma)) = p^*(\beta_V)$ on $L$.  In particular, if $Y \in H_L$ is a horizontal vector tangent to $L$, then 
$$g_{\mathrm{KE}}(p_* \mathsf{J}_VY, p_* \cdot) = g_M(\mathsf{J}_VY, \cdot) = \mathrm{Re}(\Gamma_1)(\widetilde{V}, Y, \cdot) = \beta_V(p_*Y, p_* \cdot) = g_{\mathrm{KE}}(J_Vp_*Y, p_*\cdot),$$
which shows that
\begin{equation} \label{eq:p-commutation}
p_*\mathsf{J}_Y = J_Vp_* \text{ on } H_L.
\end{equation}
\indent Finally, if $X \in T_z\Sigma \cap \mathsf{H} = H_\Sigma$, then $X = p_*(\widetilde{X})$ for some $\widetilde{X} \in H_L$.  Since $H_L$ is $\mathsf{J}_V$-invariant, it follows that $\mathsf{J}_V\widetilde{X} \in H_L$.   Therefore, $J_VX = J_Vp_*(\widetilde{X}) = p_*(\mathsf{J}_V\widetilde{X}) \in p_*(H_L) = H_\Sigma$, which shows that $H_\Sigma$ is $J_V$-invariant.
\end{proof}

\indent Conversely, we now ask which submanifolds $\Sigma \subset Z$ admit local $p_1$-horizontal lifts to CR isotropic submanifolds of $M$.  As we now show, the necessary conditions given in Proposition \ref{prop:CRIsoProjection}(b) are in fact sufficient:

\begin{prop} \label{prop:LiftingDoubleIsotropic} Let $\Sigma^k \subset Z^{4n+2}$ be a submanifold, $3 \leq k \leq 2n+1$, that is $\omega_{\mathrm{KE}}$-isotropic, $\omega_{\mathrm{NK}}$-isotropic, and HV-compatible.
\begin{enumerate}[(a)]
\item If $\Sigma$ is nowhere tangent to $\mathsf{H}$, then every local $p_1$-horizontal lift of $\Sigma$ is $\alpha_1$-isotropic and $(-s_\theta \alpha_2 + c_\theta \alpha_3)$-isotropic for some constant $e^{i\theta} \in S^1$.
\item If $\dim(T_z\Sigma \cap \mathsf{V}) = 1$ for all $z \in \Sigma$, and if $T_z\Sigma \cap \mathsf{H}$ is $J_V$-invariant for any vertical unit vector $V \in T_z\Sigma \cap \mathsf{V}$, then every local $p_1$-horizontal lift of $\Sigma$ is $(c_\theta I_2 + s_\theta I_3)$-CR isotropic for some constant $e^{i\theta} \in S^1$.
\end{enumerate}
\end{prop}

\begin{proof} (a) Let $\Sigma \subset Z$ be as in the statement.  Since $\Sigma$ is $\omega_{\mathrm{KE}}$-isotropic, Proposition \ref{prop:GenHorzLift} implies that $\Sigma$ locally admits a $p_1$-horizontal lift to a $k$-dimensional submanifold $L \subset M$, which is automatically $\alpha_1$-isotropic.  Moreover, since $\Sigma$ is HV-compatible, and since $(p_1)_*|_{\mathrm{Ker}(\alpha_1)} \colon \Ker(\alpha_1) \to TZ$ is an isomorphism that respects the horizontal-vertical splitting, it follows that $TL$ splits as
\begin{equation} \label{eq:TL-Splitting}
TL = (TL \cap \widetilde{\mathsf{H}}) \oplus (TL \cap \widetilde{\mathsf{V}}).
\end{equation}
\indent Now, note that the system $\left.\omega_{\mathrm{KE}}\right|_\Sigma = \left.\omega_{\mathrm{NK}}\right|_\Sigma = 0$ is equivalent to $\left.\omega_{\mathsf{V}}\right|_\Sigma = \left.\omega_{\mathsf{H}}\right|_\Sigma = 0$.  Since $p_1^*(\omega_{\mathsf{V}}) = \alpha_2 \wedge \alpha_3$, it follows that $\{\alpha_2|_L, \alpha_3|_L\}$ is a linearly dependent set of $1$-forms on $L$.  Moreover, since $\Sigma$ is nowhere tangent to $\mathsf{H}$, it follows that $L$ is nowhere tangent to $\widetilde{\mathsf{H}} = \Ker(\alpha_1, \alpha_2, \alpha_3)$, and thus there is no point of $L$ at which $\alpha_2|_L, \alpha_3|_L$ simultaneously vanish.  Therefore, there is an $S^1$-valued function $e^{i\theta} \colon L \to S^1$ such that the $1$-form
$$\tau_\theta := - s_\theta \alpha_2 + c_\theta \alpha_3$$
vanishes on $L$.  It remains to show that $e^{i\theta}$ is constant on $L$. For this, we compute on $L$ that
$$0 = d\tau_\theta = d \theta \wedge (- s_\theta \alpha_2 +  c_\theta \alpha_3) + 2( c_\theta \kappa_2 + s_\theta \kappa_3),$$
where we have used that $\alpha_1|_L = 0$ to compute $d\alpha_2 = 2\kappa_2$ and $d\alpha_3 = 2\kappa_3$.  Now, the first term is in $(T^*L \otimes \widetilde{\mathsf{V}}^*)|_L$ while the second is in $\Lambda^2(\widetilde{\mathsf{H}}^*)|_L$, so by equation (\ref{eq:TL-Splitting}), they vanish independently.  In particular, $d\theta \wedge (-s_\theta \alpha_2 +  c_\theta \alpha_3) = 0$.  Together with the equation $c_\theta \alpha_2 + s_\theta \alpha_3 = 0$ on $L$, this implies that $d \theta \wedge \alpha_2 = 0 $ and $d \theta \wedge \alpha_3 = 0$, which yields $d\theta = 0$, so (since $L$ is assumed connected) $\theta$ is constant. \\

\indent (b) Let $\Sigma \subset Z$ be as in the statement.  By part (a), every local $p_1$-horizontal lift $L \subset M$ of the submanifold $\Sigma \subset Z$ is $\alpha_1$-isotropic and $(-s_\theta \alpha_2 + c_\theta \alpha_2)$-isotropic for some $e^{i\theta} \in S^1$.  Thus, it remains only to show that $L$ is $(c_\theta I_2 + s_\theta I_3)$-CR. \\
\indent Fix $x \in L$, and let $z = p_1(x) \in \Sigma$.  By assumption, we may split $T_z\Sigma = H_\Sigma \oplus \R V$, where $V \in \mathsf{V}$ is a unit vector, and $H_\Sigma \subset \mathsf{H}$ is $J_V$-invariant.  Therefore, since $(p_1)_*$ yields an isomorphism $\Ker(\alpha_1)|_x \to T_zZ$ that respects the horizontal-vertical splittings, we may decompose $TL = H_L \oplus \R \widetilde{V}$, where $H_L \subset \widetilde{\mathsf{H}}$ satisfies $p_*(H_L) = H_\Sigma$ and $\widetilde{V} \in \widetilde{\mathsf{V}}$ satisfies $p_*(\widetilde{V}) = V$. \\
\indent Now, since $L$ is both $\alpha_1$-isotropic and  $(-s_\theta \alpha_2 + c_\theta \alpha_2)$-isotropic, it follows that $\widetilde{V} = c_\theta A_2 + s_\theta A_3$.  Let $\mathsf{J}_V = c_\theta \mathsf{J}_2 + s_\theta \mathsf{J}_3$.  If $X \in H_L$, then $p_*X \in H_\Sigma$, so by~\eqref{eq:p-commutation} we have $p_*(\mathsf{J}_VX) = J_V(p_*X) \in H_\Sigma = p_*(H_L)$, and therefore $\mathsf{J}_VX \in H_L$.  Thus $H_L$ is $\mathsf{J}_V$-invariant, and so $L$ is $(c_\theta I_2 + s_\theta I_3)$-CR.
\end{proof}

\indent In the highest and lowest dimensions, the relationship between CR isotropic submanifolds of $M$ and their projections in $Z$ becomes simpler.  Indeed, in the top dimension:

\begin{cor} \label{cor:DoubleLagrangian-CRLift} ${}$
\begin{enumerate}[(a)]
\item If $L^{2n+1} \subset M^{4n+3}$ is $(c_\theta I_2 + s_\theta I_3)$-CR Legendrian for some $e^{i\theta} \in S^1$, then $p_1(L) \subset Z$ is $\omega_{\mathrm{KE}}$-Lagrangian and $\omega_{\mathrm{NK}}$-Lagrangian.
\item Conversely, if $\Sigma^{2n+1} \subset Z^{4n+2}$ is $\omega_{\mathrm{KE}}$-Lagrangian and $\omega_{\mathrm{NK}}$-Lagrangian, then every local $p_1$-horizontal lift of $\Sigma$ is $(c_\theta I_2 + s_\theta I_3)$-CR Legendrian for some $e^{i\theta} \in S^1$.
\end{enumerate}
\end{cor}

\begin{proof} (a) This follows from Proposition \ref{prop:CRIsoProjection}. \\
\indent (b) Suppose $\Sigma \subset Z$ is $\omega_{\mathrm{KE}}$-Lagrangian and $\omega_{\mathrm{NK}}$-Lagrangian.  By Proposition \ref{prop:HV-KE-Iso}(b), it follows that $\Sigma$ is HV compatible, and that $\dim(T_z\Sigma \cap \mathsf{V}) = 1$ at each $z \in \Sigma$.  Therefore, by Proposition \ref{prop:LiftingDoubleIsotropic}(a), every local $p_1$-horizontal lift $L \subset M$ is $\alpha_1$-Legendrian and $(-s_\theta \alpha_2 + c_\theta \alpha_3)$-Legendrian for some constant $e^{i\theta} \in S^1$.  By Corollary \ref{cor:CRLegEquiv} (v)$\implies$(iv), it follows that $L$ is $(c_\theta I_2 + s_\theta I_3)$-CR Legendrian.
\end{proof}

\begin{cor} \label{cor:Regamma-CRLift} ${}$
\begin{enumerate}[(a)]
\item If $L^{3} \subset M^{4n+3}$ is $(c_\theta I_2 + s_\theta I_3)$-CR isotropic for some $e^{i\theta} \in S^1$, then $p_1(L) \subset Z$ is (up to a change of orientation) $\mathrm{Re}(\gamma)$-calibrated and $\omega_{\mathrm{KE}}$-isotropic.
\item Conversely, if $\Sigma^{3} \subset Z^{4n+2}$ is $\mathrm{Re}(\gamma)$-calibrated and $\omega_{\mathrm{KE}}$-isotropic, then every local $p_1$-horizontal lift of $\Sigma$ is $(c_\theta I_2 + s_\theta I_3)$-CR isotropic for some $e^{i \theta} \in S^1$.
\end{enumerate}
\end{cor}

\begin{proof} (a) Let $L^{3} \subset M^{4n+3}$ be a $(c_\theta I_2 + s_\theta I_3)$-CR isotropic $3$-fold.  By Proposition \ref{prop:CRIsoProjection}(b), $\Sigma := p_1(L) \subset Z$ is $\omega_{\mathrm{KE}}$-isotropic, so it remains only to show that $\Sigma$ is $\mathrm{Re}(\gamma)$-calibrated. \\
\indent Fix $z \in \Sigma$.  Again by Proposition \ref{prop:CRIsoProjection}(b), we may decompose $T_z\Sigma = H_\Sigma \oplus \R V$ for some $2$-plane $H_\Sigma \subset \mathsf{H}$ and vertical unit vector $T \in \mathsf{V}_z$.  Let $N \in \mathsf{V}_z$ be the vertical unit vector such that $\{T, N\}$ is an oriented orthonormal basis of $\mathsf{V}_z$, and let $\beta_T, \beta_N \in \Lambda^2(\mathsf{H}_z^*)$ be the induced non-degenerate $2$-forms from $\gamma$.  Since $H_\Sigma$ is $J_V$-invariant, it follows that $\left.\beta_V\right|_{H_\Sigma} = \pm \vol_{H_\Sigma}$.  Therefore,
\begin{align*}
\left.\mathrm{Re}(\gamma)\right|_{T_z\Sigma} & = \left.( T^\flat \wedge \beta_T + N^\flat \wedge \beta_N )\right|_{T_z\Sigma} = \pm\vol_{V_\Sigma} \wedge \vol_{H_\Sigma} + 0 = \pm\vol_\Sigma.
\end{align*}

\indent (b) Suppose $\Sigma^{3} \subset Z^{4n+2}$ is $\mathrm{Re}(\gamma)$-calibrated and $\omega_{\mathrm{KE}}$-isotropic.  By Proposition \ref{prop:HV-KE-Iso}(c), it follows that $\Sigma$ is HV compatible, so we may write $T_z\Sigma = H_\Sigma \oplus V_\Sigma$, where $H_\Sigma \subset \mathsf{H}$ and $V_\Sigma \subset \mathsf{V}$.  The same proposition shows that $\dim(V_\Sigma) = 1$.  Now, let $V \in V_\Sigma$ be a unit vector, let $\beta_V = \iota_V(\mathrm{Re}(\gamma))$ denote the induced non-degenerate $2$-form on $\mathsf{H}_z$, and let $J_V$ be the corresponding complex structure on $\mathsf{H}_z$.  Since $\left.\mathrm{Re}(\gamma)\right|_\Sigma = \vol_\Sigma = \vol_{V_\Sigma} \wedge \vol_{H_\Sigma}$, it follows that $\beta_V|_{H_\Sigma} = \pm\vol_{H_\Sigma}$, which proves that $H_\Sigma$ is $J_V$-invariant.  Therefore, Proposition \ref{prop:LiftingDoubleIsotropic}(b) gives the result.
\end{proof}

\subsubsection{$p_1$-Horizontality of Special Isotropic Submanifolds}

\indent \indent By Proposition \ref{prop:AssocCases}(b), every $-\theta_{I,3}$-special isotropic $3$-fold is $\phi_2$-associative.  Moreover, since $\iota_{A_1}(-\theta_{I,3}) = 0$ by Definition~\ref{defn:special-isotropic}, Proposition \ref{prop:CalibrationSplitting} implies that every $-\theta_{I,3}$-special isotropic $3$-fold is $p_1$-horizontal.  We now observe that these necessary conditions are sufficient:

\begin{prop} \label{prop:HorzAssoc} Let $L^{2k+1} \subset M^{4n+3}$ be a $(2k+1)$-dimensional submanifold, $3 \leq 2k+1 \leq 2n+1$.
\begin{enumerate}[(a)]
\item If $L$ is $\theta_{I, 2k+1}$-special isotropic, then $L$ is $p_1$-horizontal.
\item If $L$ is $\Psi_1$-special Legendrian, then $L$ is $p_1$-horizontal.
\item Suppose $\dim(L) = 3$.  Then $L$ is $-\theta_{I,3}$-special isotropic if and only if $L$ is $\phi_2$-associative and $p_1$-horizontal.
\end{enumerate}
\end{prop}

\begin{proof} (a) Since $\iota_{A_1}(\theta_{I,2k+1}) = 0$, Proposition \ref{prop:CalibrationSplitting} gives the result. \\
\indent (b) This is simply part (a) in the case of $\dim(L) = 2n+1$. \\
\indent (c) Suppose $\dim(L) = 3$.  Then:
\begin{align*}
L \text{ is }\phi_2\text{-associative and }p_1\text{-horizontal} & \iff \left.\left( \alpha_1 \wedge \Omega_1 - \alpha_2 \wedge \Omega_2 + \alpha_3 \wedge \Omega_3  \right)\right|_L = \vol_L \,  \text{ and } \, \alpha_1|_L = 0 \\
L \text{ is }-\theta_{I,3}\text{-special isotropic} & \iff \left.\left(  -\alpha_2 \wedge \Omega_2 + \alpha_3 \wedge \Omega_3 \right)\right|_L = \vol_L
\end{align*}
The result is now immediate.
\end{proof}

\begin{example} For $n = 1$, Proposition \ref{prop:HorzAssoc}(c) is the well-known fact that a $3$-fold $L^3 \subset M^7$ is $\phi_2$-associative and $p_1$-horizontal if and only if it is $\Psi_1$-special Legendrian of phase $-1$.
\end{example}

\subsubsection{$\mathrm{Re}(\Gamma_1)$-Calibrated $3$-folds of $M$} \label{subsub:ReGammaClassify}

\indent \indent We now observe that $\mathrm{Re}(\Gamma_1)$-calibrated $3$-folds $L^3 \subset M^{4n+3}$ are always $p_1$-horizontal, and describe their projections $p_1(L) \subset Z$.  Namely:

\begin{prop} \label{prop:GammaLifts} If $L^3 \subset M^{4n+3}$ is $\mathrm{Re}(\Gamma_1)$-calibrated, then $L$ is $p_1$-horizontal (equivalently, $\alpha_1$-isotropic).  Moreover:
\begin{enumerate}[(a)]
\item If $L^3 \subset M^{4n+3}$ is $\mathrm{Re}(\Gamma_1)$-calibrated, then $L$ is locally a $p_1$-horizontal lift of a $3$-fold in $Z$ that is both $\mathrm{Re}(\gamma)$-calibrated and $\omega_{\mathrm{KE}}$-isotropic.
\item Conversely, if $\Sigma^3 \subset Z^{4n+2}$ is both $\mathrm{Re}(\gamma)$-calibrated and $\omega_{\mathrm{KE}}$-isotropic, then $\Sigma$ locally lifts to a $\mathrm{Re}(\Gamma_1)$-calibrated $3$-fold in $M$.
\end{enumerate}
\end{prop}

\begin{proof} Let $L \subset M$ be a $\mathrm{Re}(\Gamma_1)$-calibrated $3$-fold.  Since $\mathrm{Re}(\Gamma_1) = \alpha_2 \wedge \kappa_2 + \alpha_3 \wedge \kappa_3$, we have $\iota_{A_1}(\mathrm{Re}(\Gamma_1)) = 0$.  In view of the splitting $TM = \R A_1 \oplus \mathrm{Ker}(\alpha_1)$, Proposition \ref{prop:CalibrationSplitting} implies that $TL \subset \Ker(\alpha_1)$, so that $L$ is $p_1$-horizontal (equivalently, $\alpha_1|_L = 0$). \\
\indent Parts (a) and (b) now follow from Proposition \ref{prop:GenHorzLift} and the fact that $\Gamma_1 = p_1^*(\gamma)$.
\end{proof}

\indent We are now in a position to prove Theorem \ref{thm:ReGammaChar1}, which classifies the $\mathrm{Re}(\Gamma_1)$-calibrated $3$-folds in terms of more familiar geometries.

\begin{thm} \label{thm:FirstCharDim3} Let $L^3 \subset M^{4n+3}$ be a $3$-dimensional submanifold.  The following are equivalent:
\begin{enumerate}[(i)]
\item $\mathrm{C}(L)$ is a $(c_\theta I_2 + s_\theta I_3)$-complex isotropic $4$-fold for some constant $e^{i \theta} \in S^1$.
\item $L$ is a $(c_\theta I_2 + s_\theta I_3)$-CR isotropic $3$-fold for some constant $e^{i \theta} \in S^1$.
\item $L$ is locally of the form $p_v^{-1}(S)$ for some horizontal $J_+$-complex curve $S \subset Z$ and some $v = (0, c_\theta, s_\theta)$.
\item $L$ is locally a $p_1$-horizontal lift of a $3$-fold $\Sigma^3 \subset Z$ that is $\mathrm{Re}(\gamma)$-calibrated and $\omega_{\mathrm{KE}}$-isotropic.
\item $L$ is $\mathrm{Re}(\Gamma_1)$-calibrated.
\end{enumerate}
\end{thm}

\begin{proof} (i)$\iff$(ii). This follows from Proposition \ref{prop:ConeComplexLag}. \\
\indent  (ii)$\iff$(iii).  This is Corollary \ref{cor:CRIsoChar1}. \\
\indent (ii)$\iff$(iv). This is Corollary \ref{cor:Regamma-CRLift}. \\
\indent (iv)$\iff$(v).  This is Proposition \ref{prop:GammaLifts}.
\end{proof}

\section{Submanifolds of Quaternionic K\"{a}hler Manifolds}

\indent \indent Thus far, we have studied twistor spaces $Z$ as $S^1$-quotients of $3$-Sasakian manifolds $M$.  In $\S$\ref{sec:QKManifolds}, we adopt a different perspective, viewing $Z$ as the total space of a canonical $S^2$-bundle $\tau \colon Z \to Q$ over a quaternionic-K\"{a}hler manifold $Q^{4n}$.  This leads to an alternative construction of the $\Sp(n)\U(1)$-structure on $Z$, including the $3$-form $\gamma \in \Omega^3(Z;\C)$. \\
\indent In $\S$\ref{subsec:TotallyComplex}, we turn our attention to \emph{totally-complex} submanifolds of $Q^{4n}$, a class that is intimately related to the (semi-)calibrated geometries of previous sections.  To explain these relations, we will recall that a totally-complex submanifold $U^{2k} \subset Q^{4n}$ admits two distinct lifts to $Z$, namely its \emph{$\tau$-horizontal lift} $\widetilde{U}^{2k} \subset Z$, and its \emph{geodesic circle bundle lift} $\mathcal{L}(U)^{2k+1} \subset Z$. \\
\indent Given such a circle bundle lift $\mathcal{L}(U) \subset Z$, we will prove (Corollary \ref{cor:LiftOfCircleBundle}) that its local $p_1$-horizontal lifts to $M$ are $(c_\theta I_2 + s_\theta I_3)$-CR isotropic.  The main result of this section (Theorem \ref{thm:ProjCRIsoCircleBundle}) is that the converse also holds: If $L \subset M$ is a compact $(c_\theta I_2 + s_\theta I_3)$-CR isotropic submanifold, then $L$ is a $p_1$-horizontal lift of some circle bundle $\mathcal{L}(U)$.  As an application, we prove (Theorem \ref{thm:Double-Lagrangian-Char}) that every compact $(2n+1)$-fold $\Sigma \subset Z$ that is Lagrangian with respect both $\omega_{\mathrm{KE}}$ and $\omega_{\mathrm{NK}}$ is of the form $\mathcal{L}(U)$, thereby generalizing a result of Storm \cite{Storm} to higher dimensions. \\
\indent We remind the reader that as mentioned in the introduction, we only consider submanifolds of $Q$ that do not meet any orbifold points.

\subsection{Quaternionic K\"{a}hler Manifolds} \label{sec:QKManifolds}

\indent \indent Let $Q^{4n}$ be a smooth $4n$-manifold, $n \geq 1$.

\begin{defn} \label{defn:QH}
An \emph{almost quaternionic-Hermitian structure} (or \emph{$\Sp(n)\Sp(1)$-structure}) on $Q$ is a pair $(g_Q, E)$ consisting of an orientation and a Riemannian metric $g_Q$, and a rank $3$ subbundle $E \subset \mathrm{End}(TQ)$ such that:
\begin{enumerate}
\item At each $q \in Q$, there exists a local frame $(j_1, j_2, j_3)$ of $E$, called an \emph{admissible frame}, satisfying the quaternionic relations $j_1j_2 = j_3$ and $j_1^2 = j_2^2 = j_3^2 = -\mathrm{Id}$.
\item Every $j \in E$ acts by isometries: $g_Q(jX, jY) = g_Q(X,Y)$, for all $X,Y \in TQ$.
\end{enumerate}
Equivalently, an almost quaternionic-Hermitian structure may be defined as $4$-form $\Pi \in \Omega^4(Q)$ such that at each $q \in Q$, there exists a coframe $L \colon T_qQ \to \R^{4n}$ for which $\Pi|_q = \frac{1}{6}L^*(\beta_1^2 + \beta_2^2 + \beta_3^2)$, where $\{\beta_1, \beta_2, \beta_3\}$ is the standard hyperk\"{a}hler triple on $\R^{4n} = \HH^n$. (See~\cite{Salamon} or~\cite{Boyer-Galicki} for details.)
\end{defn}

\begin{defn} \label{defn:QK}
Let $n \geq 2$. An almost quaternionic-Hermitian structure $(g_Q, E)$ is \emph{quaternionic-K\"{a}hler (QK)} if $E \subset \mathrm{End}(TQ)$ is a parallel subbundle (with respect to the connection $\nabla$ induced by $g_Q$).  That is, if $\sigma$ is a local section of $E$, then $\nabla \sigma$ is also a local section of $E$.  An equivalent condition is that the $4$-form $\Pi \in \Omega^4(Q)$ is $g_Q$-parallel. \\
\indent For $n = 1$, we say $(Q^4, g_Q)$ is \emph{quaternionic-K\"{a}hler} if the metric $g_Q$ is Einstein and anti-self-dual.
\end{defn}

\begin{rmk} It is well-known that if $(g_Q, E)$ is a QK structure, then $\mathrm{Hol}(g_Q) \leq \Sp(n)\Sp(1)$.  Conversely, for $n \geq 2$, if $g$ is a Riemannian metric on $Q$ with $\mathrm{Hol}(g) \leq \Sp(n)\Sp(1)$, then there exists a $g$-parallel rank $3$ subbundle $E \subset \End(TQ)$ such that $(g,E)$ is a QK structure.
\end{rmk}

\subsubsection{The Twistor Space} \label{subsec:TwistorSpace}

\indent \indent From now on, $(Q^{4n}, g_Q, E)$ denotes a quaternionic-K\"{a}hler $4n$-manifold with positive scalar curvature.  The \emph{twistor space of $Q$} is the $(4n+2)$-manifold
$$Z := \{ j \in E \colon j^2 = -\mathrm{Id}\}.$$
The obvious projection map $\tau \colon Z \to Q$ is then an $S^2$-bundle, and we let $\mathsf{V} \subset TZ$ denote the (rank $2$) vertical bundle.  The Levi-Civita connection of $g_Q$ induces a connection on the vector bundle $E \subset \End(TZ)$, and hence a connection on the $S^2$-subbundle $Z \subset E$, thereby yielding a $4n$-plane field $\mathsf{H} \subset TZ$ such that
$$TZ = \mathsf{H} \oplus \mathsf{V}.$$
We now recall the K\"{a}hler-Einstein structure $(g_{\mathrm{KE}}, \omega_{\mathrm{KE}}, J_{\mathrm{KE}})$ on $Z$.  First, define a Riemannian metric $g_{\mathrm{KE}}$ by requiring that $g_{\mathrm{KE}}(\mathsf{H}, \mathsf{V}) = 0$ and
\begin{enumerate}
\item For $X,Y \in \mathsf{H}$, we have $g_{\mathrm{KE}}(X,Y) = g_Q(\tau_*X, \tau_*Y)$.
\item On $\mathsf{V}$, the metric $g_{\mathrm{KE}}$ is induced by the fiber metric $\langle \cdot, \cdot \rangle$ on $E \subset \End(TZ)$ under the identifications $\mathsf{V}_z \simeq T_z(Z_{\tau(z)}) \subset T_z(E_{\tau(z)}) \simeq E_{\tau(z)}$.
\end{enumerate}
Next, define an almost-complex structure $J_{\mathrm{KE}}$ on $Z$ by requiring that both $\mathsf{H}$ and $\mathsf{V}$ are $J_{\mathrm{KE}}$-invariant, and
\begin{enumerate}
\item On $\mathsf{H}_z$, we set $J_{\mathrm{KE}} = (\left.\tau_*\right|_{\mathsf{H}_z})^{-1} \circ z \circ \tau_*$.
\item On $\mathsf{V}_z$, identifying vertical vectors $X \in \mathsf{V}_z \simeq T_z(Z_{\tau(z)})$ with endomorphisms $j_X \in z^\perp = \{j \in E_{\tau(z)} \colon \langle j,z \rangle = 0\}$, we set $J_{\mathrm{KE}}X = z \circ j_X$.
\end{enumerate}
We let $\omega_{\mathrm{KE}}(X,Y) = g_{\mathrm{KE}}(J_{\mathrm{KE}}X, Y)$.  It is well-known~\cite{Salamon} that the triple $(g_{\mathrm{KE}}, \omega_{\mathrm{KE}}, J_{\mathrm{KE}})$ is a K\"{a}hler-Einstein structure.

\begin{rmk} The $(\U(2n) \times \U(1))$-structure $(g_{\mathrm{KE}}, \omega_{\mathrm{KE}}, J_{\mathrm{KE}}, \mathsf{H})$ just defined on $Z$ coincides with the one described in $\S$\ref{sub:GeomTwistor}.  In brief, if $Q^{4n}$ is a quaternionic-K\"{a}hler manifold of positive scalar curvature, then its \emph{Konishi bundle} $M^{4n+3} = F_{\SO(3)}(E)$, which is the $\SO(3)$-frame bundle of the rank $3$ vector bundle $E \to Q$, admits a $3$-Sasakian structure, from which one can recover the $(\U(2n) \times \U(1))$-structure on $Z$.  For details, see ~\cite[$\S$12.2, $\S$13.3.2]{Boyer-Galicki}.
\end{rmk}

\indent Recall from Theorem~\ref{thm:ExtraStructureTwistor} that there exists a canonical $\Sp(n)\U(1)$-structure $\gamma \in \Omega^3(Z; \C)$ on the twistor space $(Z, g_{\mathrm{KE}}, J_{\mathrm{KE}}, \mathsf{H})$. We end this section by giving a different proof of the existence of this $\Sp(n)\U(1)$-structure, working directly from the projection $\tau \colon Z \to Q$, without reference to $M$. At a point $z \in Z$, choose an admissible frame $(z, j_2, j_3)$ at $\tau(z) \in Q$.  Via the isomorphism
$$\mathsf{V}_z \simeq z^\perp = \{j \in E_{\tau(z)} \colon \langle j,z \rangle = 0\},$$
the points $j_2, j_3 \in E_{\tau(z)}$ define vertical vectors at $z$, and hence (via the metric) $1$-forms $\mu_2, \mu_3 \in \Lambda^1(\mathsf{V}^*|_z)$ at $z$.  On the other hand,
\begin{align} \label{eq:J2J3}
J_2 & := (\left.\tau_*\right|_{\mathsf{H}_z})^{-1} \circ j_3 \circ \tau_*  & J_3 & := -(\left.\tau_*\right|_{\mathsf{H}_z})^{-1} \circ j_2 \circ \tau_* 
\end{align}
are $g_{\mathrm{KE}}$-orthogonal complex structures on $\mathsf{H}_z$, and hence yield $2$-forms $\beta_2 := g_{\mathrm{KE}}(J_2 \cdot, \cdot)$ and $\beta_3 := g_{\mathrm{KE}}(J_3 \cdot, \cdot)$ on $\mathsf{H}_z$.  We can now define a $\C$-valued $3$-form $\gamma$ at $z \in Z$ by
\begin{align} \label{eq:gamma-twistor}
\gamma := (\mu_2 - i\mu_3) \wedge (\beta_2 + i\beta_3).
\end{align}
This $3$-form is independent of the choice $(j_2, j_3)$.  That is, one can check that if $(z, \widetilde{j}_2, \widetilde{j}_3) = (z, c_\theta j_2 + s_\theta j_3, -s_\theta j_2 + c_\theta j_3)$ is another admissible frame at $\tau(z)$, then the corresponding $1$-forms $\widetilde{\mu}_2, \widetilde{\mu}_3$ on $\mathsf{V}_z$ and $2$-forms $\widetilde{\beta}_2, \widetilde{\beta}_3$ on $\mathsf{H}_z$ satisfy
\begin{align*}
(\widetilde{\mu}_2 - i\widetilde{\mu}_3) \wedge (\widetilde{\beta}_2 + i \widetilde{\beta}_3) = (\mu_2 - i\mu_3) \wedge (\beta_2 + i\beta_3).
\end{align*}

\begin{rmk} In fact, there is a natural $1$-parameter family of $3$-forms on $Z$ given by $e^{i\theta}\gamma \in \Omega^3(Z;\C)$ for constants $e^{i\theta} \in S^1$.  In particular, the $3$-form defined by (\ref{eq:gamma-twistor}) agrees with that of $\S$\ref{sub:GeomTwistor} (viz., Theorem \ref{thm:ExtraStructureTwistor}) up to a constant $\lambda \in S^1$.  The $90^\circ$ rotation in formula (\ref{eq:J2J3}) relating $(J_2, J_3)$ to $(j_2, j_3)$ was chosen to arrange for $\lambda = 1$. (This follows from Theorem~\ref{thm:LowDimension} and Proposition~\ref{prop:Re-gamma-phases}.)
\end{rmk}

\subsubsection{The Diamond Diagram}

\indent \indent Altogether, the various spaces we have considered can be summarized by the \emph{diamond diagram}:

$$\begin{tikzcd}
   & M^{4n+3} \arrow[r, hook] \arrow[rd, "{p_v}"'] \arrow[dd, "h"'] & C^{4n+4} \arrow[d] &  \\
 &      &     Z^{4n+2} \arrow[ld, "\tau"]         &        \\
                              & Q^{4n}        &          & 
\end{tikzcd}$$

\begin{example} ${}$
\begin{itemize}
\item The \emph{flat model} is $(C,M,Z,Q) = (\HH^{n+1}, \Sph^{4n+3}, \CP^{2n+1}, \HP^n)$, in which each $p_v \colon \Sph^{4n+3} \to \CP^{2n+1}$ for $v \in S^2$ is a complex Hopf fibration, $h \colon \Sph^{4n+3} \to \HP^n$ is the quaternionic Hopf fibration, and $\tau \colon \CP^{2n+1} \to \HP^n$ is the classical twistor fibration.
\item Perhaps the second simplest family of examples is
$$(M, Z, Q) = \left( \Sph(T^*\CP^{n+1}),\, \mathbb{P}(T^*\CP^{n+1}),\, \Gr_2(\C^{n+2}) \right)\!,$$
where $\mathbb{P}(T^*\CP^{n+1})$ and $\Sph(T^*\CP^{n+1})$ refer to the projectivized cotangent bundle and unit sphere subbundle of the cotangent bundle of $\CP^{n+1}$, respectively \cite{Tsukada-2016}.  In the case of $n = 1$, these spaces are $(M^7, Z^6, Q^4) = (N_{1,1},\, \frac{\SU(3)}{T^2},\, \CP^2)$, where $N_{1,1} = \frac{\SU(3)}{\U(1)}$ is an exceptional Aloff-Wallach space.
\item An exceptional example is $\textstyle (M^{11}, Z^{10}, Q^8) = \left( \frac{\G_2}{\Sp(1)_+}, \, \frac{\G_2}{\U(2)_+}, \, \frac{\G_2}{\SO(4)} \right)$.  Here, $M^{11}$ and $Z^{10}$ should not be confused with $\frac{\G_2}{\Sp(1)_-} \cong V_2(\R^7)$ and $\frac{\G_2}{\U(2)_-} \cong \Gr_2(\R^7)$.  See \cite[Example 13.6.8]{Boyer-Galicki}.
\end{itemize}
\end{example}

\subsection{Totally-complex Submanifolds} \label{subsec:TotallyComplex}

\indent \indent We now turn to the various submanifolds of a quaternionic-K\"{a}hler manifold $(Q^{4n}, g_Q, E)$, continuing to assume that $g_Q$ has positive scalar curvature.

\begin{defn} A submanifold $U^{2k} \subset Q^{4n}$ is \emph{almost-complex} if there exists a section $i \in \Gamma(Z|_U)$ such that $i(T_uU) = T_uU$ for all $u \in U$.
\end{defn}

We will be particularly interested in the following subclass of almost-complex submanifolds.

\begin{defn} \label{defn:totally-complex}
A submanifold $U^{2k} \subset Q^{4n}$, for $1 \leq k \leq 2n$, is called \emph{totally-complex} if there exists a section $i \in \Gamma(Z|_U)$ such that at each $u \in U$:
\begin{enumerate}
\item $i(T_uU) = T_uU$.
\item For all $j \in Z_u$ with $\langle j,i \rangle = 0$, we have $j(T_uU) \subset (T_uU)^\perp$.
\end{enumerate}
A totally-complex submanifold $U \subset Q^{4n}$ is called \emph{maximal} if $\dim(U) = 2n$.
\end{defn}

Totally-complex submanifolds were introduced by Funabashi \cite{Funabashi}, who proved that they are minimal (zero mean curvature) provided $n \geq 2$.

\begin{example} \label{ex:TC} ${}$
\begin{itemize}
\item In $Q = \HP^n$, the maximal totally-complex submanifolds with parallel second fundamental form were classified by Tsukada \cite{Tsukada-1985}.  The list consists of the two infinite families
\begin{align*}
\CP^n & \to \HP^n & \CP^1 \times \frac{\SO(n+1)}{\SO(2) \times \SO(n-1)} & \to \HP^n \ \ \ \ (n \geq 2)
\end{align*}
and four sporadic exceptions (in $\HP^6, \HP^9, \HP^{15}$, and $\HP^{27}$).  Bedulli--Gori--Podest\`a \cite{BGP} proved that a maximal totally-complex submanifold of $\HP^n$ is homogeneous if and only if it appears on Tsukada's list.
\item If $Q = \Gr_2(\C^{n+2})$, the maximal totally-complex submanifolds that are homogeneous have been recently classified by Tsukada \cite{Tsukada-2016}.
\item If $Q$ is a quaternionic symmetric space, the maximal totally-complex submanifolds that are totally geodesic have been classified by Takeuchi \cite{Takeuchi}.
\end{itemize}
\end{example}

\begin{rmk} \label{rmk:AM-results} Totally-complex submanifolds are also studied by Alekseevsky--Marchiafava \cite{AM-2001}, \cite{AM-2005}.  In particular, they prove the following results for almost-complex submanifolds $U^{2k} \subset Q^{4k}$:
\begin{itemize}
\item If $k \geq 2$ (so that $n \geq 2$), then
$$\nabla_Xi = 0, \ \forall X \in TU \ \ \iff \ \ U \text{ is totally-complex} \ \ \iff \ \ \left(U, \left.g_Q\right|_{U}, \left.i\right|_{U}\right) \text{ is K\"{a}hler.}$$
For this reason, totally-complex submanifolds $U$ of real dimension $\geq 4$ are sometimes called ``K\"{a}hler submanifolds" in the literature.
\item If $k = 1$ and $n \geq 2$, then the equivalence
$$\nabla_Xi = 0, \ \forall X \in TU \ \ \iff \ \ U \text{ is totally-complex}$$
continues to hold.  By contrast, the condition that $\left(U, \left.g_Q\right|_{U}, \left.i\right|_{U}\right)$ be K\"{a}hler is automatic.
\item If $k = 1$ and $n = 1$, then every oriented surface $U^2 \subset Q^4$ is totally-complex, and $\left(U, \left.g_Q\right|_{U}, \left.i\right|_{U}\right)$ is K\"{a}hler.  By contrast, $\nabla_X i = 0$ for all $X \in TU$ is equivalent to $U$ being \emph{superminimal} (or \emph{infinitesimally holomorphic}), a condition on the second fundamental form (see, e.g. \cite{Bryant}, \cite{Friedrich-1984}, \cite{Friedrich-1997}).
\end{itemize}
\end{rmk}

\subsubsection{The Horizontal Lift}

\indent \indent Given a totally-complex submanifold $U^{2k} \subset Q^{4n}$, there are two natural ways to lift $U$ to a submanifold of the twistor space $Z$.  The first of these is the \emph{horizontal lift} $\widetilde{U} \subset Z$, defined as the union of
\begin{equation*}
\widetilde{U}_p : = \left\{ z \in Z_p \colon z(T_pU) = T_pU \right\}\!
\end{equation*}
for $p \in U$.  The following results were proved in \cite[Theorem 4.1]{Takeuchi}, and later generalized in \cite[Theorem 4.2, Proposition 4.7]{AM-2005}.

\begin{lem}[\cite{AM-2005}] \label{lem:ComplexSection} Let $U \subset Q$ be a submanifold, let $i \in \Gamma(Z|_U)$ be a section over $U$, and let $N = i(U) \subset Z$ be its image.  Then $N \subset Z$ is $J_{\mathrm{KE}}$-complex and horizontal if and only if $(U,i)$ is almost-complex and $\nabla_V i = 0$ for all $V \in TU$.
\end{lem}

\begin{proof} ($\Longrightarrow)$ Suppose $N$ is $J_{\mathrm{KE}}$-complex and horizontal.  Fix $u \in U$, and let $z = i(u) \in N$.  Let $X \in T_uU$, and write $X = \tau_*(\widetilde{X})$ for some $\widetilde{X} \in T_zN$.  Since $T_zN \subset T_zZ$ is complex, we have $J_{\mathrm{KE}}\widetilde{X} \in T_zN$.  Since $\widetilde{X}$ is horizontal, we may calculate $i(u)(X) = z(\tau_*\widetilde{X}) = \tau_*(J_{\mathrm{KE}}\widetilde{X}) \in \tau_*(T_zN) = T_uU$. This shows that $(U,i)$ is almost-complex.  Moreover, since $N = i(U)$ is horizontal, it follows that $\nabla_V i = 0$ for all $V \in TU$. \\
\indent $(\Longleftarrow)$ Suppose $(U,i)$ is almost-complex and $\nabla_V i = 0$ for all $V \in TU$.  Since $i$ is a parallel section, its image $N$ is horizontal.  Now, fix $z \in N$, write $z = i(u)$ for $u \in U$, and let $Y \in T_zN$.  Since $(U,i)$ is almost-complex, we have $i(u)(\tau_*Y) \in T_uU$.  Therefore, since $Y$ is horizontal, we have $\tau_*(J_{\mathrm{KE}}Y) = i(u)(\tau_*Y) \in T_uU = \tau_*(T_zN)$. Since $\tau_* \colon \mathsf{H}_z \to T_uQ$ is an isomorphism, it follows that $J_{\mathrm{KE}}Y \in T_zN$, which proves that $N$ is $J_{\mathrm{KE}}$-complex.
\end{proof}

\begin{thm}[\cite{AM-2005, Takeuchi}] \label{prop:HorzLiftTotCplx} Let $\Sigma^{2k} \subset Z^{4n+2}$ be a submanifold, where $1 \leq k \leq n$.  Then $\Sigma$ is $J_{\mathrm{KE}}$-complex and horizontal if and only if $\Sigma$ is locally of the form $\widetilde{U}$ for some totally-complex $U^{2k} \subset Q^{4n}$ (respectively, a superminimal surface $U^2 \subset Q^4$ if $n = 1$).
\end{thm}

\begin{proof} $(\Longleftarrow)$ Suppose that $\Sigma$ is locally of the form $\widetilde{U}$ for some totally-complex $U \subset Q$ (respectively, superminimal surface if $n = 1$).  By definition, $U$ is almost-complex, so there exists a section $i \in \Gamma(Z|_U)$ such that $i(TU) = TU$, and hence $\widetilde{U} = i(U) \cup -i(U)$.  Moreover, by Remark \ref{rmk:AM-results}, we have $\nabla_Vi = 0$ for all $V \in TU$.  Therefore, by Lemma \ref{lem:ComplexSection}, the submanifolds $i(U)$ and $-i(U)$ are $J_{\mathrm{KE}}$-complex and horizontal, and hence $\Sigma$ is, too. \\
\indent $(\Longrightarrow)$ Suppose that $\Sigma$ is $J_{\mathrm{KE}}$-complex and horizontal.  Since $\Sigma$ is horizontal, the Implicit Function Theorem implies that $\Sigma$ is locally of the form $i(U)$ for some horizontal section $i \in \Gamma(Z|_U)$ over some submanifold $U \subset Q$.  By Lemma \ref{lem:ComplexSection}, $(U,i)$ is almost-complex and $\nabla_Vi = 0$.  Thus, by Remark \ref{rmk:AM-results}, $U$ is totally-complex (and, in addition, superminimal if $n = 1$).
\end{proof}

\subsubsection{The Circle Bundle Lift}

\indent \indent Let $U^{2k} \subset Q^{4n}$ be totally-complex.  The second natural lift of $U$ is the \emph{circle bundle lift} $\mathcal{L}(U) \subset Z$, defined as the union of
\begin{equation*}
\left.\mathcal{L}(U)\right|_p : = \left\{ j \in Z_p \colon j(T_pU) \subset (T_pU)^\perp \right\}\!
\end{equation*}
for $p \in U$.  Each fiber $\mathcal{L}(U)|_p$ is a great circle in the $2$-sphere $Z_p$. \\
\indent The circle bundle lift was introduced by Ejiri--Tsukada \cite{ET-2005}, who proved that if $U^{2k} \subset Q^{4n}$ is totally-complex and $k \geq 2$, then $\mathcal{L}(U) \subset Z$ is a minimal submanifold that is both $\omega_{\mathrm{KE}}$-isotropic and HV-compatible.  In particular, if $\dim(U) = 2n \geq 4$, then $\mathcal{L}(U) \subset Z$ is a minimal $\omega_{\mathrm{KE}}$-Lagrangian.  In the case of $k = n = 1$, circle bundle lifts of superminimal surfaces $U^2 \subset Q^4$ were studied by Storm~\cite{Storm}. \\
\indent We now explore these submanifolds further.  Recall that if $V \in \mathsf{V}_z$ is a vertical unit vector, we let $\beta_V := \iota_V(\mathrm{Re}(\gamma)) \in \Lambda^2(\mathsf{H}_z^*)$ denote the induced non-degenerate $2$-form on $\mathsf{H}_z$, and let $J_V$ be the corresponding complex structure on $\mathsf{H}_z$.

\begin{thm} \label{prop:NecessaryCircleBundle} Let $U^{2k} \subset Q^{4n}$ be a submanifold with $1 \leq k \leq n$.  If $U$ is totally-complex and $n \geq 2$, or if $U$ is superminimal and $n = 1$, then $\mathcal{L} := \mathcal{L}(U)$ satisfies the following:
\begin{enumerate}[(a)]
\item  $\mathcal{L} \subset Z$ is $\omega_{\mathrm{KE}}$-isotropic, $\omega_{\mathrm{NK}}$-isotropic, HV-compatible, and satisfies $\dim(T_z\mathcal{L} \cap \mathsf{V}) = 1$ at every $z \in \mathcal{L}$.
\item For any unit vector $V \in T_z\mathcal{L} \cap \mathsf{V}$, the $2k$-plane $T_z\mathcal{L} \cap \mathsf{H}$ is $J_V$-invariant.\end{enumerate}
\end{thm}

\begin{proof} Suppose $U$ is totally-complex if $n \geq 2$, or superminimal if $n = 1$, and set $\mathcal{L} := \mathcal{L}(U)$.  In either case, there exists a section $i \in \Gamma(Z|_U)$ such that $i(TU) = TU$ and $\nabla_Xi = 0$ for all $X \in TU$. \\
\indent (a) Following \cite[proof of Lemma 2.1]{ET-2005}, we orthogonally decompose $E|_U = \R i \oplus E'$.  If $\sigma \in \Gamma(E')$ is a local section, then $\langle \sigma, i \rangle = 0$, so that $\langle \nabla_X \sigma, i \rangle = 0$, and thus $\nabla_X \sigma \in \Gamma(E')$.  Thus, $E' \subset E|_U$ is a parallel subbundle.  Since $\mathcal{L} \subset E'$ is the unit sphere subbundle, it follows that $\mathcal{L} \subset E'$ is a parallel fiber subbundle.  This implies that $T\mathcal{L} = H_\Sigma \oplus V_\Sigma$ for subbundles $H_\Sigma \subset \mathsf{H}$ and $V_\Sigma \subset \mathsf{V}$, meaning that $\mathcal{L}$ is HV compatible. \\
\indent We now show that $\mathcal{L}$ is $\omega_{\mathrm{KE}}$-isotropic and $\omega_{\mathrm{NK}}$-isotropic.  Fix $z \in \mathcal{L}$, and recall that
\begin{align*}
\omega_{\mathrm{KE}} & = \omega_{\mathsf{H}} + \omega_{\mathsf{V}}, & \omega_{\mathrm{NK}} & = 2\omega_{\mathsf{H}} - \omega_{\mathsf{V}}.
\end{align*}
Since $\mathcal{L}$ is HV compatible and $\dim(T_z\mathcal{L} \cap \mathsf{V}) = 1$, it follows that $\left.\omega_{\mathsf{V}}\right|_{\mathcal{L}} = 0$.  Moreover, if $X, Y \in T_z\mathcal{L} \cap \mathsf{H}$, then
$$\omega_{\mathsf{H}}(X, Y) = g_{\mathrm{KE}}(J_{\mathrm{KE}}X, Y) = g_Q( \tau_* J_{\mathrm{KE}}X, \tau_*Y) = g_Q(z(\tau_*X), \tau_*Y) = 0,$$
where in the last step we used that $z(T_{\tau(z)}U) \subset (T_{\tau(z)}U)^\perp$.  This shows that $\left.\omega_{\mathsf{H}}\right|_\mathcal{L} = 0$, and therefore $\left.\omega_{\mathrm{KE}}\right|_{\mathcal{L}} = 0$ and $\left.\omega_{\mathrm{NK}}\right|_{\mathcal{L}} = 0$. \\

\indent (b) Fix $z \in \mathcal{L}$, let $u = \tau(z)$, and let $V \in T_z\mathcal{L} \cap \mathsf{V}$ be a vertical unit vector.  Let $j \in \mathcal{L}|_{u} \cap z^\perp$ denote the point on the great circle $\mathcal{L}|_u$ that corresponds to $V$ under the natural isomorphism $\mathsf{V}_z \simeq z^\perp$.  Set $i = z \circ j$, so that $(z,j,i)$ is an admissible frame of $E_u$.  By (\ref{eq:J2J3}), we have
$$J_V = (\left.\tau_*\right|_{\mathsf{H}_z})^{-1} \circ i \circ \tau_*.$$
Since $U$ is totally-complex, the $2k$-plane $T_uU \subset T_uQ$ is $i$-invariant.  Therefore, if $X \in T_z\mathcal{L} \cap \mathsf{H}$, then $i(\tau_*X) \in T_uU$, so that $J_VX = (\left.\tau_*\right|_{\mathsf{H}_z})^{-1}( i(\tau_*X)) \in T_z\mathcal{L} \cap \mathsf{H}$, proving that $T_z\mathcal{L} \cap \mathsf{H}$ is $J_V$-invariant.
\end{proof}

\subsubsection{Circle Bundle Lifts and CR Isotropic Submanifolds}

\indent \indent We now prove that circle bundle lifts $\mathcal{L}(U) \subset Z$ are intimately related to CR isotropic submanifolds $L \subset M$.  Indeed, the geometric properties of $\mathcal{L}(U)$ established in Theorem \ref{prop:NecessaryCircleBundle} are precisely those needed for its $p_1$-horizontal lift to be CR isotropic.  That is:

\begin{cor} \label{cor:LiftOfCircleBundle} Let $U^{2k} \subset Q^{4n}$ be a submanifold with $1 \leq k \leq n$.  If $U$ is totally-complex and $n \geq 2$, or if $U$ is superminimal and $n = 1$, then $\mathcal{L}(U) \subset Z$ admits local $p_1$-horizontal lifts to $M$, and every such lift is $(c_\theta I_2 + s_\theta I_3)$-CR isotropic for some $e^{i\theta} \in S^1$.
\end{cor}

\begin{proof} This follows from Theorem \ref{prop:NecessaryCircleBundle} and Proposition \ref{prop:LiftingDoubleIsotropic}(b).
\end{proof}

\indent We now aim to establish a converse in the case where $L$ is compact.  For this, we need a technical lemma.

\begin{lem} \label{lem:LocallyCircleBundle} Let $\Sigma^k \subset Z^{4n+2}$ be a compact submanifold.  If $\Sigma$ is $\omega_{\mathrm{KE}}$-isotropic, HV-compatible, and $\dim(T_z\Sigma \cap \mathsf{V}) = 1$ for all $z \in \Sigma$, then $U := \tau(\Sigma) \subset Q^{4n}$ is a $(k-1)$-dimensional submanifold, and $\left.\tau\right|_\Sigma \colon \Sigma \to U$ is an $S^1$-bundle whose fibers are geodesics in $Z$ with respect to the K\"ahler-Einstein metric.
\end{lem}

\begin{proof} Since $\Sigma$ is HV-compatible and $\dim(T_z\Sigma \cap \mathsf{V}) = 1$ for all $z \in \Sigma$, it follows that $\dim(T_z\Sigma \cap \mathsf{H}) = k-1$.  Therefore, the map $\left.\tau\right|_\Sigma \colon \Sigma \to Q$ has constant rank $k -1$.  By the Constant Rank Theorem, each fiber $\left.\tau\right|_\Sigma^{-1}(\tau(z)) \subset \Sigma$ is an embedded $1$-manifold, and therefore (since $\Sigma$ is compact) is an at most countable union of disjoint circles. \\
\indent We claim that each $S^1$-fiber is a geodesic.  For this, note that since $\Sigma$ is $\omega_{\mathrm{KE}}$-isotropic, it admits local $p_1$-horizontal lifts to $M$.  Let $L \subset M$ be such a lift.  Since $\Sigma$ is HV-compatible and $p_1 \colon M \to Z$ respects the horizontal-vertical splitting, we may write $TL = H_L \oplus \R \widetilde{V}$, where $H_L \subset \widetilde{\mathsf{H}}$ and $\widetilde{V} \in \widetilde{\mathsf{V}}$.  Moreover, Proposition \ref{prop:LiftingDoubleIsotropic}(a) implies that $L$ is $\alpha_1$-isotropic and $(-s_\theta \alpha_2 + c_\theta \alpha_3)$-isotropic for some constant $e^{i\theta} \in S^1$.  Therefore, $\widetilde{V} = c_\theta A_2 + s_\theta A_3$ is a Reeb vector field, so its integral curves are geodesics in $M$.  Consequently, the integral curves of $(p_1)_*(\widetilde{V}) \in \mathsf{V} \subset TZ$ are geodesics in $Z$ (and hence geodesics in $L$), and these are precisely the $S^1$-fibers $\left.\tau\right|_\Sigma^{-1}(\tau(z)) \subset \Sigma$. \\
\indent Consequently, since $\Sigma$ is compact, each $S^1$-fiber $\left.\tau\right|_\Sigma^{-1}(\tau(z)) \subset \Sigma$ is an at most countable union of disjoint \emph{great} circles in the twistor $2$-sphere.  Since any two great circles in a round $2$-sphere intersect, it follows that each $S^1$-fiber consists of a single great circle. \\
\indent It remains to show that $U := \tau(\Sigma)$ is a $(k-1)$-dimensional submanifold of $Q$.  For this, note that since $\Sigma$ is a union of great circles, each of which is the $p_1$-image of a Reeb circle in $M$, it admits a free $S^1$-action.  (The action is free because we are working on the regular part of $M$.) Therefore, the quotient $\Sigma/S^1$ admits the structure of smooth $(k-1)$-manifold, and the projection $\pi \colon \Sigma \to \Sigma/S^1$ is a smooth quotient map. \\
\indent Now, let $\widehat{\tau} \colon \Sigma \to U$ denote the map $\tau|_\Sigma$ with restricted codomain, equip $U \subset Q$ with the subspace topology, and let $\iota \colon U \hookrightarrow Q$ be the inclusion map.  If $V \subset U$ is open, then $V = U \cap W$ for some open set $W \subset Q$, and hence $\widehat{\tau}^{-1}(V) = \Sigma \cap \tau^{-1}(W)$ is open subset of $\Sigma$, which proves that $\widehat{\tau}$ is continuous.  Since $\widehat{\tau}$ is a continuous surjection from a compact domain, it follows that $\widehat{\tau}$ is a quotient map.  Since $\pi$ and $\widehat{\tau}$ are quotient maps that are constant on each other's fibers, there exists a unique homeomorphism $F \colon \Sigma/S^1 \to U$ such that $\widehat{\tau} = F \circ \pi$.  Choosing a smooth local section $\sigma \colon Y \to \Sigma$ of $\pi$, where $Y \subset \Sigma/S^1$ is an open set, we observe that $\tau|_\Sigma \circ \sigma \colon Y \to Q$ is a smooth map of rank $k-1$, which implies that $\iota \circ F \colon \Sigma/S^1 \to Q$ is also a smooth map of rank $k-1$, and therefore a smooth embedding whose image is $U$.
\end{proof}

The converse to Corollary \ref{cor:LiftOfCircleBundle} is now given by: 

\begin{thm} \label{thm:ProjCRIsoCircleBundle} Let $L^{2k+1} \subset M^{4n+3}$ be a compact submanifold, $1 \leq k \leq n$.  If $L$ is $(c_\theta I_2 + s_\theta I_3)$-CR isotropic for some $e^{i\theta} \in S^1$, and if $p_1(L) \subset Z$ is embedded, then $p_1(L) = \mathcal{L}(U)$ for some totally-complex submanifold $U^{2k} \subset Q^{4n}$ (respectively, a superminimal surface $U^2 \subset Q^4$ if $n = 1$).
\end{thm}

\begin{proof} Suppose that $L \subset M$ is a compact $(c_\theta I_2 + s_\theta I_3)$-CR isotropic $(2k+1)$-fold for some constant $e^{i\theta} \in S^1$, and that $\Sigma := p_1(L) \subset Z$ is embedded.  By Proposition \ref{prop:CRIsoProjection}(b), $\Sigma$ is $\omega_{\mathrm{KE}}$-isotropic, $\omega_{\mathrm{NK}}$-isotropic, HV-compatible, and $\dim(T_z\Sigma \cap \mathsf{V}) = 1$ for all $z \in \Sigma$.  Therefore, Lemma \ref{lem:LocallyCircleBundle} implies that $U := \tau(\Sigma) \subset Q$ is a $2k$-dimensional submanifold, and $\left.\tau\right|_\Sigma \colon \Sigma \to U$ is an $S^1$-bundle with geodesic fibers. \\
\indent Fix $z \in \Sigma$ and let $u = \tau(z)$.  Since $\Sigma$ is HV compatible and $\dim(T_z\Sigma \cap \mathsf{V}) = 1$, we can orthogonally split
$$T_z\Sigma = H_\Sigma \oplus \R V,$$
where $V \in \mathsf{V}_z$ is a vertical unit vector, and $H_\Sigma \subset \mathsf{H}_z$ is $2k$-dimensional.  On $\mathsf{H}_z$, let $\beta_V := \iota_{V}(\mathrm{Re}\,\gamma)$ denote the induced non-degenerate $2$-form, and let $J_V$ denote the corresponding complex structure.  By Proposition \ref{prop:CRIsoProjection}(b), the $2k$-plane $H_\Sigma \subset \mathsf{H}_z$ is $J_V$-invariant. \\ 
\indent Now, the $S^1$-fiber $\left.\tau\right|_\Sigma^{-1}(u) \subset \Sigma$ is a great circle through $z$ in the $2$-sphere $Z_u = \tau^{-1} (u)$.  Let $j \in \left.\tau\right|_\Sigma^{-1}(u) \cap z^\perp$ be the point on this circle that corresponds to $V$ under the natural isomorphism $\mathsf{V}_z \simeq z^\perp$.  Setting $i = z \circ j$, we see that $(z,j,i)$ is an admissible frame of $E_u$, which is the fiber over $u$ of the bundle $E$ from Definition~\ref{defn:QH}. See Figure~\ref{figure}. We also have $\left.\tau\right|_\Sigma^{-1}(u) = \{k \in Z_u \colon \langle k,i\rangle = 0 \}$, and
$$J_V = (\left.\tau_*\right|_{\mathsf{H}_z})^{-1} \circ i \circ \tau_*.$$
In particular, the $J_V$-invariance of the $2k$-plane $H_\Sigma \subset \mathsf{H}_z$ implies that $T_uU \subset T_uQ$ is $i$-invariant. \\
\indent Now, let $X_1, X_2 \in T_{u}U$, and let $\widetilde{X}_j = (\left.\tau_*\right|_{H_\Sigma})^{-1}(X_j) \in H_\Sigma$.  Since $\Sigma$ is $\omega_{\mathrm{KE}}$-isotropic, and $\omega_{\mathrm{KE}} = f^2 \wedge f^3 + \beta_1$, it follows that the $2k$-plane $H_\Sigma$ is $\beta_1$-isotropic.  Therefore,
\begin{align*}
g_Q(zX_1, X_2) = g_Q( \tau_*(J_1 \widetilde{X}_1), \tau_* \widetilde{X}_2) = g_{\mathrm{KE}}(J_1\widetilde{X}_1, \widetilde{X}_2) = \beta_1(\widetilde{X}_1, \widetilde{X}_2) = 0,
\end{align*}
which shows that $z(T_uU) \subset (T_uU)^\perp$.  Finally, if $X \in T_uU$, then $iX \in T_uU$, so $jX = -z(iX) \in (T_uU)^\perp$, demonstrating that $j(T_uU) \subset (T_uU)^\perp$.  This proves that $U$ is totally-complex, and that
$$\left.\tau\right|_\Sigma^{-1}(u) = \{k \in Z_u \colon \langle k,i\rangle = 0 \} = \{ k \in Z_u \colon k(T_uU) \subset (T_uU)^\perp\} = \left.\mathcal{L}(U)\right|_{u}.$$
\indent Finally, suppose that $n = 1$, so that $k = 1$.  Then $\Sigma^3 = \mathcal{L}(U)$ is $\omega_{\mathrm{KE}}$-Lagrangian and $\omega_{\mathrm{NK}}$-Lagrangian.  By a result of Storm \cite{Storm}, the surface $U \subset Q^4$ is superminimal.
\end{proof}

\begin{rmk} \label{rmk:embedded}
If $U$ is an embedded submanifold of $Q$, then its geodesic circle bundle is embedded in $Z$. Therefore, in order to characterize those submanifolds $\Sigma$ of $Z$ which are geodesic circle bundles in $Z$ we need to assume a priori that $\Sigma$ is embedded.
\end{rmk}

\subsubsection{Applications}

\indent \indent In previous sections, we considered $\mathrm{Re}(\gamma)$-calibrated $3$-folds $\Sigma^3 \subset Z$ that are $\omega_{\mathrm{KE}}$-isotropic, describing their $p_1$-horizontal lifts $L^3 \subset M^{4n+3}$ (Theorem \ref{thm:FirstCharDim3}).  Now, we are in a position to classify such $3$-folds in $Z$ as circle bundle lifts of totally-complex surfaces in $Q$.

\begin{thm} \label{thm:ReGammaCircleBundle} ${}$
\begin{enumerate}[(a)] 
\item If $U^2 \subset Q^{4n}$ is totally-complex and $n \geq 2$, or if $U$ is superminimal and $n = 1$, then $\mathcal{L}(U) \subset Z$ is $\mathrm{Re}(\gamma)$-calibrated and $\omega_{\mathrm{KE}}$-isotropic.
\item Conversely, if $\Sigma^3 \subset Z^{4n+2}$ is a compact $3$-dimensional submanifold that is $\mathrm{Re}(\gamma)$-calibrated and $\omega_{\mathrm{KE}}$-isotropic, then $\Sigma = \mathcal{L}(U)$ for some totally-complex surface $U^2 \subset Q^{4n}$.  Moreover, if $n = 1$, then $U$ is superminimal.
\end{enumerate}
\end{thm}

\begin{proof} (a) Let $U^2 \subset Q^{4n}$ be totally-complex if $n \geq 2$, or superminimal if $n = 1$.  By Theorem \ref{prop:NecessaryCircleBundle}(a), the $3$-fold $\mathcal{L}(U) \subset Z$ is $\omega_{\mathrm{KE}}$-isotropic.  Fix $z \in L$, and let $L \subset M$ denote a $p_1$-horizontal lift of a neighborhood of $z$.  By Corollary \ref{cor:LiftOfCircleBundle}, $L$ is $(c_\theta I_2 + s_\theta I_3)$-CR isotropic.  Therefore, by Theorem \ref{thm:FirstCharDim3}((ii)$\implies$(iv)), $p_1(L) \subset \mathcal{L}(U)$ is $\mathrm{Re}(\gamma)$-calibrated.

\begin{figure}
\centering
$$\begin{tikzpicture}

  % Define radius
  \def\r{3}

  % Points
  \draw (0, 0) node[circle, fill, inner sep=1] (orig) {};
  \draw (\r, 0) node[circle, fill, inner sep=1] (jpt)  {};
  \draw (-\r/5, -\r/3) node[circle, fill, inner sep=1] (zpt) {};
  \draw (0, \r) node[circle, fill, inner sep=1] (ipt) {};

  % Sphere + Equator
  \draw (orig) circle (\r);
  \draw[dashed] (orig) ellipse (\r{} and \r/3);

  % Vectors + Labels
  \draw[] (orig) ++(-\r/5, -\r/3) node[below] (x1) {$z$};
  \draw[] (orig) ++(\r, 0) node[right] (x2) {$j$};
  \draw[] (orig) ++(0, \r) node[above] (x3) {$i$};
  \draw[->] (zpt) -- ++(\r, 0) node[right] (x4) {$f_2$};
  \draw[->] (zpt) -- ++(0, \r) node[right] (x5) {$f_3$};

\end{tikzpicture}$$
\caption{The admissible frame $(z, j, i)$ of $E_u$.} \label{figure}
\end{figure}
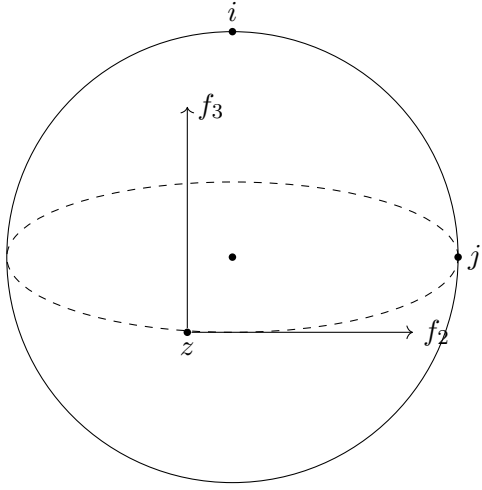

\indent (b) Suppose $\Sigma^3 \subset Z$ is a compact $3$-dimensional submanifold that is $\mathrm{Re}(\gamma)$-calibrated and $\omega_{\mathrm{KE}}$-isotropic.  By Proposition \ref{prop:HV-KE-Iso}(c), $\Sigma$ is HV-compatible and $\dim(T_z\Sigma \cap \mathsf{V}) = 1$ for all $z \in \Sigma$.  Therefore, Lemma \ref{lem:LocallyCircleBundle} implies that $U^2 = \tau(\Sigma) \subset Q$ is a $2$-dimensional surface, and that $\left.\tau\right|_\Sigma \colon \Sigma \to U$ is an $S^1$-bundle with geodesic fibers. \\
\indent Fix $z \in \Sigma$, and let $u = \tau(z)$.  We may write $T_z\Sigma = H_\Sigma \oplus V_\Sigma$ for some $2$-plane $H_\Sigma \subset \mathsf{H}$ and line $V_\Sigma \subset \mathsf{V}$.  Let $(e_{10}, \ldots, e_{n3}, f_2, f_3)$ be an $\Sp(n)\U(1)$-frame at $z$, with dual coframe $(\rho_{10}, \ldots, \rho_{n3}, \mu_2, \mu_3)$, such that
\begin{align} \label{eq:vertical-frame2}
V_\Sigma & = \mathrm{span}(f_2), & \vol_{V_\Sigma} & = \mu_2.
\end{align}
Let $(\beta_1, \beta_2, \beta_3) = (\omega_{\mathsf{H}},\, \iota_{f_2}(\mathrm{Re}\,\gamma), \, \iota_{f_3}(\mathrm{Re}\,\gamma))$ denote the induced hyperk\"{a}hler triple on $\mathsf{H}_z$, and let $(J_1, J_2, J_3)$ be the corresponding complex structures on $\mathsf{H}_z$. \\
\indent Now, the $S^1$-fiber $\left.\tau\right|_\Sigma^{-1}(u) \subset \Sigma$ is a great circle through $z$ in the twistor $2$-sphere $Z_u$.  Let $j \in \left.\tau\right|_\Sigma^{-1}(u) \cap z^\perp$ be the point on this circle that corresponds to $V$ under the natural isomorphism $\mathsf{V}_z \simeq z^\perp$.  Setting $i = z \circ j$, we see that $(z,j,i)$ is an admissible frame of $E_u$ (see Figure~\ref{figure}), that $\left.\tau\right|_\Sigma^{-1}(u) = \{k \in Z_u \colon \langle k,i\rangle = 0 \}$, and moreover,
\begin{align*}
J_2 & = (\left.\tau_*\right|_{\mathsf{H}_z})^{-1} \circ i \circ \tau_* & J_3 & = -(\left.\tau_*\right|_{\mathsf{H}_z})^{-1} \circ j \circ \tau_*.
\end{align*}
Using (\ref{eq:vertical-frame2}), we compute
\begin{align*}
\left.\mu_2\right|_{V_\Sigma} \wedge \vol_{H_\Sigma} = \vol_{T_z\Sigma} = \left.\mathrm{Re}(\gamma)\right|_{T_z\Sigma} & = \left.\left(\mu_2 \wedge \beta_2 + \mu_3 \wedge \beta_3\right)\right|_{T_z\Sigma} \\
& = \left.\mu_2\right|_{V_\Sigma} \wedge \left.\beta_2\right|_{H_\Sigma} + \left.\mu_3\right|_{V_\Sigma} \wedge \left.\beta_3\right|_{H_\Sigma} \\
& = \left.\mu_2\right|_{V_\Sigma} \wedge \left.\beta_2\right|_{H_\Sigma}.
\end{align*}
Contracting with $f_2$ gives $\beta_2|_{H_\Sigma} = \vol_{H_\Sigma}$, which implies that the real $2$-plane $H_\Sigma \subset \mathsf{H}_z$ is $J_2$-invariant.  Consequently, $T_{u}U \subset T_{u}Q$ is $i$-invariant. \\
\indent Repeating the argument at the end of the proof of Theorem \ref{thm:ProjCRIsoCircleBundle}, we observe that $z(T_uU) \subset (T_uU)^\perp$ and $j(T_uU) \subset (T_uU)^\perp$.  This proves that $U$ is totally-complex, and that
$$\left.\tau\right|_\Sigma^{-1}(u) = \{k \in Z_u \colon \langle k,i\rangle = 0 \} = \{ k \in Z_u \colon k(T_uU) \subset (T_uU)^\perp\} = \left.\mathcal{L}(U)\right|_{u}.$$
\indent Finally, suppose that $n = 1$.  Since $\Sigma^3 = \mathcal{L}(U) \subset Z^6$ is $\mathrm{Re}(\gamma)$-calibrated, it follows from Proposition~\ref{prop:Re-gamma-phases} that $\Sigma$ is $\omega_{\mathrm{NK}}$-Lagrangian. Thus, $\mathcal{L}(U)$ is both $\omega_{\mathrm{KE}}$-Lagrangian and $\omega_{\mathrm{NK}}$-Lagrangian, so the superminimality of $U^2 \subset Q^4$ follows from Storm's theorem \cite{Storm}.
\end{proof}

\indent We can now classify the compact submanifolds of $Z$ that are Lagrangian with respect to both $\omega_{\mathrm{KE}}$ and $\omega_{\mathrm{NK}}$.

\begin{thm} \label{thm:Double-Lagrangian-Char} ${}$
\begin{enumerate}[(a)]
\item If $U^{2n} \subset Q^{4n}$ is totally-complex and $n \geq 2$, or if $U$ is superminimal and $n = 1$, then $\mathcal{L}(U) \subset Z$ is $\omega_{\mathrm{KE}}$-Lagrangian and $\omega_{\mathrm{NK}}$-Lagrangian.
\item Conversely, if $\Sigma^{2n+1} \subset Z^{4n+2}$ is a compact $(2n+1)$-dimensional submanifold that is both $\omega_{\mathrm{KE}}$-Lagrangian and $\omega_{\mathrm{NK}}$-Lagrangian, then $\Sigma = \mathcal{L}(U)$ for some (maximal) totally-complex $2n$-fold $U^{2n} \subset Q^{4n}$.
\end{enumerate}
\end{thm}

\begin{proof} (a) This follows from Theorem \ref{prop:NecessaryCircleBundle}(a). \\
\indent (b) Suppose $\Sigma^{2n+1} \subset Z$ is a compact submanifold that is both $\omega_{\mathrm{KE}}$-Lagrangian and $\omega_{\mathrm{NK}}$-Lagrangian.  By Proposition \ref{prop:HV-KE-Iso}(b), $\Sigma$ is HV compatible, $\dim(T_z\Sigma \cap \mathsf{H}) = 2n$, and $\dim(T_z\Sigma \cap \mathsf{V}) = 1$ for all $z \in \Sigma$.  By Lemma \ref{lem:LocallyCircleBundle}, $U := \tau(\Sigma) \subset Q$ is a $2n$-dimensional submanifold, and $\left.\tau\right|_\Sigma \colon \Sigma \to U$ is an $S^1$-bundle with geodesic fibers. \\
\indent It remains to prove that $U$ is totally-complex, and that $\tau|_\Sigma^{-1}(u) = \left.\mathcal{L}(U)\right|_u$.  For this, note that Corollary \ref{cor:DoubleLagrangian-CRLift}(b) implies that every local $p_1$-horizontal lift of $\Sigma$ is $(c_\theta I_2 + s_\theta I_3)$-CR Legendrian for some $e^{i\theta} \in S^1$.  The proof now follows exactly as in Theorem \ref{thm:ProjCRIsoCircleBundle}.
\end{proof}

\section{Characterizations of Complex Lagrangian Cones} \label{sec:CharCplxLagCones}

\indent \indent In a hyperk\"{a}hler cone $C^{4n+4}$, recall that a $(2k+2)$-dimensional cone $\Cone(L)$ is $(c_\theta I_2 + s_\theta I_3)$-complex isotropic provided that it satisfies the following three conditions:
\begin{align*}
\left.\omega_1\right|_{\Cone(L)} & = 0, & \left.\left(-s_\theta \omega_2 + c_\theta \omega_3\right)\right|_{\Cone(L)} & = 0, & (c_\theta I_2 + s_\theta I_3)\text{-complex.}
\end{align*}
As discussed in $\S$\ref{subsub:CRIso}, this is equivalent to requiring that the $(2k+1)$-dimensional link $L$ be $(c_\theta I_2 + s_\theta I_3)$-CR isotropic, meaning
\begin{align*}
\left.\alpha_1\right|_L & = 0, & \left.\left(-s_\theta \alpha_2 + c_\theta \alpha_3\right)\right|_L & = 0, & (c_\theta I_2 + s_\theta I_3)\text{-CR}.
\end{align*}
In this short section, we characterize complex isotropic cones $\Cone(L)^{2k+2} \subset C^{4n+4}$, $1 \leq k \leq n$, in terms of related geometries in $M^{4n+3}$, $Z^{4n+2}$, and $Q^{4n}$. \\

\indent To begin, we generalize a result of Ejiri--Tsukada \cite{ET-2012} --- originally established for complex Lagrangian cones (i.e., $k = n$) in the flat model $C^{4n+4} = \HH^{n+1}$ --- to complex isotropic cones of any dimension $2k+2$ in arbitrary hyperk\"{a}hler cones $C^{4n+4}$.

\begin{thm} \label{thm:MainArbDim} Let $L^{2k+1} \subset M^{4n+3}$, where $3 \leq 2k+1 \leq 2n+1$.  The following conditions are equivalent:
\begin{enumerate}[(1)]
\item $\mathrm{C}(L)$ is $(c_\theta I_2 + s_\theta I_3)$-complex isotropic for some constant $e^{i \theta} \in S^1$.
\item $L$ is $(c_\theta I_2 + s_\theta I_3)$-CR isotropic for some constant $e^{i \theta} \in S^1$.
\item $L$ is locally of the form $p_v^{-1}(V)$ for some horizontal $J_{\mathrm{KE}}$-complex submanifold $V^{2k} \subset Z$ and some $v = (0, c_\theta, s_\theta)$.
\item  $L$ is locally of the form $p_v^{-1}(\widetilde{U})$ for some totally-complex submanifold $U^{2k} \subset Q$ (resp., superminimal surface if $n = 1$) and some $v = (0, c_\theta, s_\theta)$.
\end{enumerate}
If, in addition, $L$ is compact and $p_1(L) \subset Z$ is embedded, then the above conditions are equivalent to:
\begin{enumerate}[($\star$)]
\item $L$ is a $p_1$-horizontal lift of $\mathcal{L}(U) \subset Z$ for some totally-complex submanifold $U^{2k} \subset Q^{4n}$ (resp., superminimal surface $U^2 \subset Q^4$ if $n = 1$).
\end{enumerate}
\end{thm}

\begin{proof} The equivalence (1)$\iff$(2) is Proposition \ref{prop:ConeComplexLag}.  The equivalence (2)$\iff$(3) is Corollary \ref{cor:CRIsoChar1}.  The equivalence (3)$\iff$(4) follows from Theorem \ref{prop:HorzLiftTotCplx}. \\
\indent ($\star$) $\implies (2)$. This is Corollary \ref{cor:LiftOfCircleBundle}. \\
\indent (2)$\implies$($\star$). This is Theorem \ref{thm:ProjCRIsoCircleBundle}.
\end{proof}

\indent Therefore, given a $(c_\theta I_2 + s_\theta I_3)$-complex isotropic cone $\mathrm{C}(L) \subset C$, its link $L \subset M$ can be viewed in two ways.  On the one hand, $L$ is a $p_{(1,0,0)}$-horizontal lift of a circle bundle over a totally-complex submanifold $U \subset Q$.  On the other hand, $L$ is also a $p_{(0,c_\theta, s_\theta)}$-circle bundle over a $\tau$-horizontal lift of a totally-complex submanifold $U \subset Q$.  Thus, loosely speaking, the operations of ``horizontal lift" and ``circle bundle lift" form a commutative diagram of sorts:
$$\begin{tikzcd}
M^{4n+3} \arrow[d, "{p_{(1,0,0)}}"', bend right] \arrow[d, "{p_{(0, c_\theta, s_\theta)}}", bend left] &                                  & {\widehat{\mathcal{L}(U)} = L = p_{(0,c_\theta, s_\theta)}^{-1}(\widetilde{U})} &                               \\
Z^{4n+2} \arrow[d, "\tau"]                                                                             & \mathcal{L}(U)^{2k+1} \arrow[ru] &                                                                                 & \widetilde{U}^{2k} \arrow[lu] \\
Q^{4n}                                                                                                 &                                  & U^{2k} \arrow[lu] \arrow[ru]                                                    &                              
\end{tikzcd}$$

\indent For complex Lagrangian cones in $C^{4n+4}$, we are able to say more.

\begin{thm} \label{thm:TopDimension} Let $L^{2n+1} \subset M^{4n+3}$ be a $(2n+1)$-dimensional submanifold.  The following five conditions are equivalent:
\begin{enumerate}[(1)]
\item $\mathrm{C}(L)$ is $(c_\theta I_2 + s_\theta I_3)$-complex Lagrangian for some constant $e^{i \theta} \in S^1$.
\item $L$ is $(c_\theta I_2 + s_\theta I_3)$-CR Legendrian for some constant $e^{i \theta} \in S^1$.
\item $L$ is locally of the form $p_v^{-1}(V)$ for some horizontal $J_{\mathrm{KE}}$-complex submanifold $V^{2n} \subset Z$ and some $v = (0, c_\theta, s_\theta)$.
\item  $L$ is locally of the form $p_v^{-1}(\widetilde{U})$ for some totally-complex submanifold $U^{2n} \subset Q$ (resp., superminimal surface if $n = 1$) and some $v = (0, c_\theta, s_\theta)$.\
\item $L$ is locally a $p_1$-horizontal lift of a $(2n+1)$-fold $\Sigma^{2n+1} \subset Z$ that is $\omega_{\mathrm{KE}}$-Lagrangian and $\omega_{\mathrm{NK}}$-Lagrangian.
\end{enumerate}
If, in addition, $L$ is compact and $p_1(L) \subset Z$ is embedded, then the above conditions are equivalent to:
\begin{enumerate}[($\star$)]
\item $L$ is a $p_1$-horizontal lift of $\mathcal{L}(U) \subset Z$ for some totally-complex submanifold $U^{2n} \subset Q^{4n}$ (resp., superminimal surface $U^2 \subset Q^4$ if $n = 1$).
\end{enumerate}
\end{thm}

\begin{proof} The equivalence (1)$\iff$(2)$\iff$(3)$\iff$(4)$\iff$($\star$) was proven in Theorem \ref{thm:MainArbDim}.  It remains only to involve condition (5).  For this, note that (5)$\iff$(2) is the content of Corollary \ref{cor:DoubleLagrangian-CRLift}.  Alternatively, (5)$\iff$($\star$) is Theorem \ref{thm:Double-Lagrangian-Char}.
\end{proof}

Finally, for $4$-dimensional complex isotropic cones in $C^{4n+4}$, even more characterizations are available:

\begin{thm} \label{thm:LowDimension} Let $L^{3} \subset M^{4n+3}$ be a $3$-dimensional submanifold.  The following six conditions are equivalent:
\begin{enumerate}[(1)]
\item $\mathrm{C}(L)$ is $(c_\theta I_2 + s_\theta I_3)$-complex isotropic for some constant $e^{i \theta} \in S^1$.
\item $L$ is $(c_\theta I_2 + s_\theta I_3)$-CR isotropic for some constant $e^{i \theta} \in S^1$.
\item $L$ is locally of the form $p_v^{-1}(V)$ for some horizontal $J_{\mathrm{KE}}$-complex submanifold $V^{2} \subset Z$ and some $v = (0, c_\theta, s_\theta)$.
\item  $L$ is locally of the form $p_v^{-1}(\widetilde{U})$ for some totally-complex submanifold $U^{2} \subset Q$ (resp., superminimal surface if $n = 1$) and some $v = (0, c_\theta, s_\theta)$.
\item $L$ is locally a $p_1$-horizontal lift of a $\mathrm{Re}(\gamma)$-calibrated $3$-fold that is $\omega_{\mathrm{KE}}$-isotropic.
\item $L$ is $\mathrm{Re}(\Gamma_1)$-calibrated.
\end{enumerate}
If, in addition, $L$ is compact and $p_1(L) \subset Z$ is embedded, then the above conditions are equivalent to:
\begin{enumerate}[($\star$)]
\item $L$ is a $p_1$-horizontal lift of $\mathcal{L}(U) \subset Z$ for some totally-complex submanifold $U^{2} \subset Q^{4n}$ (resp., superminimal surface $U^2 \subset Q^4$ if $n = 1$).
\end{enumerate}
\end{thm}

\begin{proof} Theorem \ref{thm:FirstCharDim3} gives (1)$\iff$(2)$\iff$(3)$\iff$(5)$\iff$(6).  Now, as Theorem \ref{thm:MainArbDim} proves (1)$\iff$(2)$\iff$(3)$\iff$(4)$\iff$($\star$), we deduce the result.  Alternatively, Theorem \ref{prop:HorzLiftTotCplx} gives (3)$\iff$(4), and Theorem \ref{thm:ReGammaCircleBundle} gives (5)$\iff$($\star$).
\end{proof}

\begin{appendix}

\section{Appendix}

\subsection{Linear Algebra of Calibrations} \label{appendix:calib}

\indent \indent Let $(V, g)$ be an $n$-dimensional oriented real inner product space. Recall that a $k$-form $\gamma$ on $V$ is said to have \emph{comass one} if $\gamma(P) \leq 1$ for any oriented orthonormal $k$-plane $P$ in $V$, with equality on at least one such $P$. Equivalently, by writing $P = e_1 \wedge \cdots \wedge e_k$, this means that
$$ \gamma(e_1, \ldots, e_k) \leq 1 $$
whenever $e_1, \ldots, e_k$ are orthonormal in $V$, with equality on at least one such set. Throughout this paper, a $k$-form with comass one will be called a \emph{semi-calibration}. Let $\gamma \in \Lambda^k (V^*)$ be a semi-calibration. An oriented $k$-plane $P$ is called $\gamma$-calibrated if $\gamma(P) = 1$. 

It is easy to see that $\gamma \in \Lambda^k (V^*)$ is a semi-calibration if and only if $\ast \gamma \in \Lambda^{n-k} (V^*)$ is a semi-calibration, where $\ast$ is the Hodge star operator induced by the inner product and orientation on $V$. We collect here some results on semi-calibrations that we will need.

\begin{prop} \label{prop:rich-calibs}
Let $\gamma \in \Lambda^k (V^*)$, be a semi-calibration, and let $L \subset V$ be an oriented $1$-dimensional subspace with oriented orthonormal basis $\{ e_1 \}$. Write $V = L \oplus L^{\perp}$, and
$$ \gamma = e_1^\flat \wedge \alpha + \beta, $$
where $\alpha = \iota_{e_1} \gamma \in \Lambda^{k-1} (L^{\perp})^*$ and $\beta = \gamma - e_1^\flat \wedge \alpha \in \Lambda^k (L^{\perp})^*$.
\begin{enumerate}[(a)]
\item If every oriented line in $V$ lies in a $\gamma$-calibrated $k$-plane, then $\alpha$ is a semi-calibration.
\item Suppose (a) holds. Then an oriented $(k-1)$-plane $W$ in $L^{\perp}$ is $\alpha$-calibrated if and only if the oriented $k$-plane $P = L \oplus W$ is $\gamma$-calibrated.
\item If every oriented line in $V$ lies in a $(\ast\gamma)$-calibrated $(n-k)$-plane, then $\beta$ is a semi-calibration. 
\end{enumerate}
\end{prop}
\begin{proof}
Let $W$ be an oriented $(k-1)$-plane in $L^{\perp}$, where $W = e_2 \wedge \cdots \wedge e_k$ for some oriented orthonormal bases $e_2, \ldots, e_k$ of $W$. Then
\begin{equation} \label{eq:comass-temp}
\alpha(W) = \alpha(e_2, \ldots, e_k) = \gamma(e_1, e_2, \ldots, e_k) = \gamma(L \oplus W).
\end{equation}
Since $\gamma(L \oplus W) \leq 1$, the comass of $\alpha$ is at most $1$. By hypothesis, there exists a $\gamma$-calibrated $k$-plane $P$ containing $L$. Let $W$ be the unique oriented $(k-1)$-plane in $L^{\perp}$ that $P = L \oplus W$. Then $\alpha(W) = \gamma(L \oplus W) = \gamma(P) = 1$, so $\alpha$ is a semi-calibration. This proves (a), and then (b) is immediate from~\eqref{eq:comass-temp}. For (c), observe that
$$ \ast \gamma = \ast (e_1^\flat \wedge \alpha + \beta) = \ast_{L^{\perp}} \alpha + (-1)^k e_1^\flat \wedge \ast_{L^{\perp}} \beta. $$
If every oriented line $L$ lies in a $(\ast \gamma)$-calibrated $(n-k)$-plane, then (a) holds for $\ast \gamma$, so $ \iota_{e_1} (\ast \gamma) = (-1)^k \ast_{L^{\perp}} \beta$ is a semi-calibration on $L^{\perp}$, but then so is $\beta$.
\end{proof}

\begin{prop} \label{prop:CalibrationSplitting} Let $\gamma$ be a semi-calibration on $V$, and suppose we have an orthogonal splitting $V = L \oplus L^{\perp}$ for some oriented line $L$, with oriented orthonormal basis $\{ e_1 \}$. If $\iota_{e_1} \gamma = 0$, then any $\gamma$-calibrated $k$-plane lies in $L^{\perp}$.
\end{prop}
\begin{proof}
It is trivial that $\dim (P \cap L^{\perp}) \geq k-1$. Therefore we can find an oriented orthonormal basis $v_1, w_2, \ldots, w_k$ of $P$ such that $v_1 = \cos (\theta) e_1 + \sin (\theta) w_1$ and $w_1, \ldots, w_k \in L^{\perp}$ orthonormal. Then since $\iota_{e_1} \gamma = 0$, we have
$$ 1 = \gamma(v_1, w_2, \ldots, w_k) = \sin (\theta) \, \gamma(w_1, w_2, \ldots, w_k) \leq \sin (\theta). $$
Thus $\sin (\theta) = 1$, and $v_1 = w_1 \in P$.
\end{proof}

\begin{prop} \label{prop:scaling-calib}
Let $(W, g)$ be a finite-dimensional real inner product space, and suppose we have an orthogonal splitting $W = H \oplus V$, so that the inner product is given by $g = g_H + g_V$. Define a new inner product $\tilde g$ on $V$ by $\tilde g = t^2 g_H + g_V$. Let $\gamma$ be a semi-calibration on $V$ such that $\gamma \in \Lambda^m (H^*) \otimes \Lambda^{k-m} (V^*)$. Then $t^m \gamma$ is a semi-calibration on $(W, \tilde g)$.
\end{prop}
\begin{proof}
Let $\tilde e_1, \ldots, \tilde e_k$ be orthonormal for $\tilde g$. We can decompose $\tilde e_j = h_j + v_j$ where $h_j \in H$ and $v_j \in V$, so
$$ \delta_{ij} = \tilde g (e_i, e_j) = t^2 g(h_i, h_j) + g(v_i, v_j). $$
Thus if we define $e_j = t h_j + v_j$, then $e_1, \ldots, e_k$ are orthonormal for $g$. Using the fact that $\gamma \in \Lambda^m (H^*) \otimes \Lambda^{k-m} (V^*)$, we have
$$ (t^m \gamma)(\tilde e_1, \ldots, \tilde e_k) = t^m \gamma(h_1 + v_1, \ldots, h_k + v_k) $$
is a sum of terms, each of which has exactly $m$ of the $h_j$'s and $k-m$ of the $v_j$'s in the argument of $t^m \gamma$. By multilinearity, we can bring one factor of $t$ in to each of the $h_j$ arguments, to get
$$ \gamma( t h_1 + v_1, \ldots, t h_k + v_k ) = \gamma (e_1, \ldots, e_k) \leq 1. $$
Thus $t^m \gamma$ has comass at most one with respect to $\tilde g$. But now it is clear that if $P = e_1 \wedge \cdots \wedge e_k$ is $\gamma$-calibrated with respect to $g$, then $\tilde P = \tilde e_1 \wedge \cdots \wedge \tilde e_k$ is $t^m \gamma$-calibrated with respect to $\tilde g$, where $\tilde e_j = t^{-1} h_j + v_j$ if $e_j = h_j + v_j$.
\end{proof}

\begin{prop} \label{prop:submersion-calibration}
Let $(V, g)$ and $(W, h)$ be finite-dimensional real inner product spaces, and let $p \colon V \to W$ be a Riemannian submersion. That is, $p$ is a linear surjection that maps $(\Ker\, p)^{\perp} \subset V$ \emph{isometrically} onto $W$. If $\alpha \in \Lambda^k (W^*)$ is a semi-calibration on $(W, h)$, then $p^* \alpha$ is a semi-calibration on $(V, g)$.
\end{prop}
\begin{proof}
Let $v_1, \ldots, v_k$ be orthonormal vectors in $V$. We can orthogonally decompose $v_j = u_j + w_j$ where $u_j \in \Ker\, p$ and $w_j \in (\Ker\, p)^{\perp}$. Using that $\alpha$ is a semi-calibration, $p \colon ((\Ker\, p)^{\perp}, g) \to (W, h)$ is an isometry, and Hadamard's inequality, we have
\begin{align*}
(p^* \alpha) (v_1, \ldots, v_k) & = (p^* \alpha)(u_1 + w_1, \ldots, u_k + w_k) = \alpha(p(w_1), \ldots, p(w_k)) \\
& \leq | p(w_1) \wedge \cdots \wedge p(w_k) | \leq | p(w_1) | \cdots | p(w_k) | = |w_1| \cdots |w_k| \leq 1.
\end{align*}
Thus the comass of $p^* \alpha$ is at most one. Let $L \subset W$ be an oriented $k$-plane calibrated by $\alpha$, with oriented orthonormal basis $e_1, \ldots, e_k$. For $1 \leq j \leq k$, let $w_j$ be the unique vector in $(\Ker\, p)^{\perp}$ such that $p(w_j) = e_j$. Then it is clear that $w_1 \wedge \cdots \wedge w_k \subset V$ is calibrated by $p^* \alpha$.
\end{proof}

\begin{prop} \label{prop:IsotropyLemma} Let $(V, g, \omega, I)$ be a Hermitian vector space of real dimension $2n$, where $I$ is the complex structure and $\omega = g(I \cdot, \cdot)$ is the associated real $(1,1)$-form. Let $\gamma \in \Lambda^k(V^*)$ be of type $(k,0) + (0,k)$, where $k \leq n$. If $P \subset V$ is an oriented $k$-plane on which $\gamma$ attains its maximum, then $P$ is $\omega$-isotropic. That is, $\omega|_P = 0$.
\end{prop}
\begin{proof} Let $P \subset V$ be an oriented $k$-plane, and write $k = 2m+1$ if $k$ is odd, and $k = 2m$ if $k$ is even. By~\cite[Lemma 7.18]{Harvey}, which actually works for any $k$, there exists an orthonormal basis $(e_1, Ie_1, \ldots, e_n, Ie_n)$ of $V$ and constants $\theta_1, \ldots, \theta_m \in [0,2\pi)$ such that
\begin{align*}
P & = e_1 \wedge (\sin(\theta_1) Ie_1 + \cos(\theta_1) e_2) \wedge \cdots \wedge ( \sin(\theta_m) I e_{2m-1} + \cos(\theta_m) e_{2m} ) \wedge e_{2m+1}, & & (k = 2m+1), \\
P & = e_1 \wedge (\sin(\theta_1) Ie_1 + \cos(\theta_1) e_2) \wedge \cdots \wedge ( \sin(\theta_m) I e_{2m-1} + \cos(\theta_m) e_{2m} ), & & (k = 2m).
\end{align*}
Since $\gamma$ is of type $(k,0) + (0,k)$, we have $\iota_{e_i}(\iota_{Ie_i}\gamma) = 0$. Therefore, we have
$$\gamma(P) = \cos(\theta_1) \cdots \cos(\theta_m)\,\gamma(e_1, \ldots, e_k).$$
Since $\gamma$ attains its maximum at $P$, it follows that $\theta_1 = \theta_2 = \cdots = \theta_m = 0$. Therefore, $P = e_1 \wedge \cdots \wedge e_k$. In particular, if $v \in P$ then $I v \in P^{\perp}$. Hence $P$ is $\omega$-isotropic.
\end{proof}

\begin{thm} \label{thm:special-isotropic-comass-one}
Let $(V, g, \omega_1, \omega_2, \omega_3, I_1, I_2, I_3)$ be a quaternionic-Hermitian vector space of real dimension $4n$, where $\omega_p = g(I_p \cdot, \cdot)$ is the associated real $2$-form of $I_p$-type $(1,1)$. Let $\sigma = \omega_2 + i \omega_3$. It is easy to check that $\sigma$ is of $I_1$-type $(2,0)$. Let $\Theta_{2k} = \mathrm{Re} (\frac{1}{k!} \sigma^k) \in \Lambda^{2k} (V^*)$. Then $\Theta_{2k}$ has comass one.
\end{thm}
\begin{proof}
We prove this by induction on $k$, for any $n$. The case $k=1$ is clear, because then $\Theta_{2} = \omega_2$. Note also that if $\Theta_{2k} = \mathrm{Re} (\frac{1}{k!} \sigma^k)$ has comass one, then so does $\mathrm{Re} (e^{- i \theta} \frac{1}{k!} \sigma^k)$ for any $e^{i \theta} \in S^1$, since this just corresponds to rotating the complex structures $I_2, I_3$ by $\theta$, and thus again corresponds to a quaternionic-Hermitian structure. Thus we can assume that $k \geq 2$ and that both $\mathrm{Re} (\frac{1}{(k-1)!} \sigma^{k-1})$ and $\mathrm{Im} (\frac{1}{(k-1)!} \sigma^{k-1})$ have comass one for any quaternionic dimension $n$.

Let $P$ be an oriented $2k$-plane on which $\Theta_{2k}$ attains its maximum. Since $\Theta_{2k}$ is of $I_1$-type $(2k,0) + (0,2k)$, we can apply Proposition~\ref{prop:IsotropyLemma} to deduce that $P$ is $I_1$-isotropic. In particular, $P$ does not contain any $I_1$-complex lines. Let $e_1$ be a unit vector in $P$. Complete $e_1$ to a quaternionic orthonormal basis
$$ \{ e_1, I_1 e_1, I_2 e_1, I_3 e_1, \ldots, e_n, I_1 e_n, I_2 e_n, I_3 e_n \}, $$
so that
$$ \omega_1 = \sum_{j=1}^n (e_j \wedge I_1 e_j + I_2 e_j \wedge I_3 e_j), $$
and similarly for $\omega_2, \omega_3$ by cyclically permuting $1, 2, 3$ above. In particular, we have
\begin{equation} \label{eq:SI-temp}
\iota_{e_1} \sigma = I_2 e_1 + i I_3 e_1.
\end{equation}

Write $P = e_1 \wedge Q$ for an oriented $(2k-1)$-plane, so
\begin{align} \label{eq:SI-temp2}
\Theta_{2k}(P) & = \Theta_{2k}(e_1 \wedge Q) = (\iota_{e_1} \Theta_{2k})(Q).
\end{align}
Moreover, we have
$$ Q \subset (\mathrm{span}(e_1, I_1 e_1))^{\perp} = W \oplus \widetilde V, $$
where
$$ W = \mathrm{span} (I_2 e_1, I_3 e_1) \quad \text{is an $I_1$-complex line}, $$
and
$$ \widetilde V = \mathrm{span} (e_2, I_1 e_2, I_2 e_2, I_3 e_2, \ldots, e_n, I_1 e_n, I_2 e_n, I_3 e_n \} $$
is a quaternionic-Hermitian subspace of real dimension $4(n-1)$. In particular, our induction hypothesis tells us that both $\mathrm{Re} (\frac{1}{(k-1)!} \sigma^{k-1})$ and $\mathrm{Im} (\frac{1}{(k-1)!} \sigma^{k-1})$ have comass one on $\widetilde V$.

We observe from $Q + \widetilde V \subset W \oplus \widetilde V$ that
\begin{align*}
\dim(Q \cap \widetilde V) & = \dim Q + \dim \widetilde V - \dim (Q + \widetilde V) \\
& \geq (2k-1) + (4n-4) - (4n-2) = 2k-3,
\end{align*}
so we can write $Q = u_2 \wedge u_3 \wedge v_4 \wedge \cdots \wedge v_{2k}$ for an oriented orthonormal basis $\{ u_2, u_3, v_4, \ldots, v_{2k} \}$ of $Q$, where $v_4, \ldots, v_{2k} \in \widetilde V$. We also have
$$ u_2 = \cos(\phi) w_2 + \sin(\phi) v_2, \qquad u_3 = \cos(\psi) w_3 + \sin(\psi) v_3, $$
for some unit vectors $w_2, w_3 \in W$ and $v_2, v_3 \in \widetilde V$. Abbreviating $R = v_4 \wedge \cdots \wedge v_{2k}$, $\cos(\phi) = c_{\phi}$ and similarly, we have
\begin{align*}
Q & = u_2 \wedge u_3 \wedge R = (c_{\phi} w_2 + s_{\phi} v_2) \wedge (c_{\psi} w_3 + s_{\psi} v_3) \wedge R \\
& = c_{\phi} c_{\psi} w_2 \wedge w_3 \wedge R + c_{\phi} s_{\psi} w_2 \wedge v_3 \wedge R + s_{\phi} c_{\psi} v_2 \wedge w_3 \wedge R + s_{\phi} s_{\psi} v_2 \wedge v_3 \wedge R.
\end{align*}
From~\eqref{eq:SI-temp2} and the above, we get
\begin{equation} \label{eq:SI-temp3}
\Theta_{2k}(P) = (\iota_{e_1} \alpha) (c_{\phi} c_{\psi} w_2 \wedge w_3 \wedge R + c_{\phi} s_{\psi} w_2 \wedge v_3 \wedge R + s_{\phi} c_{\psi} v_2 \wedge w_3 \wedge R + s_{\phi} s_{\psi} v_2 \wedge v_3 \wedge R).
\end{equation}
Since $\iota_{e_1} \Theta_{2k}$ is of $I_1$-type $(2k-1,0) + (0, 2k-1)$, the first term in~\eqref{eq:SI-temp3} must vanish because it contains the $I_1$-complex line $w_2 \wedge w_3$. Moreover, from~\eqref{eq:SI-temp}, we have
\begin{align*}
\iota_{e_1} \Theta_{2k} & = \iota_{e_1} \mathrm{Re} \left( \frac{1}{k!} \sigma^k \right) = \mathrm{Re} \left( (\iota_{e_1} \sigma) \wedge \frac{1}{(k-1)!} \sigma^{k-1} \right) \\
& = I_2 e_1 \wedge \mathrm{Re} \left( \frac{1}{(k-1)!} \sigma^{k-1} \right) - I_3 e_1 \wedge \mathrm{Im} \left( \frac{1}{(k-1)!} \sigma^{k-1} \right).
\end{align*}
Using the orthogonality of $W$ and $\widetilde V$ and the above, the fourth term in~\eqref{eq:SI-temp3} must also vanish, and we are left with
\begin{align*}
\Theta_{2k}(P) & = c_{\phi} s_{\psi} \, g(I_2 e_1, w_2) \, \mathrm{Re} \left( \frac{1}{(k-1)!} \sigma^{k-1} \right) (v_3 \wedge R) \\
& \qquad {} + s_{\phi} c_{\psi} \, g(I_3 e_1, w_3) \, \mathrm{Im} \left( \frac{1}{(k-1)!} \sigma^{k-1} \right) (v_2 \wedge R).
\end{align*}
Applying the induction hypothesis and Cauchy-Schwarz, we deduce that
$$ \Theta_{2k}(P) \leq c_{\phi} s_{\psi} + s_{\phi} c_{\psi} = \sin(\phi + \psi) \leq 1, $$
so $\Theta_{2k}$ has comass at most one. But letting $v_3 \wedge \cdots \wedge v_{2k}$ be a calibrated $(2k-2)$-plane for $\mathrm{Re} (\frac{1}{(k-1)!} \sigma^{k-1})$ and choosing
\begin{align*}
u_2 & = I_2 e_1 \in W, & & \text{so that $\cos(\phi) = 1$, $\sin(\phi) = 0$, and $g(I_2 e_1, w_2) = 1$}, \\
u_3 & = v_3 \in \widetilde V, & & \text{so that $\cos(\psi) = 0$, $\sin(\psi) = 1$},
\end{align*}
gives $\Theta_{2k}(P) = 1$. Thus the comass of $\Theta_{2k}$ is exactly one.
\end{proof}

\begin{rmk} \label{rmk:SI-BH}
The case $k=2$ of Theorem~\ref{thm:special-isotropic-comass-one} is proved in Bryant--Harvey~\cite[Theorem 2.38]{Bryant-Harvey}, where they also prove that a $\Theta_{4}$-calibrated $4$-plane is contained in a quaternionic $2$-plane in $V$. It is likely that this fact remains true for general $k$. That is, a $\Theta_{2k}$-calibrated $2k$-plane in $V$ is contained in a quaternionic $k$-plane. However, we do not have need for this fact.
\end{rmk}

\subsection{Riemannian Cones and Homogeneous Forms} \label{appendix:cones}

Let $(M, g_M)$ be a Riemannian manifold. Let $C = \mathrm{C}(M) = (0, \infty) \times M$, and let $r$ denote the standard coordinate on $(0, \infty)$. The \emph{cone metric} $g_C$ on $C$ induced by $g_M$ is defined to be
\begin{equation} \label{eq:cone-metric}
g_C = dr^2 + r^2 g_M.
\end{equation}
The codimension one submanifold $\{ 1 \} \times M \cong M$ is called the \emph{link} of the cone. We have a projection map $\pi \colon C \to M$ given by $\pi(r, x) = x$. Given differential forms on the link $M$, we can regard them as forms on the cone $C$ by pulling back by $\pi \colon C \to M$. We omit the explicit pullback notation.

\begin{defn} \label{defn:dil}
Consider the vector field
\begin{equation} \label{eq:dilation-vector-field}
\dil = r \ddr
\end{equation}
on the cone $C$. The flow $F_s$ of $\dil$ is given by $(r,p) \mapsto (e^s r, p)$. For this reason, $\dil$ is called the \emph{dilation vector field} on the cone.
\end{defn}
It follows that $\mathscr{L}_{\dil} g_C = 2 g_C$. We say that $g_C$ is \emph{homogeneous of degree $2$} under dilations.

\begin{defn} \label{defn:homogeneous-form}
Let $\gamma \in \Omega^k (C)$. We say that $\gamma$ is \emph{conical} if $\gamma$ is \emph{homogeneous of degree $k$}, or equivalently if $\mathscr{L}_{\dil} \gamma = k \gamma$.
\end{defn}

\begin{prop} \label{prop:homogeneous-form}
Let $\gamma \in \Omega^k (C)$ be a \emph{closed} form which is homogeneous of degree $k$. Then in fact
$$ \gamma = dr \wedge (r^{k-1} \alpha_0) + \frac{r^k}{k} \hat d \alpha_0 = d \Big( \frac{r^k}{k} \alpha_0 \Big), $$
where $\alpha_0 = (\iota_{\dil} \gamma)|_M \in \Omega^{k-1} (M)$.
\end{prop}
\begin{proof}
Write $\gamma = dr \wedge \alpha + \beta$ for some $(k-1)$-form $\alpha$ and $k$-form $\beta$ on $C$ such that $\iota_{\ddr} \alpha = \iota_{\ddr} \beta = 0$. That is, $\alpha$ and $\beta$ have no $dr$ factor, so they can be considered as forms on $M$ depending on a parameter $r$, pulled back to $C$ by $\pi$.

From $\gamma = dr \wedge \alpha + \beta$, and denoting by $\hat d$ the exterior derivative on $M$, we have
$$ 0 = d \gamma = - dr \wedge \hat d \alpha + dr \wedge \beta' + \hat d \beta, $$
and thus
\begin{equation} \label{eq:induced-link-temp0}
\beta' = \hat d \alpha \quad \text{and} \quad \hat d \beta = 0.
\end{equation}
But from $\mathscr{L}_{\dil} \gamma = k \gamma$, since $d \gamma = 0$, we have $k \gamma = d ( \iota_{\dil} \gamma)$. Hence, since $\iota_{\dil} \gamma = r \alpha$, we obtain
$$ k( dr \wedge \alpha + \beta) = k \gamma = d (r \alpha) = dr \wedge \alpha + r dr \wedge \alpha' + r \hat d \alpha. $$
Comparing the two sides above gives
\begin{equation} \label{eq:induced-link-temp1}
k \alpha = \alpha + r \alpha' \quad \text{and} \quad k \beta = r \hat d \alpha.
\end{equation}
The first equation in~\eqref{eq:induced-link-temp1} gives $r \alpha' = (k-1) \alpha$, so $\alpha = r^{k-1} \alpha_0$ where $\alpha_0$ is independent of $r$. Then the second equation gives $k \beta = r \hat d (r^{k-1} \alpha_0) = r^k \hat d \alpha_0$, so $\beta = \frac{r^k}{k} \hat d \alpha_0$. Note that the two equations in~\eqref{eq:induced-link-temp0} are now automatically satisfied. Since $\iota_{\dil} \gamma = r \alpha = r^k \alpha_0$, we therefore conclude that
$$ \gamma = dr \wedge (r^{k-1} \alpha_0) + \frac{r^k}{k} \hat d \alpha_0 = d \Big( \frac{r^k}{k} \alpha_0 \Big), $$
where $\alpha_0 = (r^k \alpha_0)|_M = (\iota_{\dil} \gamma)|_M$.
\end{proof}

\end{appendix}

\bibliographystyle{plain}
\bibliography{CalGeo-HKcones.bib}

\Addresses

\end{document}